\documentclass[11pt]{article}

\usepackage[letterpaper,margin=1in]{geometry}
\usepackage[T1]{fontenc}
\usepackage[utf8]{inputenc}
\usepackage{times}
\usepackage{microtype}
\usepackage{amsmath,amssymb,amsfonts,amsthm,mathtools}
\usepackage{bm}
\usepackage{graphicx}
\usepackage[font=small,skip=1pt]{subcaption}
\usepackage{array}
\usepackage{booktabs}
\usepackage{multirow}
\usepackage{xcolor}
\usepackage{algorithm}
\usepackage[noend]{algpseudocode}
\usepackage{tikz}
\usetikzlibrary{arrows.meta,positioning,calc,shapes.geometric,fit,backgrounds}
\usepackage{enumitem}
\usepackage[round,sort&compress]{natbib}
\usepackage{authblk}
\usepackage{fancyhdr}
\usepackage[colorlinks=true,linkcolor=blue!55!black,citecolor=green!45!black,urlcolor=blue!55!black]{hyperref}
\usepackage{url}

\theoremstyle{plain}
\newtheorem{theorem}{Theorem}[section]
\newtheorem{lemma}[theorem]{Lemma}
\newtheorem{proposition}[theorem]{Proposition}
\newtheorem{corollary}[theorem]{Corollary}
\theoremstyle{definition}

\newtheorem{assumption}{Assumption}

\theoremstyle{remark}
\newtheorem{remark}[theorem]{Remark}

\newcommand{\R}{\mathbb{R}}
\newcommand{\E}{\mathbb{E}}
\newcommand{\PP}{\mathbb{P}}
\newcommand{\Var}{\operatorname{Var}}
\newcommand{\Cov}{\operatorname{Cov}}
\newcommand{\simplex}{\Delta_n}
\newcommand{\bias}{b}
\newcommand{\eps}{\varepsilon}
\newcommand{\norm}[1]{\lVert #1 \rVert}
\newcommand{\inner}[2]{\langle #1 ,\, #2 \rangle}
\newcommand{\wt}[1]{\widetilde{#1}}
\newcommand{\wh}[1]{\widehat{#1}}
\newcommand{\one}{\mathbf{1}}
\newcommand{\CROWD}{\textsc{Crowd}}
\newcommand{\proj}{\Pi_{\mathcal{X}}}
\newcommand{\Xset}{\mathcal{X}}
\newcommand{\HM}{\mathrm{HM}}
\newcommand{\AM}{\mathrm{AM}}
\DeclareMathOperator*{\argmin}{arg\,min}

\DeclareMathOperator{\diag}{diag}
\DeclareMathOperator{\tr}{tr}

\title{\bfseries Asking the Crowd the Right Question:\\[2pt]
Bias-Cancelling Weights for Federated Learning}

\author[1,2]{Ilya Kuruzov\thanks{Corresponding author: \href{mailto:kuruzov.ia@phystech.edu}{kuruzov.ia@phystech.edu}}}
\author[1]{Dmitrii Vishovan}
\author[1]{Kirill Novoselov}
\author[3]{Yuriy Dorn}
\author[4]{Darina Dvinskikh}
\author[1,5,6]{Alexander Gasnikov}
\affil[1]{Innopolis University}
\affil[2]{Moscow Independent Research Institute of Artificial Intelligence}
\affil[3]{Lomonosov Moscow State University}
\affil[4]{Higher School of Economics}
\affil[5]{Trusted AI Research Center, RAS}
\affil[6]{Steklov Mathematical Institute, RAS}
\date{}

\begin{document}
\maketitle
\thispagestyle{empty}

\begin{abstract}
A federated objective is a weighted sum of client risks, and the weights are almost
always fixed in advance. We treat them instead as the only instrument of a
wisdom-of-crowds mechanism: clients are noisy views of one truth, each seeing it through
an independent distortion that is unbiased across the crowd. That the optimal weights
are inversely proportional to the clients' error energies is classical; we begin at the
question that answer presupposes, which energies belong there and whether a crowd can
recover them from itself. Excess risk on the truth is of the exact order of the
aggregate bias energy, so no optimizer can repair a bad weight vector; the truth itself
is identifiable only up to a linear tilt, so subtracting estimated client biases
provably reproduces uniform weighting. The expected per-client second moments, however,
are exactly identified from the law of the crowd's disagreement by a well-conditioned
linear inversion, a step random-effects meta-analysis cannot take because a source
reports once; their realized counterparts are estimable up to an incoherence floor the
algorithm can measure. This yields \CROWD, which reads the disagreement off the
optimization trajectory at no extra cost and matches a Bayesian minimax lower bound in
the same constant: per instance as the horizon grows, and unconditionally as the prior
becomes diffuse. For arbitrary distortions it stays competitive with the optimal
weights, at a ratio governed by a geometric incoherence the algorithm can measure. On
real scans split into sites with their own miscalibrated detectors it attains the
oracle excess risk; on a companion federation that pulls bias and noise apart,
weighting by noise variance is worse than not weighting at all, and \CROWD{} is not.
\end{abstract}

\section{Introduction}
\label{sec:intro}

Every federated system optimises a weighted objective $F_w=\sum_{i=1}^nw_if_i$
over $w\in\simplex$, and the weights are chosen once and never revisited.
\textsc{FedAvg} sets $w_i\propto n_i$ \citep{mcmahan2017}; agnostic and fair variants
replace $w$ by a worst-case or tilted simplex point
\citep{mohri2019,li2020fairness,li2021tilted}; client sampling chooses whom
to query but leaves the target fixed \citep{cho2022,chen2020optimal}. The weights
encode a preference, and $F_w$ is treated as ground truth.

This paper takes the opposite view. Suppose instead that the clients are $n$ noisy
views of one truth: $n$ hospitals measuring the same biology through different
instruments, $n$ annotators labelling one concept through different habits. Then there
is a single correct objective $f_0$, no client's risk equals it, and every weight
vector produces a different wrong objective. This is the wisdom of crowds
\citep{galton1907,surowiecki2004}: the crowd beats its members when the errors are
individual, independent and not systematically aligned. \emph{If the weights are the
only instrument, what can the crowd's own disagreement tell us about how to set them,
and what survives when the crowd does not behave as assumed?}

We posit $f_i=f_0+h_i$ with $h_i$
independent across clients and $\E\nabla h_i(x)=0$. Writing $H_w=\sum_iw_ih_i$, the
weighted objective is $F_w=f_0+H_w$, and everything is governed by the single vector
\begin{equation}
\label{eq:intro-bias}
\bias(w):=\nabla H_w(x_0^\star)=\sum_{i=1}^nw_i\beta_i,
\qquad \beta_i:=\nabla h_i(x_0^\star),\qquad x_0^\star:=\argmin f_0,
\end{equation}
through its \emph{bias energy} $\norm{\bias(w)}^2$. Section~\ref{sec:setup} sandwiches
the excess risk on the truth between $c\norm{\bias(w)}^2/L$ and
$C\norm{\bias(w)}^2/\mu$. Hence no optimiser can repair a bad $w$, and no $w$ is
punished beyond its bias energy. The linear tilt is also the only crowd model
whose members are simultaneously convex and mean-zero.

\paragraph{Contributions.} Write $\zeta_i^2=\E\norm{\beta_i}^2$ and $\AM,\HM$ for the
arithmetic and harmonic means of $\zeta_1^2,\dots,\zeta_n^2$. Weights inversely
proportional to error energies, with the harmonic mean as value, are classical
\citep{hedges1985,dersimonian1986,markowitz1952} and have been applied to federated
aggregation \citep{reyes2021,toelle2025}. We ask which energies belong there when the
target is unobservable, and whether a crowd can recover them from itself.

\begin{enumerate}[label=(\arabic*),itemsep=1pt,topsep=2pt,leftmargin=6mm]
\item \emph{The right energy is the bias energy, and it is identifiable}
(Sections~\ref{sec:setup}--\ref{sec:ident}). The truth is identifiable only up to a
linear tilt; the expected per-client bias energies are identified from the law of the
disagreement by a well-conditioned inversion, the realised ones up to a measurable
incoherence floor.
\item \emph{An estimator at no extra cost, with a matching lower bound}
(Sections~\ref{sec:alg}--\ref{sec:lower}). The residual cancels $\nabla f_0$ at any
iterate, which gives \CROWD{}. Its excess risk matches the bias part
$\Psi_0/(2\mu)$ of a Bayesian minimax lower bound for a Gaussian crowd in the same
constant, to a factor $1+o(1)$. The match holds per instance as the horizon grows,
and unconditionally in the diffuse-prior limit (Theorem~\ref{thm:closure}).
\item \emph{A guarantee beyond the model} (Section~\ref{sec:beyond}). For arbitrary
realised tilts \CROWD{} is $\tfrac{1+\kappa}{1-\kappa}$-competitive with the optimal
weights up to estimation error, $\kappa$ being a geometric incoherence.
\end{enumerate}

Section~\ref{sec:related} reviews related work. Section~\ref{sec:setup} states the
model and the oracle target, Section~\ref{sec:alg} derives the algorithm,
Section~\ref{sec:theory} contains the theory, and Section~\ref{sec:exp} the
experiments. Every numbered statement is proved
in its own subsection of the appendix, which also contains the notation, the
experimental protocols and the further experiments.

\section{Related work}
\label{sec:related}

\paragraph{Heterogeneity as a rate obstacle.} Federated optimisation mostly treats
client heterogeneity as an obstacle to \emph{optimising} a given weighted objective
\citep{kairouz2021}: local-update drift
\citep{mcmahan2017,stich2019,koloskova2020,woodworth2020}, proximal terms
\citep{li2020fedprox}, control variates \citep{karimireddy2020},
objective-inconsistency corrections \citep{wang2020fednova}, server adaptivity
\citep{reddi2021}. There the target is fixed and heterogeneity inflates the rate
through a gradient-dissimilarity constant; under linear tilts that constant instead
shifts the limit, and weighted minibatching converges at the homogeneous rate.

\paragraph{Weights as a preference, and as a quality proxy.} The rules named above
encode a normative goal and treat the resulting objective as ground truth
\citep{mohri2019,li2020fairness,li2021tilted,sagawa2020}. Reading the weights as an
estimator is orthogonal to these: it replaces the target rather than the rule. A second
family weights by a surrogate for update quality: inverse gradient or parameter
variance \citep{reyes2021,toelle2025}, a generalisation-bound surrogate
\citep{wang2023weighting}, validation loss, or consistency with an estimated consensus
\citep{chen2026td}. These rules reach inverse-cost closed forms by the meta-analytic
route, take the cost to be inverted as given -- in practice a noise variance or a proxy
for it -- and treat the consensus as a legitimate target. We argue that the cost to
invert is a bias energy, give an estimator of it with a guarantee, and show that the
consensus itself is not identifiable.

\paragraph{Meta-analysis.} Inverse-variance weighting is the canonical combination of
unbiased estimators of one quantity \citep{hedges1985}. Random-effects meta-analysis
\citep{dersimonian1986} adds a between-study variance: the normal--normal hierarchical
model assigns each study a within-study variance and a single shared between-study
variance, estimated by the method of moments, by \textsc{REML}, or by one of the many
alternatives surveyed in \citet{veroniki2016}, and the standard heterogeneity summaries
\citep{higgins2002} are likewise scalar. This is structural: a study reports once, so
the observed effects cannot support one variance component per study alongside the
common mean. A federated round is instead a repeated draw from the same client, which
is what makes per-client quantities estimable in our setting.

\paragraph{Pooling and forecast combination.} The optimal-weight formula for
correlated sources is the minimum-variance portfolio of \citet{markowitz1952}, and
opinion pooling names the conditions under which averaging works
\citep{genest1986,page2007}. Forecasting reports the opposite difficulty: estimated
optimal weights often lose to the simple average, because estimating them adds variance
\citep{stock2004,smith2009,claeskens2016,demiguel2009}.

\paragraph{Learning whom to learn from.} Collaborative learning weights peers'
gradients to trade bias against variance. \citet{even2022} prove matching upper and
lower bounds on sample complexity for collaborative learning, with the gain governed by
the number of agents close to the target agent, and \citet{chayti2022} and
\citet{grimberg2021} give related bias--variance trade-offs. In all of these the
weights are chosen from pairwise discrepancy bounds that the algorithm is given, and
the authors name the knowledge of those bounds as the main limitation of the approach.
That line also personalises: each agent has its own risk and its own weight vector, and
can evaluate its own risk on held-out data. In the one-truth setting there is a single
objective that nobody observes, and the quantities that play the role of the pairwise
discrepancies have to be estimated from the crowd itself.

\paragraph{Bias-aware combination with an observable target.}
\citet{alamouti2025combining} study multi-agent combination with agent-specific biases
and reach the same harmonic value at inverse-cost weights. Their biases depend on
observable covariates, and their learnability ratio measures the fraction that is
predictable from them, so correcting bias is legitimate in their model. In their
zero-learnability regime their main theorem reduces to oracle inverse-variance
weighting without addressing how the energies are obtained, and their plug-in estimates
each agent's energy by the average squared distance to the current consensus, computed
from the same data that built the consensus. \citet{kiselman2026spcacw} select
aggregation weights by minimising an upper bound on a designated client's convergence
error. The rule trades peer bias against stochastic variance and assigns zero weight to
harmful peers; it is evaluated on \textsc{Mnist}, \textsc{Cifar}-100 and \textsc{Leaf}
Shakespeare. Their target is the designated client's own risk, which that client can
evaluate on held-out data, so peer bias relative to the target is a measurable quantity
and estimating it is a regression problem rather than an identification problem. They
produce one weight vector per client and do not aim at a common objective; the two
approaches are complementary, since a federation could use their rule for
personalisation and a common-objective rule for its shared model. In both works the
target is observable, so neither an identification question nor a matching lower bound
arises. We return to both comparisons after the corresponding results.

\paragraph{Crowdsourcing, differencing, and the optimisation toolkit.} Estimating
annotator reliabilities from agreement patterns is classical
\citep{dawid1979,karger2011,zhang2014spectral}. Those models are discrete-label and
typically identify reliabilities from the spectrum of an agreement matrix. Our setting
is continuous, high-dimensional and coupled to an optimisation loop. The items are the
algorithm's own iterates, chosen adaptively, and the answers are stochastic gradients
whose common part is an unknown nuisance that changes every round. The centring that
removes it is the analogue of fixed-effects differencing in the two-way models of
econometrics \citep{abowd1999}. There the same limited-mobility bias in variance
components is well documented \citep{andrews2008}. Our cross-fitting over disjoint
halves of the rounds is the sample-splitting device of \citet{chernozhukov2018}, used
here to remove the noise-induced inflation exactly. The optimisation components are
standard: projected \textsc{SGD} for smooth strongly convex objectives
\citep{robbins1951,nemirovski2009,polyak1992} and minibatch stochastic approximation
\citep{lan2012,dekel2012}. The lower-bound technique is that of Gaussian location
problems with a diffuse nuisance \citep{lehmann2005,tsybakov2009}, and the
concentration tools are standard \citep{vershynin2018,laurent2000,jin2019subgaussian}.

\paragraph{Applications.} Cross-institution medical federations are the motivating
example \citep{rieke2020,sheller2020,dayan2021}. Site-specific measurement offsets,
coding conventions and label-prevalence shifts are the mechanism that the linear-tilt
model abstracts. Linearised covariate-shift corrections \citep{shimodaira2000} give
another route to the same tilt structure.

\section{Problem statement}
\label{sec:setup}

We work on a closed convex $\Xset\subseteq\R^d$ with Euclidean projection $\proj$,
write $\simplex$ for the simplex on $[n]$ and $\one$ for the all-ones vector.

\begin{assumption}[Crowd]
\label{as:crowd}
$f_i=f_0+h_i$, where $h_1,\dots,h_n$ are independent random $C^2$ functions with
$\E\nabla h_i(x)=0$ for every $x$, drawn once and then fixed. We write
$\beta_i:=\nabla h_i(x_0^\star)$ for client $i$'s realised tilt at the truth's
optimum, $q_i:=\norm{\beta_i}^2$ for its realised bias energy, and
$\zeta_i^2:=\E q_i$ for the expected one. We also write $H_w=\sum_iw_ih_i$ and
$\lambda_w:=\sup_{\Xset}\norm{\nabla^2H_w}_{\mathrm{op}}$, bounded by
$\ell:=\max_i\ell_i$ when $\ell_i$ bounds $\norm{\nabla^2h_i}_{\mathrm{op}}$. The
\emph{linear-tilt} case is $h_i=\inner{\beta_i}{\cdot}$, i.e.\ $\ell=0$.
\end{assumption}

Two classical wisdom-of-crowds hypotheses do the work. Independence says clients err
for unrelated reasons; it fails under a shared preprocessing pipeline, and
Section~\ref{sec:beyond} shows what survives. Mean-zero says the crowd is centred on
the truth rather than on a shared distortion; it is untestable, and assumed whenever
an aggregate is called closer to ground truth than its parts.

\begin{assumption}[Truth]
\label{as:reg}
$f_0$ is $\mu$-strongly convex and $L$-smooth on $\Xset$, $0<\mu\le L$, with
$f_0^\star=\min_\Xset f_0$ attained at $x_0^\star$.
\end{assumption}

\begin{assumption}[Oracle]
\label{as:oracle}
At round $k$ the server broadcasts $x^{k-1}$ and each client returns
$\wh g_i^k=\nabla f_i(x^{k-1})+\eps_i^k$ computed from $b$ fresh samples, with
$\E[\eps^k_i\mid\mathcal F_{k-1}]=0$,
$\E[\norm{\eps^k_i}^2\mid\mathcal F_{k-1}]\le\sigma_i^2/b$ and conditional
independence across $i$. Query points may be adaptive. We write
$\tau_i^2=\sigma_i^2+\tfrac1{n^2}\sum_j\sigma_j^2$, which bounds the noise level of
client $i$'s report after the crowd mean has been subtracted.
\end{assumption}

Assumption~\ref{as:reg} applies to the truth only. The perturbations $h_i$ need not
be convex, and in a non-linear crowd they cannot all be: a crowd of almost
surely convex mean-zero perturbations is almost surely affine
(Proposition~\ref{prop:rigid} below). The way past the linear model is therefore perturbative. Assumption~\ref{as:oracle} bounds a conditional rather
than a uniform second moment and lets the query point depend on everything seen so
far, which matters because our statistic is computed at the algorithm's own iterates.

Formally, the rigidity statement reads as follows.

\begin{proposition}[Rigidity]
\label{prop:rigid}
Let $h_1,\dots,h_n$ satisfy Assumption~\ref{as:crowd} on an open convex set. If each
$h_i$ is almost surely convex, then almost surely every $h_i$ is affine.
\end{proposition}

Since $F_w=f_0+H_w$, choosing $w$ and optimising $F_w$ are the only two things one can
do. The bias-transfer lemma says they combine additively through the aggregate bias
$\Gamma_w:=\norm{\bias(w)}$.

\begin{lemma}[Bias transfer]
\label{lem:transfer}
\label{lem:transfer-full}
Let Assumptions~\ref{as:crowd}--\ref{as:reg} hold, $w\in\simplex$, $\Xset=\R^d$ and
$\lambda_w<\mu$, and write $\Gamma_w:=\norm{\bias(w)}$.
\begin{enumerate}[label=(\roman*),itemsep=1pt,topsep=2pt]
\item $F_w$ is $(\mu-\lambda_w)$-strongly convex and $(L+\lambda_w)$-smooth, and its
minimiser obeys $\norm{x_w^\star-x_0^\star}\le\Gamma_w/(\mu-\lambda_w)$.
\item For every $x$, the excess risk on the truth is at most
$\tfrac{2L}{\mu-\lambda_w}[F_w(x)-\min F_w]+\tfrac{L\Gamma_w^2}{(\mu-\lambda_w)^2}$, and
in the linear-tilt case the bias constant can be made sharp: for every $\epsilon>0$,
\[
f_0(x)-f_0^\star\ \le\ (1+\epsilon)\,\frac{\Gamma_w^2}{2\mu}\ +\ \Big(1+\tfrac1\epsilon\Big)\frac{L}{\mu}\big[F_w(x)-\min F_w\big].
\]
\item If $\lambda_w<\mu/2$ then $f_0(x_w^\star)-f_0^\star\ge
\tfrac{1}{2L}(\tfrac{\mu-2\lambda_w}{\mu-\lambda_w})^2\Gamma_w^2$, which is
$\Gamma_w^2/(2L)$ in the linear case.
\end{enumerate}
\end{lemma}

So any method converging to $\argmin F_w$, even exactly with infinite data, pays
$\Theta(\Gamma_w^2)$ on the truth. The $1+\epsilon$ refinement of (ii) is what lets the
two bounds of Section~\ref{sec:theory} meet in one constant. The non-linear analysis of
Section~\ref{sec:beyond} runs on a bounded domain, where the minimiser of $F_w$ may lie
on the boundary; part (i) and the general bound of (ii) survive there.

\begin{lemma}[Constrained bias transfer]
\label{lem:transfer-c}
Let Assumptions~\ref{as:crowd}--\ref{as:reg} hold on a closed convex $\Xset$ whose
interior contains $x_0^\star$, let $w\in\simplex$ with $\lambda_w<\mu$, and let
$x^\star_w$ minimise $F_w$ over $\Xset$. Then
$\norm{x^\star_w-x_0^\star}\le\Gamma_w/(\mu-\lambda_w)$, and for every $x\in\Xset$,
$f_0(x)-f_0^\star\le\tfrac{2L}{\mu-\lambda_w}[F_w(x)-F_w(x^\star_w)]+\tfrac{L\Gamma_w^2}{(\mu-\lambda_w)^2}$.
\end{lemma}

\subsection{The oracle target: a harmonic law of crowds}
\label{sec:oracle}

Suppose an oracle reveals the bias energies. For weights independent of the tilts,
independence and mean-zero kill the cross terms, and the resulting quadratic is
minimised in closed form.

\begin{proposition}[Exact bias identity, oracle weights, and the gain]
\label{prop:oracle}
Under Assumption~\ref{as:crowd}, for every $w\in\simplex$ independent of the tilts,
$\E\Gamma_w^2=\sum_iw_i^2\zeta_i^2$. For any $c\in\R^n_{>0}$ the minimum of
$\sum_iw_i^2c_i$ over $\simplex$, and also over the affine hull $\one^\top w=1$,
equals $(\sum_ic_i^{-1})^{-1}$ and is attained at $w_i^\star\propto c_i^{-1}>0$.
With $c_i=\zeta_i^2$ this gives $\Psi_0=\HM/n$ against $\AM/n$ for uniform weights.
Moreover \emph{(i)} the ratio is $\AM/\HM\ge1$, with equality iff all $\zeta_i$ are
equal; \emph{(ii)} $\Psi_0\le\min_i\zeta_i^2$; \emph{(iii)} the ratio is unbounded.
\end{proposition}

A correctly
weighted crowd is never worse than its best member and behaves as a crowd of
$n\cdot\AM/\HM$. With $n=21$ and $\zeta_1^2=100$, $\zeta_j^2=1$ otherwise, uniform
weighting gives $\AM/n=120/441=0.272$, the value of a uniform crowd of $3.7$ unit-energy
members: one miscalibrated site among twenty unit-energy ones makes the uniform crowd
behave like a crowd of fewer than four. The simplex constraint is inactive, so signed
aggregation helps only when biases are correlated, which is where the weights cease to be identifiable (Proposition~\ref{prop:corr} in Section~\ref{sec:ident}).
The size-proportional weights used in practice are therefore arbitrarily suboptimal
when site quality and site size are unrelated. Precision-weighted rules
\citep{reyes2021,toelle2025} reach the same closed form by the meta-analytic route, and
so, read as an algebraic rule, does the weighting of \citet{kiselman2026spcacw}; what
the proposition adds is which cost belongs in it, and the rest of the paper asks how to
obtain that cost.

Since the aggregate gradient variance is $\sum_iw_i^2\sigma_i^2/b$, the same
proposition with $c_i=\zeta_i^2+\sigma_i^2/K$ gives the finite-sample target
\begin{equation}
\label{eq:psiK}
\Psi(K):=\Big(\sum_i\big(\zeta_i^2+\sigma_i^2/K\big)^{-1}\Big)^{-1},
\qquad w^\star_i(K)\propto\big(\zeta_i^2+\sigma_i^2/K\big)^{-1},
\end{equation}
where $K$ is the number of samples per client. A client is discounted for bias and
for noise on one scale, at one variance unit per $K$ samples. On a Gaussian crowd
$\Psi(K)/(2\mu)$ is a worst-case lower bound for every algorithm, and \CROWD{}
approaches it per instance as the horizon grows (Section~\ref{sec:lower}).

The lower bound of Section~\ref{sec:lower} splits a harmonic value into a bias part and
a sampling part, and needs the following property.

\begin{lemma}[Harmonic superadditivity]
\label{lem:harmonic}
$H(c):=(\sum_ic_i^{-1})^{-1}$ is concave and positively homogeneous of degree one on
$\R^n_{>0}$; hence $H(a+b)\ge H(a)+H(b)$, with equality if and only if $b$ is
proportional to $a$.
\end{lemma}
The defect can be large. With $n=2$, $a=(1,10)$ and $b=(10,1)$ one has $H(a+b)=5.50$
against $H(a)+H(b)=1.82$, so the two sides differ by a factor three; they coincide
when the bias and noise profiles are proportional.
\section{Algorithm design}
\label{sec:alg}

Choosing $w_i\propto\zeta_i^{-2}$ requires knowing the $\zeta_i^2$; this section asks
what a crowd can learn about itself and turns the answer into an algorithm.

\begin{proposition}[Two impossibilities]
\label{prop:ident-neg}
Under Assumptions~\ref{as:crowd}--\ref{as:oracle}: (i) for every $c\in\R^d$, replacing
$f_0$ by $f_0+\inner{c}{\cdot}$ and each $h_i$ by $h_i-\inner{c}{\cdot}$ leaves every
client risk, and hence the law of the transcript of any protocol, unchanged, while for
$c\ne0$ and $\Xset=\R^d$ the target moves; (ii) subtracting from each client's
gradients its deviation from the crowd mean makes the aggregate an unbiased gradient of
$F_{w_{\mathrm{unif}}}$ for every $w\in\simplex$, with expected bias energy $\AM/n$.
\end{proposition}

No amount of data identifies the truth, so a claim that an aggregate is closer to
ground truth is a claim about the crowd. By (ii) correcting clients does not help --
the disagreement reveals deviations from the consensus, not the consensus error --
but weighting does, because unequal $\zeta_i$ make the dispersion informative.

Both parts separate the present setting from neighbouring ones. Truth-discovery
schemes \citep{chen2026td} score each client by its consistency with an estimated
``ideal'' aggregate, which by (i) is not an identifiable object. In the model of
\citet{alamouti2025combining} biases are partly predictable from observable covariates,
so correcting them is legitimate; by (ii) it is futile here, where no such side
information exists. And no analogue of (i) arises when the target is one client's own
risk, as in \citet{kiselman2026spcacw}: there peer bias can be estimated by comparison
with the designated client. Here no reference client exists, which is why the
statistic below is a maximal invariant of the shift group and why an inversion is
needed.

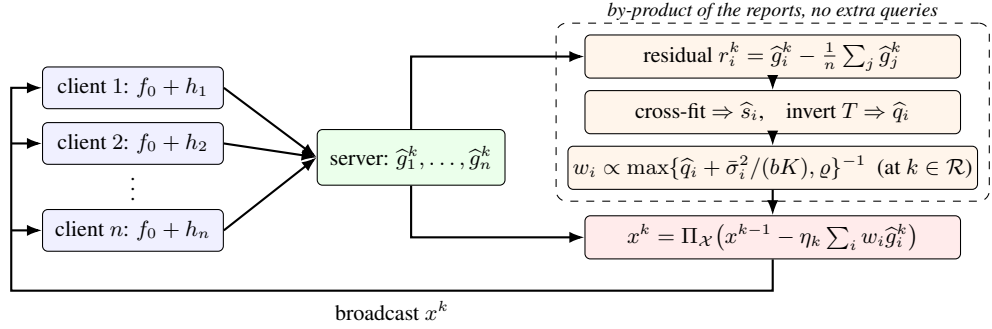
\begin{figure}[t]
\centering
\begin{tikzpicture}[font=\footnotesize,scale=0.92,every node/.style={transform shape},
  client/.style={draw,rounded corners=2pt,minimum width=26mm,minimum height=6mm,fill=blue!6,inner sep=2pt},
  server/.style={draw,rounded corners=2pt,minimum width=27mm,minimum height=8mm,fill=green!8,inner sep=2pt},
  box/.style={draw,rounded corners=2pt,minimum width=54mm,minimum height=6.2mm,fill=orange!8,inner sep=2pt},
  ar/.style={-{Latex[length=2mm]},thick}]
\node[client] (c1) at (0,1.6) {client $1$: $f_0+h_1$};
\node[client] (c2) at (0,0.8) {client $2$: $f_0+h_2$};
\node at (0,0.25) {$\vdots$};
\node[client] (cn) at (0,-0.45) {client $n$: $f_0+h_n$};
\node[server] (agg) at (4.0,0.6) {server: $\wh g^k_1,\dots,\wh g^k_n$};
\node[box] (res) at (9.2,2.05) {residual $r_i^k=\wh g_i^k-\frac1n\sum_j \wh g_j^k$};
\node[box] (est) at (9.2,1.25) {cross-fit $\Rightarrow \wh s_i$,\quad invert $T$ $\Rightarrow \wh q_i$};
\node[box] (wgt) at (9.2,0.45) {$w_i\propto\max\{\wh q_i+\bar\sigma_i^2/(bK),\varrho\}^{-1}$ \ (at $k\in\mathcal R$)};
\node[box,fill=red!8] (step) at (9.2,-0.55) {$x^{k}=\proj\big(x^{k-1}-\eta_k\sum_i w_i\wh g_i^k\big)$};
\draw[ar] (c1.east) -- (agg.west);
\draw[ar] (c2.east) -- (agg.west);
\draw[ar] (cn.east) -- (agg.west);
\draw[ar] (agg.north) |- (res.west);
\draw[ar] (res) -- (est);
\draw[ar] (est) -- (wgt);
\draw[ar] (wgt) -- (step);
\draw[ar] (agg.south) |- (step.west);
\draw[ar] (step.south) -- ++(0,-0.45) -| node[pos=0.25,below] {broadcast $x^k$} (-1.75,0) -- (-1.75,1.6) -- (c1.west);
\draw[ar] (-1.75,0.8) -- (c2.west);
\draw[ar] (-1.75,-0.45) -- (cn.west);
\begin{scope}[on background layer]
\node[draw,dashed,rounded corners,fit=(res)(est)(wgt),inner sep=1.5mm,
      label={[font=\scriptsize\itshape,inner sep=1pt]above:by-product of the reports, no extra queries}] {};
\end{scope}
\end{tikzpicture}
\caption{One round of \CROWD. The unknown $\nabla f_0(x^{k-1})$ cancels in the
residual, so the disagreement is measured at any query point at no extra cost.}
\label{fig:pipeline}
\end{figure}

\paragraph{The residual statistic.} Part (i) also tells us which statistic to use. Define the residual
\begin{equation}
\label{eq:resid}
r_i^k:=\wh g_i^k-\frac1n\sum_j\wh g_j^k
\qquad\text{and write }\ \wt\beta_i:=\beta_i-\tfrac1n\textstyle\sum_j\beta_j
\end{equation}
for the corresponding disagreement in the tilts. The residual cancels $\nabla
f_0(x^{k-1})$ identically, at any query point and for any $f_0$. It is a maximal
invariant of the per-round shift group (Lemma~\ref{lem:residual} below), so by
Proposition~\ref{prop:ident-neg}(i) it loses nothing a shift-invariant
procedure could have used. For linear tilts $r_i^k=\wt\beta_i+\wt\eps^k_i$
exactly.

\begin{lemma}[Residual representation; maximal invariant]
\label{lem:residual}
Under Assumptions~\ref{as:crowd} and \ref{as:oracle}, the residuals
$r_i^k=\wh g_i^k-\frac1n\sum_j\wh g_j^k=\wt{\nabla h_i}(x^{k-1})+\wt\eps_i^k$, with
$\wt{\nabla h_i}:=\nabla h_i-\tfrac1n\sum_j\nabla h_j$ and
$\wt\eps^k_i:=\eps^k_i-\tfrac1n\sum_j\eps^k_j$, do not depend on
$\nabla f_0(x^{k-1})$, whatever the query point and whatever $f_0$ is; the array
$(r^k_i)$ is a maximal invariant of $(\wh g^k_i)$ under per-round common shifts; and
$\E[\wt\eps^k_i\mid\mathcal F_{k-1}]=0$,
$\E[\norm{\wt\eps^k_i}^2\mid\mathcal F_{k-1}]\le\tau_i^2/b$.
\end{lemma}

Two obstacles separate this signal from the weights. First, noise inflates it: the
plug-in estimate $\norm{\wh\beta_i}^2$ over $K_0$ rounds carries an offset of at most
$\tau_i^2/(bK_0)$, and removing it would need the unknown $\tau_i$. We instead split the
probing rounds into two interleaved halves by parity, average within each, and take the
inner product of the two averages:
\begin{equation}
\label{eq:est}
\wh\beta_i^{(a)}=\frac{1}{|\mathcal K_a|}\sum_{k\in\mathcal K_a}r^k_i,\qquad
\wh s_i=\big\langle\wh\beta^{(1)}_i,\wh\beta^{(2)}_i\big\rangle .
\end{equation}
The halves use disjoint rounds with uncorrelated noise, so the inner product is
exactly unbiased for $\norm{\wt\beta_i}^2$ with no knowledge of the noise level
\citep{chernozhukov2018}. The factor-two variance cost is acceptable: an offset
would distort most the energies of the low-bias clients, which the weights should
favour. Interleaving keeps both halves on the same stretch of trajectory.

Second, subtracting the crowd mean removes part of each client's own tilt and mixes in
part of everyone else's, so $\wh s_i$ estimates $(Tq)_i$ rather than $q_i$, for the map
$(Tq)_i=(1-\tfrac2n)q_i+\tfrac1{n^2}\sum_jq_j$. The inversion of $T$ is therefore part
of the estimator, and Section~\ref{sec:ident} shows it is well conditioned.
Algorithm~\ref{alg:weights} inverts $T$ and weights by $\wh q_i+\bar\sigma_i^2/(bK)$
as in~\eqref{eq:psiK}, $K$ being the remaining horizon if known and the rounds seen so
far otherwise. The floor $\varrho>0$ is not cosmetic. The clipped inversion can
return $\wh q_i=0$ exactly, and on an admissible noiseless instance
($\bar\sigma_i=0$, possible with $\kappa<1$) the unfloored cost would
vanish and the normalisation would divide by zero. Flooring costs at most $\varrho$
in the $\ell_\infty$ error that the weight-stability bound of Section~\ref{sec:est} charges, so any $\varrho$
far below the energy scale (e.g.\ $10^{-12}$ times the largest unfloored cost) is
safe.

\begin{figure}[t]
\footnotesize
\begin{minipage}[t]{0.535\linewidth}
\begin{algorithm}[H]
\caption{\CROWD{}}
\label{alg:crowd}
\begin{algorithmic}[1]
\Require $x^0$; stepsizes $\eta_k$; batch $b$; noise bounds $\bar\sigma_i\ge\sigma_i$;
refresh rounds $\mathcal R\subseteq\{2,3,\dots\}$
\State $w^1\gets\one/n$;\ \ $S^{(1)},S^{(2)}\gets0$;\ \ $m_1,m_2\gets0$
\For{$k=1,2,\dots$}
  \Statex \hspace{\algorithmicindent}\textit{Clients, in parallel}
  \State $\wh g^k_i\gets\nabla f_i(x^{k-1})+\eps^k_i$ \Comment{$b$ fresh samples}
  \State $x^k\gets\proj\big(x^{k-1}-\eta_k\sum_iw^k_i\wh g^k_i\big)$
  \Statex \hspace{\algorithmicindent}\textit{Server: statistics from the same reports}
  \State $r^k_i\gets\wh g^k_i-\frac1n\sum_j\wh g^k_j$ \Comment{$\nabla f_0$ cancels}
  \State $a\gets2-(k\bmod2)$ \Comment{odd or even half}
  \State $S^{(a)}\gets S^{(a)}+r^k$;\ \ $m_a\gets m_a+1$
  \If{$k\in\mathcal R$}
    \State $w^{k+1}\gets\Call{Weights}{S^{(1)}/m_1,\,S^{(2)}/m_2}$
  \Else\ $w^{k+1}\gets w^k$
  \EndIf
\EndFor

\end{algorithmic}
\end{algorithm}
\end{minipage}\hfill
\begin{minipage}[t]{0.445\linewidth}
\begin{algorithm}[H]
\caption{\textsc{Weights}$(\wh\beta^{(1)},\wh\beta^{(2)})$}
\label{alg:weights}
\begin{algorithmic}[1]
\Require half-averages $\wh\beta^{(1)},\wh\beta^{(2)}$; rounds $K$ the weights will serve; $b$ and $\bar\sigma_i$ as in Algorithm~\ref{alg:crowd}; floor $\varrho>0$
\State $\wh s_i\gets\big\langle\wh\beta^{(1)}_i,\wh\beta^{(2)}_i\big\rangle$ \Comment{by~\eqref{eq:est}}
\State $\wh q_i\gets\big[\tfrac{n}{n-2}\big(\wh s_i-\tfrac{\sum_j\wh s_j}{n(n-1)}\big)\big]_+$
\Statex \Comment{invert $T$, clip at zero}
\State $\wh c_i\gets\max\big\{\wh q_i+\bar\sigma_i^2/(bK),\,\varrho\big\}$ \Comment{by~\eqref{eq:psiK}; floor}
\State \Return $w_i=\wh c_i^{-1}/\sum_j\wh c_j^{-1}$
\end{algorithmic}
\end{algorithm}
\end{minipage}
\end{figure}

Algorithm~\ref{alg:crowd} (Figure~\ref{fig:pipeline}) costs $O(nd)$ per round beyond
the optimisation step and calls Algorithm~\ref{alg:weights} only at refresh rounds;
since $w^{k+1}$ depends on rounds $\le k$, the anytime analysis of Section~\ref{app:anytime} applies.
\section{Main theoretical results}
\label{sec:theory}

Section~\ref{sec:ident} shows what the disagreement determines, and
Section~\ref{sec:est} turns that identity into an estimator and the end-to-end excess
risk of \CROWD; Section~\ref{app:anytime} extends the bound to the refreshing schedule
used in practice. Section~\ref{sec:lower} matches the upper bound from below, first per
instance and then, in Section~\ref{app:closure}, unconditionally.
Section~\ref{sec:beyond} treats non-linear crowds and then drops the crowd model while
keeping a competitive guarantee. Section~\ref{app:check} asks what the transcript can
be used to check, which gives a guard, and Section~\ref{app:partial} treats partial
participation.

\subsection{What the disagreement determines}
\label{sec:ident}

The tilts themselves are out of reach, by Proposition~\ref{prop:ident-neg}; their
second moments are not. Centring mixes the energies across clients in a way that
depends only on the crowd size, and the mixing can be undone.

\begin{theorem}[Second-moment identifiability]
\label{thm:ident}
Let Assumption~\ref{as:crowd} hold with $h_i$ linear, and $n\ge3$. With $T$ as in
Section~\ref{sec:alg}, the expected disagreement energies $s_i=\E\norm{\wt\beta_i}^2$
satisfy $s=T\zeta^2$, and $T$ is invertible with
\[
(T^{-1}s)_i=\frac{n}{n-2}\Big(s_i-\frac{1}{n(n-1)}\sum_js_j\Big),
\qquad
\norm{T^{-1}}_{\infty\to\infty}=\frac{n^2-2}{(n-1)(n-2)}=:C_n .
\]
Hence $\zeta^2$, and with it the oracle weights, are identified by the law of the
disagreement residuals. Moreover $C_3=3.5$, $C_{20}\le1.17$, $C_n\to1$, and $T$ is
singular for $n=2$.
\end{theorem}

The constant $C_n$ is the conditioning of the inversion: an $\ell_\infty$ error in
$s$ is amplified by at most $C_n$; independence is equally necessary (Proposition~\ref{prop:corr} below). Three claims of different strength must be kept apart.
\emph{Population identifiability} (Theorem~\ref{thm:ident}): the law of the
disagreement determines the expected energies $\zeta_i^2$ exactly. \emph{Realised
estimation} (Section~\ref{sec:est}): the realised $q_i=\norm{\beta_i}^2$ are
estimable only up to an incoherence floor of order $\kappa\bar q$, which no number of
rounds removes. \emph{Exact recovery} of the $q_i$ or of the tilts is neither claimed
nor possible (Proposition~\ref{prop:ident-neg}(i); the floor survives
$K_0\to\infty$). The weights need only the second, crudest claim, through the
$2$-Lipschitz map of Lemma~\ref{thm:weightsub} in Section~\ref{sec:est}.

Theorem~\ref{thm:ident} is the step random-effects meta-analysis cannot take. The
parallel with the moment identity of \citet{dersimonian1986} is exact in form: the
heterogeneity statistic of meta-analysis and our cross-fitted disagreement statistic
are both moment inversions of an \textsc{ANOVA}-type decomposition. Three differences
matter. Our sources are optimisation clients queried repeatedly at moving,
algorithm-chosen points, rather than fixed summary statistics. The effect is
high-dimensional and the relevant functional is a quadratic form. And we need
per-client heterogeneity rather than a pooled variance, which the harmonic law of
Proposition~\ref{prop:oracle} requires; with a common heterogeneity variance it reduces
to inverse-noise weighting, and the common-variance rule of Figure~\ref{fig:baselines}
recovers none of the gain. The per-client energies also play the role of the pairwise
discrepancies that collaborative learning \citep{even2022,chayti2022,grimberg2021}
takes as given: their expectations are identified by the law of the transcript, and the
realised ones are estimable up to the incoherence floor of Section~\ref{sec:est}.

Independence is what makes the cross terms vanish in Proposition~\ref{prop:oracle} and
what makes $T$ the right map to invert. Without it the optimal weights are a
minimum-variance portfolio, and the disagreement no longer determines them.

\begin{proposition}[Correlation destroys identifiability]
\label{prop:corr}
For jointly Gaussian centred crowds with $\Cov(\beta_i,\beta_j)=\Sigma_{ij}I_d/d$,
$\E\norm{\bias(w)}^2=w^\top\Sigma w$ for deterministic $w$, minimised over
$\{\one^\top w=1\}$ by the minimum-variance portfolio $w\propto\Sigma^{-1}\one$
\citep{markowitz1952}, which is also the minimiser over $\simplex$ when
$\Sigma^{-1}\one\ge0$. The map $\Sigma\mapsto J\Sigma J$ has kernel exactly
$\{\one u^\top+u\one^\top:u\in\R^n\}$, and for $\Sigma=I_3$ versus
$\Sigma'=I_3+\one u^\top+u\one^\top$, $u=(0.1,0,0)^\top$ -- both positive definite --
the disagreement residuals have identical laws while the optimal weights differ,
$(\tfrac13,\tfrac13,\tfrac13)$ versus $(0.267,0.367,0.367)$. Under $\Sigma'$ the
uniform weights attain $0.4000$ against the optimum $0.3933$, and under $\Sigma$ the
weights optimal for $\Sigma'$ attain $0.3400$ against $\tfrac13$. Hence no rule that
is a function of the residuals is optimal for both.
\end{proposition}

The kernel is $n$-dimensional and meets the diagonal matrices only at zero when
$n\ge3$, since $\one u^\top+u\one^\top$ is diagonal only if $u_i+u_j=0$ for all
$i\ne j$. Independence restricts $\Sigma$ to the diagonal family, where $J\Sigma J$
is therefore injective; that injectivity is the invertibility of $T$ in
Theorem~\ref{thm:ident}.
\subsection{Estimation error, and why a crude estimate suffices}
\label{sec:est}

Theorem~\ref{thm:ident} concerns expectations, while an algorithm sees one realisation.
For it the map from bias to disagreement energies is only approximately $T$, and one
geometric quantity controls the gap.

\begin{assumption}[Incoherence]
\label{as:incoh}
With $Q$ the Gram matrix of the realised tilts and $D$ its diagonal,
$\kappa:=\norm{D^{-1/2}(Q-D)D^{-1/2}}_{\mathrm{op}}<1$.
\end{assumption}

Assumption~\ref{as:incoh} is a property of one realisation, on which the
assumption-free guarantee of Section~\ref{sec:beyond} also rests. It gives two
deterministic facts:
\begin{equation}
\label{eq:incoh-facts}
\norm{\wt\beta_i}^2=(Tq)_i+\eta_i\ \ \text{with}\ \ |\eta_i|\le3\kappa\bar q,
\qquad\quad
(1-\kappa)\,V(w)\ \le\ \Gamma_w^2\ \le\ (1+\kappa)\,V(w),
\end{equation}
where $V(w):=\sum_iw_i^2q_i$ is the diagonal proxy for the bias energy
(Lemmas~\ref{lem:realised} and~\ref{lem:proxy}). For Gaussian tilts with $d\gtrsim n$,
Lemma~\ref{lem:kappa} gives $\kappa\le C(\sqrt{n/d}+\sqrt{\log(n/\delta)/d})$. The
operative condition, however, compares the crowd size with the number of
\emph{independent causes} of heterogeneity. That number is at most $c^2$ for $c$
classes differing through a label channel, and the number of detectors when the
instruments differ. For a near-linear crowd $\kappa$
is estimable from the two half-averages the algorithm already holds (Section~\ref{app:check}).

The facts in~\eqref{eq:incoh-facts} and the bound on $\kappa$ are the following
lemmas.

\begin{lemma}[Diagonal proxy]
\label{lem:proxy}
Under Assumption~\ref{as:incoh}, $(1-\kappa)V(w)\le\norm{\bias(w)}^2\le(1+\kappa)V(w)$
for every $w\in\R^n$.
\end{lemma}
\begin{lemma}[Realised inversion error]
\label{lem:realised}
For \emph{arbitrary} vectors $\beta_1,\dots,\beta_n$ with $q_i=\norm{\beta_i}^2>0$
and incoherence $\kappa$, and with $\bar q=\tfrac1n\sum_jq_j$,
\[
\norm{\wt\beta_i}^2=(Tq)_i+\eta_i,\qquad|\eta_i|\le3\kappa\bar q\quad\text{for every }i,
\]
so $\norm{T^{-1}(\norm{\wt\beta}^2)-q}_\infty\le3C_n\kappa\bar q$. No distributional
assumption is used.
\end{lemma}
The raw consensus-distance statistic $\norm{\wt\beta_i}^2$ thus estimates $(Tq)_i$, a
version of $q_i$ shrunk towards the crowd mean. This is the plug-in of
\citet{alamouti2025combining} and the scoring step of truth discovery
\citep{chen2026td}; the inversion of $T$ is what corrects it
(Figure~\ref{fig:baselines}).

\begin{lemma}[Incoherence of a Gaussian crowd]
\label{lem:kappa}
Let $\beta_i\sim\mathcal N(0,\tfrac{\zeta_i^2}dI_d)$ independently, $\delta\in(0,1)$,
and put
\[
\epsilon_0:=2\sqrt{\frac{\log(4n/\delta)}{d}}+\frac{2\log(4n/\delta)}{d},
\qquad
\epsilon_1:=C\Big(\sqrt{\frac nd}+\sqrt{\frac{\log(4/\delta)}{d}}\Big),
\]
with $C:=K^2C_V$, where $C_V$ is the absolute constant of
\citet[Thm.~4.6.1]{vershynin2018}, taken to be at least one (enlarging it only
weakens that theorem's bound, and it is needed so that $\epsilon_1/K^2$ dominates the
theorem's $C_V\sqrt{n/d}+t/\sqrt d$), and $K\ge1$ the sub-Gaussian norm of a standard
Gaussian vector. If $\epsilon_0\le\tfrac12$ and $\epsilon_1\le1$, then with probability
at least $1-\delta$, $\max_i|q_i/\zeta_i^2-1|\le\epsilon_0$ and
$\kappa\le(\epsilon_1+\epsilon_0)/(1-\epsilon_0)$. Both conditions hold when
$d\ge C'(n+\log(4n/\delta))$ for an absolute constant $C'$.
\end{lemma}
\begin{lemma}[Gaussian refinement of Lemma~\ref{lem:realised}]
\label{lem:gauss-eta}
If $\beta_i\sim\mathcal N(0,\tfrac{\zeta_i^2}{d}I_d)$ independently and
$d\ge C'(n+\log(12n/\delta))$ with $C'$ as in Lemma~\ref{lem:kappa}, then there is an
absolute constant $C$ such that with probability at least $1-\delta$,
$\max_i|\eta_i|\le C\,\bar q\sqrt{\log(6n/\delta)/d}$.
\end{lemma}
The next lemma carries a \emph{staleness} parameter $\delta_i$: zero for a linear
crowd, $\wt\ell_i\bar D$ for a non-linear one, where the residual is measured at a
moving point.

\begin{lemma}[Estimation error]
\label{thm:est}\label{thm:est-full}
Let Assumption~\ref{as:oracle} hold with arbitrary tilts, $n\ge3$, $K_0$ even, and
suppose the disagreement measured at the query points differs from $\wt\beta_i$ by at
most $\delta_i$ almost surely, for all $k\le K_0$. Put
$\vartheta_i:=\tau_i\sqrt{2/(bK_0)}$ and $\wt\zeta_i:=\norm{\wt\beta_i}$.
\begin{enumerate}[label=(\roman*),itemsep=1pt,topsep=2pt]
\item \emph{(Linear crowd, $\delta_i=0$.)} $\wh s_i$ is exactly unbiased for
$\norm{\wt\beta_i}^2$ given the tilts, for every adaptive query sequence, with
conditional variance at most
$8\wt\zeta_i^2\tau_i^2/(bK_0)+8\tau_i^4/(bK_0)^2$; the bound is dimension-free.
\item \emph{(General.)} $\E|\wh s_i-\norm{\wt\beta_i}^2|\le2\wt\zeta_i(\delta_i+\vartheta_i)+(\delta_i+\vartheta_i)^2$.
\item \emph{(Weights.)} With $\wh q=(T^{-1}\wh s)_+$ and Assumption~\ref{as:incoh},
$\Delta:=\max_i|\wh q_i-q_i|\le C_n(\max_i|\wh s_i-\norm{\wt\beta_i}^2|+3\kappa\bar q)$,
which converts into a high-probability bound by Chebyshev and a union bound, improving
to a $\sqrt{\log(dn/\delta)}$ dependence under sub-Gaussian noise
(Remark~\ref{rem:subgauss}).
\end{enumerate}
\end{lemma}
\begin{remark}[Sub-Gaussian improvement]
\label{rem:subgauss}
Suppose the centred noises are conditionally norm-sub-Gaussian,
$\PP(\norm{\wt\eps^k_i}\ge t\mid\mathcal F_{k-1})\le2e^{-bt^2/(2\tau_i^2)}$. The
concentration inequality for vector martingales with norm-sub-Gaussian increments of
\citet{jin2019subgaussian} then gives, with probability at least $1-\delta$ and for
all $i$ and $a$ simultaneously,
$\norm{\nu^{(a)}_i}\le C\tau_i\sqrt{\log(8dn/\delta)/(bK_0)}$ with an absolute $C$.
On that event the bound of case (ii) with $\delta_i=0$ gives
$|\wh s_i-\norm{\wt\beta_i}^2|\le2\wt\zeta_iC\tau_i\sqrt{\log(8dn/\delta)/(bK_0)}+C^2\tau_i^2\log(8dn/\delta)/(bK_0)$,
so the factors $\sqrt{n/\delta}$ and $n/\delta$ of the moment version become
$\sqrt{\log(dn/\delta)}$ and $\log(dn/\delta)$. We keep the moment version in the main
text because it assumes nothing beyond Assumption~\ref{as:oracle}.
\end{remark}
The three
error sources differ in character: $\vartheta_i$ decays as $(bK_0)^{-1/2}$,
the floor $\kappa\bar q$ decays with $d$ and not with $K_0$, and non-linearity enters
through $\delta_i$ alone. A crude estimate suffices, the map from cost estimates to
attained cost being $2$-Lipschitz in $\ell_\infty$ uniformly in $n$.

\begin{lemma}[Weight stability]
\label{thm:weightsub}
Let $c,\wh c\in\R^n_{>0}$ with $\norm{\wh c-c}_\infty\le\Delta$ and let
$\wh w_i\propto\wh c_i^{-1}$. Then
$\sum_i\wh w_i^2c_i\le\min_{w\in\simplex}\sum_iw_i^2c_i+2\Delta$. In particular, with
$c_i=q_i+\bar\sigma_i^2/(bT)$ and the floored costs
$\wh c_i=\max\{(\wh q_i)_++\bar\sigma_i^2/(bT),\varrho\}$ of Algorithm~\ref{alg:weights},
$V(\wh w)\le\wh\Psi(bT)+2(\Delta+\varrho)$, where
$\wh\Psi(K)=(\sum_i(q_i+\bar\sigma_i^2/K)^{-1})^{-1}$ is the realised counterpart
of~\eqref{eq:psiK}: clipping at zero never hurts, and the floor adds at most $\varrho$
to the $\ell_\infty$ error.
\end{lemma}

Lemma~\ref{thm:weightsub} is deterministic and assumes nothing about how $\wh c$ was
obtained. The price of estimation enters additively and uniformly in $n$, unlike in
the forecast-combination puzzle \citep{stock2004,claeskens2016}. Composing it with
Lemma~\ref{thm:est} and a standard projected-\textsc{SGD} bound gives the end-to-end
bound. We state the \textsc{SGD} bound in a form whose stepsize does not depend on the
unknown minimiser.

\begin{lemma}[Constant-step projected SGD]
\label{lem:sgd}
Let $F$ be $\mu_F$-strongly convex and $L_F$-smooth on a closed convex $\Xset$ with
minimiser $x^\star$ over $\Xset$, and $x_t=\proj(x_{t-1}-\eta g_t)$ with
$\E[g_t\mid\mathcal F_{t-1}]=\nabla F(x_{t-1})$ and
$\E[\norm{g_t-\nabla F(x_{t-1})}^2\mid\mathcal F_{t-1}]\le V$. If $\eta\le1/L_F$,
then $a_t:=\E[\norm{x_t-x^\star}^2\mid\mathcal F_0]$ obeys
$a_t\le(1-\mu_F\eta)a_{t-1}+\eta^2V$, hence $a_T\le e^{-\mu_F\eta T}a_0+\eta V/\mu_F$.
For any $R>0$ let $\ell:=\ln(e+\mu_F^2R^2T/V)$ and
$\eta:=\min\{\tfrac1{2L_F},\tfrac{\ell}{\mu_FT}\}$. Then
\[
a_T\ \le\ a_0e^{-\mu_FT/(2L_F)}+\Big(\frac{a_0}{R^2}+\ell\Big)\frac{V}{\mu_F^2T},
\]
which is at most $a_0e^{-\mu_FT/(2L_F)}+(1+\ell)V/(\mu_F^2T)$ when $a_0\le R^2$.
\end{lemma}
The bound is linear in $a_0$, so it can be averaged over the conditioning: if only
$\E a_0\le R^2$ is known, the same display holds for $\E\norm{x_T-x^\star}^2$ with
$a_0$ replaced by $R^2$.
\begin{theorem}[End-to-end excess risk, linear crowd]
\label{thm:endtoend}
Let Assumptions~\ref{as:crowd}--\ref{as:incoh} hold with $h_i$ linear, $\Xset=\R^d$,
$n\ge3$. Run Algorithm~\ref{alg:crowd} with $\mathcal R=\{K_0\}$, then $T$ rounds of
\textsc{SGD} on $F_{\wh w}$ with the tuned stepsize of Lemma~\ref{lem:sgd}, writing
$V_\sigma:=\sum_i\wh w_i^2\bar\sigma_i^2/b$, $R_0$ an $\mathcal F_{K_0}$-measurable
bound on $\norm{x^{K_0}-x^\star_{\wh w}}$ used in the stepsize (it cannot depend on the
unknown minimiser) and $\ell_T:=1+\ln(e+\mu^2R_0^2T/V_\sigma)$. Then,
conditionally on the first $K_0$ rounds and with $\Delta$ the realised error of
Lemma~\ref{thm:est}(iii) enlarged by the floor $\varrho$ of Algorithm~\ref{alg:weights},
for every $\epsilon>0$,
\[
\E\big[f_0(x^{K_0+T})-f_0^\star\big]\ \le\
\frac{(1+\epsilon^{-1})L^2}{2\mu}\Big(R_0^2e^{-\mu T/(2L)}+\frac{\ell_TV_\sigma}{\mu^2T}\Big)
\ +\ \frac{(1+\epsilon)(1+\kappa)}{2\mu}\big(\wh\Psi(bT)+2\Delta\big) .
\]
\end{theorem}

The first bracket vanishes with the horizon; the second is what the weights cost and
survives $T\to\infty$. A small $\epsilon$ brings the bias constant arbitrarily close
to $1/(2\mu)$ at the price of a larger multiple of the vanishing terms. Any
$\epsilon_T\to0$ with $\epsilon_T^{-1}=o(T/\log T)$ works. A single refresh at $K_0$
freezes the sampling part of $\Delta$; refreshing with all rounds seen so far keeps
shrinking it. Theorem~\ref{thm:anytime} in the next subsection covers the geometric grid
$\mathcal R=\{\lceil\gamma^j\rceil\}$, with $\lceil\log_\gamma T\rceil$ refreshes in $T$
rounds, at a factor $(1-\gamma^{-1})^{-1}$ on the vanishing terms. We use $\gamma=2$.
Refreshing every round behaves the same, and a single early refresh can stall well
above the oracle floor (Figure~\ref{fig:refresh}).

\subsection{The anytime algorithm}
\label{app:anytime}

Theorem~\ref{thm:endtoend} analyses the two-phase schedule, in which the weights are
computed once and then frozen. The experiments run the anytime version, in which the
weights are refreshed and therefore depend on the trajectory they themselves generate.
This subsection analyses the refresh schedule we recommend, a geometric grid.
\begin{theorem}[Anytime \CROWD]
\label{thm:anytime}
Let the assumptions of Theorem~\ref{thm:endtoend} hold, fix $\gamma>1$, and run
Algorithm~\ref{alg:crowd} with the refresh set
$\mathcal R=\{\lceil\gamma^j\rceil:j\ge j_0\}\cap[1,T/\gamma]$ with $j_0\ge1$ and
$T\ge\gamma^{j_0+1}$, so that $\mathcal R$ contains at least one refresh (repeated
values of $\lceil\gamma^j\rceil$ are counted once), using at each refresh
all rounds seen so far and the horizon of the coming segment as $K$, and on each
segment the tuned stepsize of Lemma~\ref{lem:sgd} restarted with that segment's
length and a fixed constant $\bar R>0$ in place of $R$. Let $t_J\le T/\gamma$ be the
last refresh, $\Delta_J$ the weight error there, $T_J=T-t_J\ge(1-\gamma^{-1})T$ the
length of the final segment, $\bar V:=\max_i\bar\sigma_i^2/b$ and
$\bar\ell:=1+\ln(e+\mu^2\bar R^2T/\bar V)$. Then, conditionally on the tilts, for every
$\epsilon>0$,
\begin{multline*}
\E\big[f_0(x^{T})-f_0^\star\big]\ \le\
\frac{(1+\epsilon^{-1})L^2}{2\mu}\Big(\wt R^2\Big(e^{-\mu T_J/(2L)}+\frac{\bar V}{\mu^2\bar R^2T_J}\Big)+\frac{\bar\ell\bar V}{\mu^2T_J}\Big)\\
+\ \frac{(1+\epsilon)(1+\kappa)}{2\mu}\big(\wh\Psi(bT_J)+2\E\Delta_J\big),
\end{multline*}
where $\wt R^2:=\E\norm{x^{t_J}-x^\star_{w^{(J)}}}^2$ is bounded by a constant that
depends on the instance and on $\gamma$ but not on $T$ (Remark~\ref{rem:earlyrefresh}),
and $\Delta_J$ is built from $t_J\ge T/\gamma^2-1$ rounds. The bound is that of
Theorem~\ref{thm:endtoend} with $T$ replaced by $T_J$, so the anytime schedule costs a
factor $(1-\gamma^{-1})^{-1}$ in the vanishing terms and nothing in the bias term.
\end{theorem}

\begin{remark}[A bad early refresh cannot spoil the bound]
\label{rem:earlyrefresh}
Every $w^{(j)}$ lies in the simplex, so $V(w^{(j)})\le\max_iq_i$ and, by
Lemma~\ref{lem:proxy} and Lemma~\ref{lem:transfer}(i), every target lies within
$G:=\sqrt{(1+\kappa)\max_iq_i}/\mu$ of $x_0^\star$, whatever the estimation errors
$\Delta_j$ were. Appendix~\ref{app:pf-anytime} turns this into
\[
\wt R\ \le\ \norm{x^0-x_0^\star}+G+(N_\star+2)\,c_\star,
\qquad
c_\star:=2G+\sqrt{\frac{\bar V}{2L\mu}}+\sup_{t\ge T_\star}\sqrt{\frac{(1+\ln(e+\mu^2\bar R^2t/\bar V))\bar V}{\mu^2t}},
\]
with $T_\star:=\max\{4L\ln8/\mu,\,8\bar V/(\mu^2\bar R^2)\}$ and
$N_\star:=\lceil\log_\gamma(2T_\star/(\gamma-1))\rceil+1$, none of which depends on $T$.
A refresh that is badly wrong moves the target by a bounded amount.
\end{remark}

Two further remarks. First, the result does not require $\Delta_j$ to be monotone in
$j$: only the last refresh enters the bias term, and earlier refreshes affect the
bound solely through $\wt R$. Second, a uniform grid would give $J=\Theta(T)$ segments
and the restart cost of the tuned stepsize would no longer be negligible, which is why
the geometric grid is the schedule we analyse and recommend.

\subsection{A matching lower bound}
\label{sec:lower}

\CROWD{} attains the harmonic value; this section bounds what any algorithm can
attain. By Proposition~\ref{prop:ident-neg} a weighting rule can only use the
disagreement, and the oracle weights exhaust that channel.

\begin{proposition}[The oracle weights exhaust the disagreement]
\label{prop:orth}
With $w^\star_i=\Psi_0/\zeta_i^2$: \emph{(i)}
$\E\inner{\bias(w^\star)}{\wt\beta_j}=0$ for every $j$; \emph{(ii)}
$\min_{v\in\R^n}\E\norm{\bar\beta-\sum_iv_i\wt\beta_i}^2=\Psi_0$, attained by
$\bar\beta-\bias(w^\star)$; \emph{(iii)} in the Gaussian model
$\E[\bar\beta\mid\wt\beta]=\bar\beta-\bias(w^\star)$ and the residual variance is
exactly $\Psi_0$.
\end{proposition}

So $\Psi_0$ is the variance of the consensus error that remains after conditioning on
everything the crowd can observe about itself, and weighting is optimal Bayesian
denoising through the only available channel. The next theorem drops even the
restriction to weighting rules.

\begin{theorem}[Bayesian minimax lower bound]
\label{thm:lower}
Let $\mu,d,n$, the energies $\zeta_i$, the noise levels $\sigma_i$ and a per-client
query budget $K$ all be given, and consider the following instances: the truth is the
quadratic $f_0=\tfrac\mu2\norm{x-x_c}^2$ with unknown centre $x_c$, the clients are
$f_i=f_0+\inner{\beta_i}{\cdot}$ with $\beta_i\sim\mathcal N(0,\tfrac{\zeta_i^2}dI_d)$,
and each query to client $i$ returns its gradient corrupted by
$\mathcal N(0,\tfrac{\sigma_i^2}dI_d)$ noise. Then over all algorithms making at most
$K$ queries per client -- adaptive query points, arbitrary computation on the reports,
full knowledge of $\mu$, the $\zeta_i$ and the $\sigma_i$ --
\[
\inf_A\sup_{x_c\in\R^d}\E\big[f_0(\wh x)-\min f_0\big]\ \ge\ \frac{1}{2\mu}\Psi(K).
\]
An algorithm additionally handed the exact disagreements still pays at least
$(\Psi_0+s^2)/(2\mu)$, where $s^2=(\sum_iK/\sigma_i^{2})^{-1}$; the two floors
coincide when the bias and noise profiles are proportional.
\end{theorem}

Three features matter. Firstly, the hard instance is perfectly conditioned, so the difficulty
is identification rather than optimisation. Secondly, adaptivity does not help, since $\mu$ is
known and a report carries the same information at every query point. Finally, the
upper bound of \CROWD{} has the same form $\Psi(bT)$ up to an estimation term.
The bound is not an oracle-complexity bound: it restricts only the number of gradient
queries per client, not computation or adaptivity, and the obstruction is purely
statistical, a Gaussian location problem with a diffuse nuisance. It has no counterpart
in \citet{kiselman2026spcacw}, whose guarantee is the minimisation of an upper bound:
with an observable target the problem is not statistically obstructed.

\begin{corollary}[Per-instance sharpness against the minimax floor]
\label{cor:optimal}
Run Algorithm~\ref{alg:crowd} as in Theorem~\ref{thm:endtoend} at a fixed centre and
let $T\to\infty$. On every realisation with $2\Delta\le\epsilon_1\wh\Psi(bT)$ and
$\kappa\le\epsilon_2$, the excess risk is at most
$(1+o(1))(1+\epsilon_1)(1+\epsilon_2)\wh\Psi(bT)/(2\mu)$, the
$o(1)$ vanishing more slowly than the horizon term of Theorem~\ref{thm:endtoend}. The
lower bound of Theorem~\ref{thm:lower} for the same per-client budget is
$\Psi(b(K_0+T))/(2\mu)$. Both values tend to their bias parts, $\wh\Psi_0$ and
$\Psi_0$, where $\wh\Psi_0:=(\sum_iq_i^{-1})^{-1}$: at every fixed centre \CROWD{}
attains, up to the factor $(1+\epsilon_1)(1+\epsilon_2)\wh\Psi_0/\Psi_0$, the floor
below which no algorithm can go uniformly over centres.
\end{corollary}

The two bounds share the functional and the constant $1/(2\mu)$, but they are not
the same quantity, and the corollary is not a claim of uniform minimax optimality.
The upper bound is conditional: it fixes the centre, requires its two hypotheses and
uses the realised benchmark $\wh\Psi$. The lower bound is a supremum over centres and
an average over the crowd prior. Three bridges narrow the gap. Over
centres, the transient of Theorem~\ref{thm:endtoend} depends on the centre only
through $R_0$, so the upper bound is uniform over any fixed ball once $T$ is large.
Over the prior, $H$ is concave (Lemma~\ref{lem:harmonic}), so
$\E\wh\Psi_0\le\Psi_0$ by Jensen. Lemma~\ref{lem:kappa} puts $\wh\Psi_0/\Psi_0$
within $\epsilon_0$ of one and $\kappa$ below $\epsilon_2$ once
$d\gtrsim\epsilon_2^{-2}(n+\log(n/\delta))$. Over the hypotheses,
$2\Delta\le\epsilon_1\wh\Psi$ asks for $K_0\gtrsim\epsilon_1^{-2}$ probing rounds.
The next subsection assembles the three into an unconditional statement: the Bayes
risk of a prior-aware variant of \CROWD{} matches the Bayes floor as the prior becomes
diffuse.

\subsection{Closing the loop: unconditional Bayes optimality}
\label{app:closure}

Corollary~\ref{cor:optimal} compares a conditional, fixed-centre upper bound with a
worst-case, prior-averaged lower bound, and says so. This subsection reduces the
asymmetry. The device is to compare \emph{Bayes risks in one and the same
formulation}. The centre is drawn from the Gaussian prior at a finite scale $\rho$.
Both the floor and \CROWD's risk are then unconditional expectations over the centre,
the crowd, the oracle noise and the algorithm's randomness. Neither side is
conditional or pointwise, and both are the same functional of the same experiment.

Throughout, fix $\mu>0$, $n\ge3$, energies $\zeta\in(0,\infty)^n$ with means
$\AM,\HM$, noise levels $\sigma\in[0,\infty)^n$, and a batch size $b$; for each
dimension $d$ consider the Gaussian class of Theorem~\ref{thm:lower},
\[
f_0=\tfrac\mu2\norm{x-x_c}^2,\qquad
\beta_i\sim\mathcal N\big(0,\tfrac{\zeta_i^2}dI_d\big),\qquad
\text{query noise }\mathcal N\big(0,\tfrac{\sigma_i^2}dI_d\big),
\]
now with the centre itself random, $x_c\sim\mathcal N(x^0,\tfrac{\rho^2}dI_d)$,
independent of everything else. For an algorithm $A$ with output $\wh x$ write
$\mathfrak B_\rho(A):=\E[f_0(\wh x)-f_0^\star]$ for its Bayes excess risk, the
expectation running over $x_c$, the tilts, the noises and $A$'s randomness. Step~4 of
the proof of Theorem~\ref{thm:lower} establishes the exact finite-$\rho$ floor
\emph{before} the limit $\rho\to\infty$ is taken. Every algorithm making at
most $K$ queries per client satisfies
\begin{equation}
\label{eq:bayes-floor}
\mathfrak B_\rho(A)\ \ge\ \mathfrak B^\star_\rho(K)
:=\frac{1}{2\mu}\Big(\frac{1}{\mu^2\rho^2}+\frac{1}{\Psi(K)}\Big)^{-1}.
\end{equation}

The algorithm to be compared is the two-phase \CROWD{} of
Theorem~\ref{thm:endtoend}, made uniform in the centre by a projection whose radius
the prior supplies. Let \CROWD$^\star$ denote Algorithm~\ref{alg:crowd} run on the ball
$\Xset=B(x^0,D)$ of radius $D:=2\rho+\sqrt2\,\zeta_{\max}/\mu$,
$\zeta_{\max}:=\max_i\zeta_i$, with the projection $\proj$ applied in \emph{every}
round, probing included. Started at $x^0$, it runs $K_0$ probing rounds ($K_0$ even)
with uniform weights, all $n$ clients reporting in every round. It is then refreshed
once by Algorithm~\ref{alg:weights} with horizon $K=T$, noise bounds
$\bar\sigma_i=\sigma_i$ and floor $\varrho>0$. Finally it runs $T$ rounds of projected
\textsc{SGD} on $F_{\wh w}$ with the tuned stepsize of Lemma~\ref{lem:sgd} and $R=2D$.
Theorem~\ref{thm:lower} grants every algorithm $\mu$, the $\sigma_i$ and $T$; besides
these, \CROWD$^\star$ uses the prior scale $\rho$ and $\zeta_{\max}$ through $D$. It
is therefore a prior-aware variant of Algorithm~\ref{alg:crowd}, not
Algorithm~\ref{alg:crowd} itself.

\begin{theorem}[The loop closed]
\label{thm:closure}
Consider \CROWD$^\star$ as just defined, on the Gaussian class with linear tilts. Fix
$\delta\in(0,1)$ and let $\epsilon_0\le\tfrac12$, $\epsilon_1\le1$ be the quantities
of Lemma~\ref{lem:kappa} at confidence $\delta$, with
$\epsilon_2:=(\epsilon_1+\epsilon_0)/(1-\epsilon_0)<1$. Put
$\bar V:=\max_i\sigma_i^2/b$, $\ell_T:=1+\ln(e+4\mu^2D^2T/\bar V)$ (read
$\ell_T\bar V$ as $0$ when $\bar V=0$),
$s_i:=(1-\tfrac2n)\zeta_i^2+\AM/n$, $\vartheta_i:=\tau_i\sqrt{2/(bK_0)}$ and
\[
a(K_0):=\sum_{i=1}^n\big(2\sqrt{s_i}\,\vartheta_i+\sqrt2\,\vartheta_i^2\big).
\]
Then for every $\epsilon_T>0$,
\begin{align*}
\mathfrak B_\rho(\CROWD^\star)\ \le\
&\frac{(1+\epsilon_T)(1+\epsilon_2)(1+\epsilon_0)}{2\mu}
\Big[\Psi(bT)+2C_na(K_0)+6C_n\epsilon_2\AM+2\varrho\Big]\\
&+(1+\epsilon_T^{-1})\,\frac{\mu}{2}\Big(4D^2e^{-T/2}+\frac{\ell_T\bar V}{\mu^2T}\Big)
\ +\ 2\mu\big(D^2+\rho^2\big)\big(\delta+e^{-d/8}\big).
\end{align*}
\end{theorem}

Every term is an unconditional expectation, and each excess factor is priced. The
crowd-realisation risk is paid by the tail term $2\mu(D^2+\rho^2)(\delta+e^{-d/8})$
rather than assumed away. The incoherence floor appears as $6C_n\epsilon_2\AM$ with
$\epsilon_2$ supplied by Lemma~\ref{lem:kappa}, and the estimation error as the
explicit $a(K_0)$. Uniformity in the centre is bought from the prior itself, through
the projection radius $D$. Matching \eqref{eq:bayes-floor} is then a matter of sending the
parameters to their limits in the right order, the prior scale $\rho$ last.

\begin{corollary}[Unconditional matching in the diffuse-prior limit]
\label{cor:closure}
Run \CROWD$^\star$ with $\varrho=1/T$, $\epsilon_T=T^{-1/4}$ and $\delta=e^{-\sqrt d}$,
and let $K=b(K_0+T)$ be the per-client budget shared by both sides. At every fixed
prior scale $\rho$,
\[
\limsup_{d\to\infty}\ \limsup_{K_0\to\infty}\ \limsup_{T\to\infty}\
\frac{\mathfrak B_\rho(\CROWD^\star)}{\mathfrak B^\star_\rho(b(K_0+T))}\ \le\
1+\frac{\Psi_0}{\mu^2\rho^2},
\]
and the ratio is at least one at every parameter value by~\eqref{eq:bayes-floor}.
Hence the ratio tends to one as the prior becomes diffuse, $\rho\to\infty$: for every
$\epsilon\in(0,1]$ there is $\rho_0(\epsilon)$ and, for each $\rho\ge\rho_0$,
thresholds $d_0,K_0',T_0$ such that $d\ge d_0$, $K_0\ge K_0'$ and $T\ge T_0$ give
$\mathfrak B_\rho(\CROWD^\star)\le(1+\epsilon)\,\mathfrak B^\star_\rho(K)$.
\end{corollary}

Four remarks on scope. First, at a fixed prior scale the two sides do not match. The
gap $\Psi_0/(\mu^2\rho^2)$ is the information in the prior: the Bayes floor uses it,
and \CROWD$^\star$, which does not shrink towards $x^0$, does not. The matching is
therefore a statement about the diffuse-prior limit, the regime of the minimax bound
of Theorem~\ref{thm:lower}. Second, what is matched is the bias part
$\Psi_0/(2\mu)$ and the constant $1/(2\mu)$. At finite $K_0$ the upper side carries
the estimation term $2C_na(K_0)$, of order $K_0^{-1/2}$, whereas $\Psi(bT)-\Psi_0$
is of order $1/T$. The finite-sample exchange rate $\sigma_i^2/K$ inside $\Psi(K)$
is therefore not certified by this comparison. Third, the order of limits is forced. The
horizon must grow at fixed prior scale, since a fixed number of steps cannot serve an
arbitrarily diffuse prior, and $d/n\to\infty$ drives the incoherence floor to zero,
as in Corollary~\ref{cor:optimal}. The finite-sample terms are explicit in the
theorem. Fourth, the statement is about the Gaussian class of
Theorem~\ref{thm:lower}, the class the lower bound is proved for. For a single
deployment the conditional, per-instance reading of Corollary~\ref{cor:optimal}
remains the operative statement.

\subsection{Beyond the linear model}
\label{sec:beyond}

A non-linear crowd necessarily perturbs curvature (Proposition~\ref{prop:rigid}).
By Lemma~\ref{lem:transfer}(i) the cost of that perturbation is governed by
$\lambda_w/\mu$. The residual~\eqref{eq:resid} still cancels $\nabla f_0$ for arbitrary
$h_i$, but it now measures the disagreement at the current iterate. Therefore the
estimator is stale by $\delta_i\le\wt\ell_i\norm{x^{k-1}-x_0^\star}$, which does not
vanish as probing converges to the uniform minimiser. The analysis goes through with
$\mu-\lambda$ in place of $\mu$ and with the crude transfer constant; the sharp one is
available only for a linear crowd.

\begin{theorem}[Non-linear crowd]
\label{thm:nonlin}
Let Assumptions~\ref{as:crowd}--\ref{as:incoh} hold on a closed ball $\Xset$ of
diameter $\bar D$, with $\lambda:=\max_{w\in\simplex}\lambda_w\le\ell<\mu/2$, and write
$L'=L+\lambda$ and $\Gamma_{\max}^2:=(1+\kappa)\max_iq_i$. Assume the interior of
$\Xset$ contains the closed ball of radius $\Gamma_{\max}/(\mu-\lambda)$ about
$x_0^\star$. Run
Algorithm~\ref{alg:crowd} with $\mathcal R=\{K_0\}$ and uniform weights during probing,
then $T$ rounds of projected \textsc{SGD} on $F_{\wh w}$ with the tuned stepsize of
Lemma~\ref{lem:sgd} and $R=\bar D$, and put $V_\sigma:=\sum_i\wh w_i^2\bar\sigma_i^2/b$
and $\ell_T:=1+\ln(e+(\mu-\lambda)^2\bar D^2T/V_\sigma)$.
Then Lemma~\ref{thm:est} applies with $\delta_i=\wt\ell_i\bar D$,
$\wt\ell_i=\ell_i+\tfrac1n\sum_j\ell_j$, and, conditionally on the first $K_0$ rounds,
with $R_0=\bar D$ below,
\[
\E\big[f_0(x^{K_0+T})-f_0^\star\big]\ \le\
\frac{LL'}{\mu-\lambda}\Big(R_0^2e^{-(\mu-\lambda)T/(2L')}+\frac{\ell_TV_\sigma}{(\mu-\lambda)^2T}\Big)
\ +\ \frac{L(1+\kappa)}{(\mu-\lambda)^2}\big(\wh\Psi(bT)+2\Delta\big).
\]
The sharp bias constant of Lemma~\ref{lem:transfer}(ii) is available only for a linear
crowd, so the bias term carries $L/(\mu-\lambda)^2$ here rather than the $1/(2\mu)$ of
Theorem~\ref{thm:endtoend}; at $\lambda=0$ the two differ by the factor $2L/\mu$.
\end{theorem}

To keep the non-linear part of the estimation error below $\epsilon\Psi_0$, a
well-started run needs $\ell/\mu\lesssim\epsilon n^{-1/2}\HM/\AM$; the next remark
derives this condition. The runs of Figure~\ref{fig:misspec} (Appendix~\ref{app:exp})
show no degradation even where it fails.

\begin{remark}[Where the condition on $\ell/\mu$ comes from]
\label{rem:ellmu}
The condition is an order-of-magnitude reading of Theorem~\ref{thm:nonlin}. Its proof
uses $\bar D$ only as a bound on $\norm{x^{k-1}-x_0^\star}$ during probing (Step~1 of
Appendix~\ref{app:pf-nonlin}), so if the probing iterates stay within distance $r$ of
$x_0^\star$, Lemma~\ref{thm:est} applies with $\delta_i=\wt\ell_i r$. By
Lemma~\ref{thm:est}(ii)--(iii) the non-linear contribution to $\Delta$ is then
$C_n(2\max_i\wt\zeta_i\wt\ell_ir+\max_i\wt\ell_i^2r^2)$, and $\wt\ell_i\le2\ell$.
Suppose, as in a well-started run, that probing takes place near the minimiser of the
uniform objective, $r\le c\sqrt{\AM/n}/\mu$ (the radius Lemma~\ref{lem:transfer}(i)
gives for $x^\star_{w_{\mathrm{unif}}}$ when $\Gamma^2_{w_{\mathrm{unif}}}\approx\AM/n$),
and that no single energy dominates, $\max_i\wt\zeta_i^2\le c'\AM$. Then the first term
is at most $4cc'^{1/2}C_n\ell\,\AM/(\mu\sqrt n)$, and asking it to be at most
$\epsilon\Psi_0=\epsilon\HM/n$ gives $\ell/\mu\lesssim\epsilon\,n^{-1/2}\HM/\AM$. The
quadratic term is smaller by a further factor of the same order. Without the two
suppositions the condition changes accordingly; the theorem itself does not need them.
\end{remark}

The final result of this section drops the crowd model entirely.

\begin{proposition}[Assumption-free operating guarantee]
\label{thm:free}
Let $\bar x$ be any point, and let the tilts $\beta_i=\nabla h_i(\bar x)$ there be
arbitrary vectors: no independence, no zero mean, no linearity, no distributional
assumption. Let $\kappa<1$ be their incoherence and let $\wh w_i\propto\wh c_i^{-1}$ be
\CROWD's weights with $\norm{\wh c-q}_\infty\le\Delta$. Then, deterministically,
\begin{equation}
\label{eq:free-1}
\norm{\bias(\wh w)}^2\ \le\ \frac{1+\kappa}{1-\kappa}\min_{w\in\simplex}\norm{\bias(w)}^2+2(1+\kappa)\Delta ,
\end{equation}
and the gain over uniform weighting is at least
$\tfrac{1-\kappa}{1+\kappa}\AM(q)/(\HM(q)+2n\Delta)$, for $\AM,\HM$ the means of the
realised energies.
\end{proposition}

Inequality~\eqref{eq:free-1} makes \CROWD{} $\tfrac{1+\kappa}{1-\kappa}$-competitive
with the exactly optimal weights for the realised instance. This holds even where
those weights exploit correlations that \CROWD{} does not model. The crowd hypotheses
enter only through the benchmark and the size of $\Delta$. As $\kappa\to1$ the ratio
degrades: nearly aligned tilts admit cancellations that a diagonal rule cannot
exploit.

\begin{remark}[What is and is not assumption-free]
Proposition~\ref{thm:free} uses Assumption~\ref{as:oracle} (about the sampling noise)
and Assumption~\ref{as:incoh} (a property of the realised tilts), and nothing else:
in particular no independence, no zero mean, no linearity, and no distribution over
crowds. The crowd hypotheses enter only the \emph{interpretation} of the benchmark:
under them $\min_w\norm{\bias(w)}^2$ concentrates around $\Psi_0$, and
$\Psi_0/(2\mu)$ is the lower bound of Theorem~\ref{thm:lower}.
\end{remark}

\subsection{What the transcript can and cannot be used to check}
\label{app:check}

Centring (Assumption~\ref{as:crowd}) and incoherence (Assumption~\ref{as:incoh})
differ in whether a practitioner can check them from the transcript. By
Proposition~\ref{prop:ident-neg}(i), centring cannot be tested at all. Incoherence is a
property of the realised tilts, and the residuals carry enough information to estimate
it.

\begin{proposition}[Checkability]
\label{prop:check}
Under Assumptions~\ref{as:crowd}--\ref{as:oracle}:
\begin{enumerate}[label=(\roman*),itemsep=1pt,topsep=2pt]
\item \emph{(Centring is not checkable.)} For every $c\in\R^d$ the instance
$(f_0,h_1,\dots,h_n)$ and the instance obtained by replacing $f_0$ with
$f_0+\inner{c}{\cdot}$ and each $h_i$ with $h_i-\inner{c}{\cdot}$ generate the same
joint law of the transcript under every protocol, while their targets and their
optimal weights differ. Hence no test based on the transcript has power against
``the crowd is centred on a distorted objective''.
\item \emph{(Incoherence is estimable up to a rank-two correction.)} Assume in addition
that the tilts are linear, and let $J:=I-\tfrac1n\one\one^\top$. The shift of (i) maps
$\beta_i$ to $\beta_i-c$ and $Q$ to $Q-\one a^\top-a\one^\top+\norm c^2\one\one^\top$
with $a_i=\inner{\beta_i}{c}$, which changes $\kappa$ in general, so $\kappa$ is not a
function of the transcript law; only $JQJ$ is. Let $\wh G$ be the matrix with entries
$\wh G_{ij}=\tfrac12(\inner{\wh\beta^{(1)}_i}{\wh\beta^{(2)}_j}+\inner{\wh\beta^{(2)}_i}{\wh\beta^{(1)}_j})$
formed from the two half-averages of~\eqref{eq:est}, let
$\wh q=(T^{-1}\diag\wh G)_+$, and define
\[
\wh Q_{ij}:=\wh G_{ij}+\frac{\wh q_i+\wh q_j}{n}-\frac{1}{n^2}\sum_k\wh q_k,
\qquad
\wh\kappa:=\big\|\wh D^{-1/2}(\wh Q-\wh D)\wh D^{-1/2}\big\|_{\mathrm{op}},\quad \wh D=\diag\wh q .
\]
Then $\E[\wh G\mid B]=JQJ$ exactly, for every adaptive query sequence. The correction makes every row sum of
$\wh Q-\wh D$ vanish identically, and zero is also the mean, under
Assumption~\ref{as:crowd}, of the off-diagonal row sums
$\varsigma_i:=\sum_{j\ne i}Q_{ij}$ of $Q$. With
$\wh G$ and $\wh q$ replaced by their targets $JQJ$ and $q$, the resulting
$\wh\kappa_\circ$ satisfies, deterministically,
\[
|\wh\kappa_\circ-\kappa|\ \le\ \kappa\Big(2\sqrt{\AM(q)/\HM(q)}+\AM(q)/\HM(q)\Big).
\]
\end{enumerate}
\end{proposition}

The hypothesis that makes the target meaningful cannot be tested. Any claim that a
federated aggregate is closer to ground truth rests on it. The hypothesis that makes
the estimator work can be estimated cheaply: $\wh\kappa$ reuses the two half-averages
the algorithm already holds and costs one $n\times n$ Gram product. The estimate is
not consistent for $\kappa$: its population error is a rank-two term of the same order
as $\kappa$ itself, which is the price of $\kappa$ not being identified. Where the
estimate is reliable is therefore a question for measurement, and
Appendix~\ref{app:regime} calibrates the threshold empirically. Without the linearity
hypothesis the residual measures the disagreement at the iterate rather than at
$x_0^\star$, and in our runs this deflates $\wh\kappa$ (Appendix~\ref{app:regime}).
This suggests running \CROWD{} with a guard.

\begin{remark}[Guarded \CROWD: a screening heuristic, not a certificate]
\label{rem:guard}
Compute $\wh\kappa$ at the refresh rounds. If $\wh\kappa<1$, use \CROWD{} and report
$\tfrac{1+\wh\kappa}{1-\wh\kappa}$ as an estimate of the competitive ratio of
Proposition~\ref{thm:free}. If $\wh\kappa\ge1$, the diagonal proxy is uninformative on
this instance and the bound is vacuous; fall back to uniform weighting. The guard
indicates which regime the deployment is likely in but does not certify it.
Proposition~\ref{prop:check} shows that $\wh\kappa$ is not consistent for $\kappa$:
its population error is of the same order as $\kappa$ itself. The reported ratio is
therefore an estimate, not a bound. In place of a guarantee the guard comes with a
measured error rate. It is cheap, and in the linear regime it gives the deployment an
estimate of the quantity Assumption~\ref{as:incoh} bounds.
Appendix~\ref{app:diag} measures $\wh\kappa$ against $\kappa$ on three mechanisms, and
Appendix~\ref{app:regime} calibrates the threshold: the empirical probability that
$\kappa\ge1$ given $\wh\kappa<1$ is $4.2\%$. The restriction to a near-linear crowd
matters. The estimator inherits the staleness of Lemma~\ref{thm:est}, so when probing
happens far from the optimum $\wh\kappa$ under-reports, by a factor three to five in
Appendix~\ref{app:regime}. That failure is
detectable and partly repairable, and Remark~\ref{rem:staleguard} does both. In the
real-data study of Appendix~\ref{app:gain}, where the probe is warm, $\wh\kappa=0.52$
against a realised $\kappa=0.48$. The guard is also conservative: \CROWD{} often still
improves on uniform weighting when $\kappa$ exceeds one, at $\kappa=1.28$ in
Appendix~\ref{app:region}, $\kappa\approx2$ in Appendix~\ref{app:femnist} and
$\kappa\approx2$ to $3.9$ in Appendix~\ref{app:mlp}.
\end{remark}

\begin{remark}[A staleness-aware guard, and what it costs]
\label{rem:staleguard}
The restriction to a near-linear crowd can be checked on the trajectory the algorithm
already produces. Split the probing window into two disjoint
\emph{time}-halves, compute $\wh\kappa$ on each from its own parity split, and read the
drift between them. A crowd that is effectively linear over the probing region returns
the same value twice; a curved one drifts, because the two halves measure the Gram at
different points. Scaling that drift by how far the iterate still has to travel, which
the probing displacement itself estimates, gives
\[
\wh\kappa_{\mathrm{corr}}
=\wh\kappa_{\text{late}}+\big|\wh\kappa_{\text{late}}-\wh\kappa_{\text{early}}\big|
\cdot\frac{\norm{x^{\text{end}}-x^{\text{start}}}}{\norm{x^{\text{end}}-x^{\text{mid}}}} ,
\]
at the cost of one extra pair of accumulators and no extra samples. The rule is a
heuristic. The supplementary materials measure it on two mechanisms and three
probing budgets:

\begin{center}\small
\begin{tabular}{@{}llcccc@{}}
\toprule
Mechanism & Probe & $\kappa$ & $\wh\kappa$ & drift & $\wh\kappa_{\mathrm{corr}}$ \\
\midrule
low rank  & cold, short & $2.19$ & $1.48$ & $0.054$ & $1.88$ \\
low rank  & cold, long  & $2.19$ & $1.48$ & $0.007$ & $1.53$ \\
low rank  & warm, long  & $2.19$ & $1.48$ & $0.001$ & $1.48$ \\
full rank & cold, short & $0.49$ & $0.49$ & $0.019$ & $0.71$ \\
full rank & cold, long  & $0.49$ & $0.49$ & $0.004$ & $0.57$ \\
full rank & warm, long  & $0.49$ & $0.49$ & $0.002$ & $0.49$ \\
\bottomrule
\end{tabular}
\end{center}

Three observations follow. The drift falls by one to two orders of magnitude as the probe warms
and lengthens, so it is a working detector of a stale measurement. The correction is
inert on a settled probe, leaving both estimates untouched, so it does not penalise a
good measurement with a false refusal. And on a cold short probe it inflates in the
safe direction, moving the low-rank case from $1.48$ towards the true $2.19$ while
leaving the full-rank case at $0.71$, still below the threshold. The gap does not close
entirely: even on the best probe $\wh\kappa=1.48$ against $\kappa=2.19$. The remaining
gap comes from the change of evaluation point. The estimate is of the incoherence \emph{at the probing point}, which is the quantity
Lemma~\ref{lem:realised} needs -- the inversion is performed on the disagreement
measured there -- whereas Proposition~\ref{thm:free} compares bias energies at the truth's
own optimum. The practical rule therefore has two parts. First, probe after a burn-in and
check the drift, which supports the inversion. Second, treat the bound of
Proposition~\ref{thm:free} as conditional on the drift being small, since only then are
the two evaluation points interchangeable.
\end{remark}

\subsection{Partial participation}
\label{app:partial}

In the usual federated setting only part of the crowd reports in each round. The
energies stay identified as long as participation does not depend on the tilts.

\begin{assumption}[Participation]
\label{as:part}
At round $k$ a set $P_k\subseteq[n]$ with $|P_k|=m\ge2$ reports, drawn uniformly at
random and independently of $\mathcal F_{k-1}$ and of the round-$k$ noises. The residual is
centred \emph{within} the participating set,
$r^k_i=\wh g^k_i-\tfrac1m\sum_{j\in P_k}\wh g^k_j$ for $i\in P_k$.
\end{assumption}

Lemma~\ref{lem:residual} applies verbatim with $[n]$ replaced by $P_k$: the residual
still cancels $\nabla f_0(x^{k-1})$ identically, since the centring is again an
average of the participants' reports. What changes is the linear map to be inverted
and the effective sample size.
\begin{proposition}[Identifiability under uniform participation]
\label{prop:partial}
Let Assumptions~\ref{as:crowd}, \ref{as:oracle} (linear tilts) and~\ref{as:part} hold,
$n\ge3$, and put $p:=m/n$ and $\rho:=\frac{(m-1)n}{m(n-1)}\in[\tfrac12,1]$.
\begin{enumerate}[label=(\roman*),itemsep=1pt,topsep=2pt]
\item \emph{(The disagreement shrinks by a known factor.)} For every $i$ and $k$,
$\E[r^k_i\mid\mathcal F_{k-1},\,i\in P_k]=\rho\,\wt\beta_i$, with $\wt\beta_i$ the
full-crowd disagreement of Theorem~\ref{thm:ident}.
\item \emph{(An unbiased statistic.)} For the odd and even halves $\mathcal K_a$ of the
probing rounds let $N^{(a)}_i:=|\{k\in\mathcal K_a:i\in P_k\}|$ and, when both are
positive, $\wh\beta^{(a)}_i:=\frac1{N^{(a)}_i}\sum_{k\in\mathcal K_a,\,i\in P_k}r^k_i$
and $\wh s_i:=\inner{\wh\beta^{(1)}_i}{\wh\beta^{(2)}_i}$. Conditionally on the tilts
and on the participation indicators of client $i$, $\E\wh s_i=\rho^2\norm{\wt\beta_i}^2$
and
$\E(\wh s_i-\rho^2\norm{\wt\beta_i}^2)^2\le2\rho^2\norm{\wt\beta_i}^2(\vartheta_{i,1}^2+\vartheta_{i,2}^2)+2\vartheta_{i,1}^2\vartheta_{i,2}^2$,
where
\[
\vartheta_{i,a}^2:=\frac{1}{N^{(a)}_i}\Big(\psi_{m,n}\bar q+\frac{\tau_{i,P}^2}{b}\Big),
\]
$\psi_{m,n}:=\frac{n(n-m)}{m(n-1)(n-2)}$ and
$\tau_{i,P}^2:=\sigma_i^2+\frac{1}{m(n-1)}\sum_{j\ne i}\sigma_j^2$. With probability at
least $1-2ne^{-pK_0/16}$, every $N^{(a)}_i\ge pK_0/4$.
\item \emph{(Weights.)} $\wh q:=(T^{-1}\wh s/\rho^2)_+$ satisfies
$\max_i|\wh q_i-q_i|\le C_n\big(\max_i|\wh s_i/\rho^2-\norm{\wt\beta_i}^2|+3\kappa\bar q\big)$.
\end{enumerate}
The energies are therefore identified for every $m\ge2$ once $n\ge3$. The price of
partial participation is a factor of order $1/p$ in the number of reports, a factor
$\rho^{-2}\le4$ from rescaling, and the variance $\psi_{m,n}\bar q$ of drawing the
participants, which vanishes at $m=n$.
\end{proposition}

Lemma~\ref{thm:weightsub} and Proposition~\ref{thm:free} then apply verbatim with the
$\Delta$ of (iii). The optimisation phase under partial participation has a sampling
variance of its own, which we do not analyse. A per-round squared disagreement has a
different mean, $\E[\norm{\wt\beta_i^P}^2\mid i\in P]$, because it also contains the
variance of drawing $P$; cross-fitting over rounds with independently drawn
participants removes that variance, so the map to invert is $\rho^2T$. Participation that is observed and independent of the tilts is
therefore harmless. Selection on the tilts themselves is a different matter.
\begin{proposition}[Selection on the tilts moves the target]
\label{prop:part-corr}
Suppose that the clients that report are those in a set $\mathcal R$, and that their
tilts are $\beta_i=c+\beta_i'$, $i\in\mathcal R$, with $c\in\R^d$ fixed and the
$\beta'_i$ independent and centred, as happens when tilts are drawn i.i.d.\ and
selection depends on each client's own tilt. Then:
(i) the transcript of any protocol has the same law under
$(f_0,(h_i)_{i\in\mathcal R})$ and under
$(f_0+\inner{c}{\cdot},(h_i-\inner{c}{\cdot})_{i\in\mathcal R})$, and the second
instance satisfies Assumption~\ref{as:crowd} on $\mathcal R$;
(ii) for every $w$ in the simplex on $\mathcal R$, $\bias(w)=c+\bias'(w)$ with
$\bias'(w):=\sum_iw_i\beta'_i$, so no weighting removes $c$;
(iii) the disagreements of the two instances coincide, so \CROWD{} computes the same
statistics under both, and they estimate the energies $\norm{\beta'_i}^2$ of the
second.
\end{proposition}

The mechanism is the one of Proposition~\ref{prop:corr}: a component shared by the
reporting clients is invisible to the disagreement. Stragglers who are systematically
the miscalibrated sites therefore shift the target the crowd is centred on, by an
amount that no weighting rule built on the residuals can detect. The guarantee of
Proposition~\ref{prop:partial} requires participation to be ignorable, an assumption
about the deployment that the data cannot check.
\section{Numerical study}
\label{sec:exp}

\paragraph{What is compared.} We compare nine rules: \emph{uniform}; \emph{size}
$w_i\propto n_i$ (\textsc{FedAvg}); \emph{inv-$\sigma$}, $w_i\propto\bar\sigma_i^{-2}$
($n_i/\sigma_i^2$ on real data); \emph{valid}, a tuned power of each site's held-out
loss; \emph{common}, one shared heterogeneity variance; \emph{no-$T^{-1}$}, \CROWD{}
without the inversion; \emph{best single}, the client with the smallest bias; \CROWD{};
and the \emph{oracle}, handed the realised $q_i$. All rules share optimiser, seeds and
gradients, so only $w$ differs.

\begin{figure}[t]
\centering
\begin{subfigure}[t]{0.49\linewidth}\centering
\includegraphics[width=\linewidth]{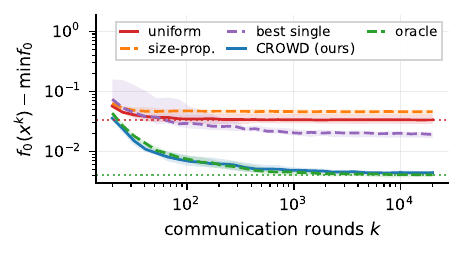}
\caption{Synthetic: bias floors}\label{fig:exp-a}
\end{subfigure}\hfill
\begin{subfigure}[t]{0.49\linewidth}\centering
\includegraphics[width=\linewidth]{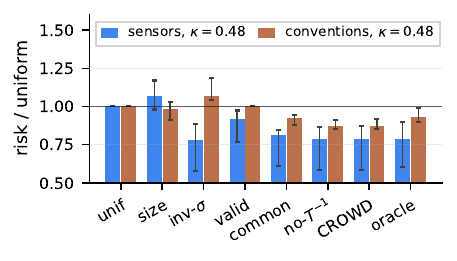}
\caption{Real federations}\label{fig:exp-b}
\end{subfigure}\\[2pt]
\begin{subfigure}[t]{0.49\linewidth}\centering
\includegraphics[width=\linewidth]{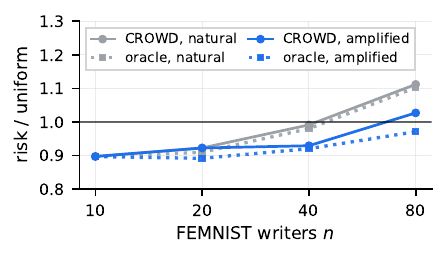}
\caption{FEMNIST writers}\label{fig:exp-c}
\end{subfigure}\hfill
\begin{subfigure}[t]{0.49\linewidth}\centering
\includegraphics[width=\linewidth]{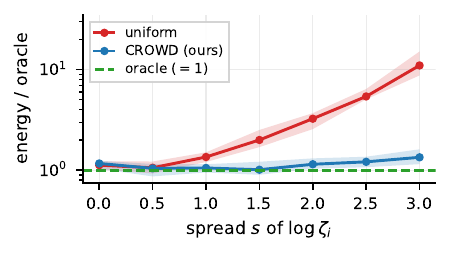}
\caption{Gain against the spread of $\zeta_i$}\label{fig:exp-d}
\end{subfigure}
\caption{(a) Excess risk on the truth, 12 seeds; dotted lines are the exact floors
$\tfrac12\bias(w)^\top A^{-1}\bias(w)$. (b, c) Excess risk relative to uniform, 8 and
12 seeds. (d) Bias energy relative to the oracle, 24 seeds. Bands and bars are
interquartile.}
\label{fig:exp}
\end{figure}

\paragraph{Bias floors (Figure~\ref{fig:exp-a}).} The synthetic instance satisfies the
model exactly. The truth $f_0$ is a quadratic on $\R^{150}$ with spectrum geometric in
$[\mu,L]=[0.5,5]$, and each of $n=20$ clients returns the gradient of
$f_i=f_0+\inner{\beta_i}{\cdot}$ with Gaussian noise. Bias and noise scales are drawn
\emph{independently}, log-uniformly in $[0.2,3]$ and $[0.5,3]$. Every rule converges
and stalls at the floor set by its own bias. Uniform weighting stops at
$3.358\cdot10^{-2}$ against $3.361\cdot10^{-2}$ predicted, and size weighting and the
best single client stop higher still. \CROWD{} stops $8\%$ above the oracle floor,
because its weights are built from the identified energies $\wh q_i$. The $7.56\times$
separation is in limit points, not rates.

\paragraph{Separating bias from noise (Figure~\ref{fig:exp-b}).} The data are $1797$
real digit scans, the model $\ell_2$-regularised multinomial logistic regression on a
frozen random feature map, so every risk is convex. Eight sites hold disjoint cohorts
seen through their own distortion, and $f_0$ averages the risk over forty distortions
from the same law. In the \emph{sensor} federation each site has a detector with its
own gain and offset on every pixel. \CROWD{}, the oracle and inverse-noise reach
$0.787$, $0.789$ and $0.778$ of the uniform excess risk: a miscalibrated
detector inflates noise along with bias. The \emph{convention} federation separates the two costs at the same $\kappa=0.48$.
All detectors share one severity, and sites differ only through their own
label-confusion channel. Cohort sizes run \emph{against} quality, so the worst site
is the largest one: quiet, but biased. Here precision weighting reaches $1.07$, worse than not weighting at all, and
validation tuning returns uniform weights. \CROWD{} reaches $0.875$, against $0.934$
for the bias-energy oracle. The reason is that $\wh s_i$ keeps only the part of a
site's disagreement that persists across rounds. Noise averages out, and the quiet
but biased site is down-weighted.

\paragraph{Real writers (Figure~\ref{fig:exp-c}).} In FEMNIST each client is one writer
and $f_0$ the writer-averaged risk of $400$ others, so the federation is not
constructed. With ten writers $\kappa=0.52<1$, the regime of
Theorem~\ref{thm:endtoend}, and \CROWD{} reaches $0.897$ of the uniform excess risk,
as does the oracle. For natural writers the gain fades and is gone at $n=40$. The
reason is that the bias energy $\AM/n$ of the uniform crowd decays with $n$, so the
weights act on a shrinking part of the risk. Amplifying each style by $\sqrt{n/10}$
slows that decay, and the gain holds: $0.929$ at $n=40$, although $\kappa\approx2$.
At $n=80$ ($\kappa\approx3.4$) the estimate stops working: \CROWD{} reaches $1.03$
and the oracle $0.97$.

\paragraph{Unbounded gain (Figure~\ref{fig:exp-d}).} We take a linear crowd with $n=20$,
$d=60$, unit noise and $\log\zeta_i$ uniform on an interval of width $s$. At $s=3$
uniform weighting falls $11.0\times$ behind the oracle, while \CROWD{} stays within
$1.34\times$. The gain of the oracle over uniform weighting is $\AM/\HM$ of the bias
energies, which grows without bound with the spread
(Proposition~\ref{prop:oracle}). \CROWD{} captures it because it estimates every
$q_i$ separately rather than one shared variance.

\begin{table}[t]
\centering
\small
\setlength{\tabcolsep}{4pt}
\begin{tabular}{@{}llcccccc@{}}
\toprule
Model, trained end to end & clients & $\kappa$ & $\wh\kappa$ & inv-$\sigma$ & \CROWD{} & oracle & \CROWD{}$<$unif. \\
\midrule
two-layer ReLU, raw FEMNIST pixels & $10$ writers & $1.09$ & $1.05$ & $0.815$ & $0.933$ & $0.816$ & $8/8$ \\
 & $20$ writers & $2.13$ & $2.52$ & $0.971$ & $\mathbf{0.891}$ & $0.970$ & $7/8$ \\
 & $40$ writers, ampl. & $3.88$ & $4.39$ & $1.434$ & $\mathbf{0.950}$ & $1.331$ & $4/6$ \\
convolutional net, raw digit scans & $8$ sensor sites & $1.63$ & $1.18$ & $1.141$ & $1.057$ & $1.218$ & $1/6$ \\
\bottomrule
\end{tabular}
\caption{Neural networks trained end to end: excess risk relative to uniform
weighting, medians (Appendix~\ref{app:mlp}). \CROWD{} is unguarded; $\wh\kappa$ is
the estimate the guard uses.}
\label{tab:nn}
\end{table}

\paragraph{Neural networks (Table~\ref{tab:nn}).} Outside the convex theory we train
a two-layer ReLU network on the writers of Figure~\ref{fig:exp-c} and a small
convolutional network on the sensor sites. The oracle uses the energies at a
reference network. On the two-layer network \CROWD{} lowers the pool risk by $5$--$11\%$ in the median,
with accuracy changing by less than one point. It is the only rule that improves on
uniform weighting in all three settings. \CROWD{} measures the disagreement along its
own trajectory, whereas inv-$\sigma$ and the oracle use one fixed point, and at
$n=20,40$ both lose to it. These gains are unguarded: $\wh\kappa\ge1$ in $20$ of $22$
seeds, so the guard below would have kept uniform weights and is conservative here.
On the convolutional network no rule improves on uniform weighting, and the guard
declines in five of six seeds.

\paragraph{Using \CROWD{} in practice.} Probe with uniform weights, refresh on the
doubling grid, and at each refresh estimate $\kappa$ and the drift between the two
half-averages (Proposition~\ref{prop:check}, Remark~\ref{rem:staleguard}). If
$\wh\kappa<1$ and the drift is small, weight, otherwise keep uniform weights. This is
a screening diagnostic, not a certificate. The estimate $\wh\kappa$ is not consistent for
$\kappa$: its false-accept rate is $4.2\%$ in the linear regime, and on a strongly
curved crowd it under-reports $\kappa$ three- to five-fold
(Appendix~\ref{app:regime}). Appendix~\ref{app:exp} adds further ablations. Among
them, local-step \textsc{FedAvg}, \textsc{FedProx} and \textsc{SCAFFOLD} reach the
same floor under the same weights (Figure~\ref{fig:fedopt}).

\section{Discussion and limitations}
\label{sec:conclusion}

Reading the weights of a federated objective as an estimator of a common truth turns
weight design into an identification problem. The truth itself is identifiable only up
to a linear tilt, and correcting individual clients cannot help. The expected per-client
bias energies are identified by the law of the crowd's disagreement, and the realised
ones are estimable up to an incoherence floor. \CROWD{} obtains them at no extra cost
from the reports of the optimisation run. Its excess risk meets a Bayesian minimax floor
in the same constant, per instance as the horizon grows and unconditionally in the
diffuse-prior limit, and for arbitrary realised tilts it stays
$\tfrac{1+\kappa}{1-\kappa}$-competitive with the optimal weights.

\paragraph{Limitations.} Centring is untestable, so calling an aggregate closer to
ground truth is a claim about the crowd (Propositions~\ref{prop:ident-neg}
and~\ref{prop:check}); participation that depends on the clients' own tilts shifts the
target by an amount no residual-based rule can detect
(Proposition~\ref{prop:part-corr}). The gain also needs a slowly vanishing summed tilt
and not many more clients than independent causes of heterogeneity
(Figures~\ref{fig:exp} and~\ref{fig:regime}): the window in which $\kappa<1$ is narrow,
and the guard of Remark~\ref{rem:guard} is a screening diagnostic rather than a
certificate, which under-reports $\kappa$ on strongly curved crowds. The guarantees
assume a strongly convex truth; the neural-network runs of Table~\ref{tab:nn} are a
stress test outside the theory.

\section*{Acknowledgements}
The work of I.~Kuruzov and A.~Gasnikov was supported by the Ministry of Science and
Higher Education of the Russian Federation (Grant No.~075-15-2024-529).

\bibliographystyle{plainnat}
\bibliography{refs}

\appendix
\section{Notation}
\label{app:notation}

Symbols used in more than one section, with the place they are defined. Two
symbols collide: $\ell$ and $\ell_i$ are curvature bounds while $\ell_T$ is a
logarithmic factor in the stepsize, and $K$ counts samples per client in $\Psi(K)$ but
rounds in Algorithm~\ref{alg:weights}, where $K$ rounds give $bK$ samples. The family
$\Psi$ carries several decorations, listed together below.

\begin{table}[h]
\centering
\small
\setlength{\tabcolsep}{3.5pt}
\begin{tabular}{@{}@{}l@{\hspace{5pt}}l@{\hspace{5pt}}p{72mm}@{}}
\toprule
Symbol & Defined in & Meaning \\
\midrule
$n,d,b,K$ & Sec.~\ref{sec:setup} & clients, dimension, batch, samples per client \\
$f_0,\ x_0^\star,\ \mu,L$ & Assum.~\ref{as:reg} & the truth, its optimum, its curvature bounds \\
$h_i,\ f_i=f_0+h_i$ & Assum.~\ref{as:crowd} & client $i$'s perturbation and risk \\
$\ell_i,\ \ell,\ \lambda_w,\ L'$ & Assum.~\ref{as:crowd} & curvature bound of $h_i$; $\ell=\max_i\ell_i$; $\lambda_w=\sup\norm{\nabla^2H_w}$; $L'=L+\lambda$ \\
$\ell_T$ & Thm.~\ref{thm:endtoend} & logarithmic factor in the tuned stepsize, unrelated to $\ell$ \\
$\beta_i,\ q_i,\ \zeta_i^2$ & Assum.~\ref{as:crowd} & realised tilt, realised energy $\norm{\beta_i}^2$, expected energy \\
$\bar\beta,\ \wt\beta_i$ & Eq.~\eqref{eq:resid} & crowd mean and disagreement of the tilts \\
$\bias(w),\ \Gamma_w$ & Eq.~\eqref{eq:intro-bias} & aggregate bias vector and its norm \\
$V(w)$ & Sec.~\ref{sec:est} & diagonal proxy $\sum_iw_i^2q_i$ for $\Gamma_w^2$ \\
$\AM,\HM$ & Prop.~\ref{prop:oracle} & arithmetic and harmonic means of the energies \\
$\Psi_0$ & Prop.~\ref{prop:oracle} & $\HM/n$, the attainable expected bias energy \\
$\Psi(K),\ \wh\Psi(K)$ & Eq.~\eqref{eq:psiK} & finite-sample target and its realised counterpart \\
$\kappa,\ \wh\kappa$ & Assum.~\ref{as:incoh}, Prop.~\ref{prop:check} & incoherence of the realised tilts and its estimate \\
$\sigma_i^2,\ \tau_i^2$ & Assum.~\ref{as:oracle} & oracle noise and its centred counterpart \\
$\wh g^k_i,\ r^k_i$ & Eq.~\eqref{eq:resid} & report and residual at round $k$ \\
$T,\ C_n$ & Thm.~\ref{thm:ident} & moment map and its $\ell_\infty$ condition number \\
$\wh s_i,\ \wh q_i,\ \wh w$ & Eq.~\eqref{eq:est} & cross-fitted statistic, inverted energy, weights \\
$\Delta,\ \delta_i,\ \vartheta_i$ & Lem.~\ref{thm:est} & weight error; its staleness and statistical parts \\
$\epsilon,\ \epsilon_T$ & Lem.~\ref{lem:transfer}(ii) & Young trade-off parameter of the analysis \\
$P_k,\ m,\ p,\ \rho$ & Sec.~\ref{app:partial} & participating set, its size and rate $m/n$, the shrinkage of the disagreement \\
\bottomrule
\end{tabular}
\end{table}

Hats denote estimates computed from the transcript and tildes denote quantities
centred within the crowd, so $\wt x_i=x_i-\tfrac1n\sum_jx_j$ throughout. Statements
labelled \emph{realised} or \emph{deterministic} hold for one fixed draw of the tilts
with no expectation.
\section{Proofs for Section~\ref{sec:setup}}
\label{app:setup}

Throughout, $\E[\cdot\mid B]$ denotes expectation conditional on the realised tilts.
Recall $H_w=\sum_iw_ih_i$, $F_w=f_0+H_w$, $\bias(w)=\nabla H_w(x_0^\star)=B^\top w$,
$\Gamma_w=\norm{\bias(w)}$, $V(w)=\sum_iw_i^2q_i$.

\subsection{Proof of Proposition~\ref{prop:rigid}}
\label{app:pf-rigid}

The proof works with gradients only, so it needs no integrability of $h_i$ itself:
Assumption~\ref{as:crowd} already presupposes that $\nabla h_i(x)$ is integrable for
every $x$. The idea is that convexity makes a certain gradient pairing non-negative
while the centring makes its mean zero, which forces it to vanish.

\begin{proof}
Let $U$ be the open convex set, fix $i$, and write $h=h_i$.

\emph{Step 1: the monotonicity defect vanishes pointwise.} For $x,y\in U$ put
$M(x,y):=\inner{\nabla h(x)-\nabla h(y)}{x-y}$. A differentiable convex function has a
monotone gradient, so $M(x,y)\ge0$ almost surely. By linearity of expectation and
Assumption~\ref{as:crowd},
$\E M(x,y)=\inner{\E\nabla h(x)-\E\nabla h(y)}{x-y}=0$. A non-negative random variable
with zero mean vanishes almost surely, so $M(x,y)=0$ almost surely, for each fixed pair.

\emph{Step 2: simultaneously for all pairs.} Let $U_0\subseteq U$ be countable and
dense. By countable additivity there is an event of probability one on which
$M(x,y)=0$ for all $x,y\in U_0$. On that event, continuity of $\nabla h$ ($h\in C^2$ by
Assumption~\ref{as:crowd}; $C^1$ would suffice) makes $M$ continuous on $U\times U$, so $M(x,y)=0$ for all $x,y\in U$.

\emph{Step 3: a vanishing defect forces affinity.} Fix $x_0\in U$ and $x\in U$, and put
$g(t):=h(x_0+t(x-x_0))$ for $t\in[0,1]$. Then $g'(t)-g'(0)=\inner{\nabla
h(x_0+t(x-x_0))-\nabla h(x_0)}{x-x_0}=M(x_0+t(x-x_0),x_0)/t=0$ for $t\in(0,1]$. Hence
$g$ is affine and $h(x)=h(x_0)+\inner{\nabla h(x_0)}{x-x_0}$ for every $x\in U$: $h$ is
affine on $U$. Since $i$ was arbitrary and there are finitely many clients, the
conclusion holds for all of them on one event of probability one.
\end{proof}

The same argument with a deterministic crowd shows: if $h_1,\dots,h_n$ are
differentiable and convex and $\sum_i\lambda_i\nabla h_i$ is constant for some
$\lambda_i>0$, then every $h_i$ is affine, since $\sum_i\lambda_iM_i(x,y)=0$ with
$M_i\ge0$ forces each $M_i(x,y)=0$.

\subsection{Proof of Lemma~\ref{lem:transfer}}
\label{app:pf-transfer}

\begin{proof}
\emph{(i)} $\nabla^2F_w=\nabla^2f_0+\nabla^2H_w$ and
$\norm{\nabla^2H_w}_{\mathrm{op}}\le\lambda_w$ give the moduli. For the
displacement, $\nabla F_w(x_0^\star)=\nabla f_0(x_0^\star)+\nabla H_w(x_0^\star)=\bias(w)$,
so by $(\mu-\lambda_w)$-strong convexity of $F_w$,
$\norm{x_0^\star-x_w^\star}\le\norm{\nabla F_w(x_0^\star)}/(\mu-\lambda_w)=\Gamma_w/(\mu-\lambda_w)$.

\emph{(ii)} By $L$-smoothness of $f_0$ and $\nabla f_0(x_0^\star)=0$,
$f_0(x)-f_0^\star\le\tfrac L2\norm{x-x_0^\star}^2\le L\norm{x-x^\star_w}^2+L\norm{x^\star_w-x_0^\star}^2$,
and $\norm{x-x_w^\star}^2\le\tfrac{2}{\mu-\lambda_w}[F_w(x)-\min F_w]$ by strong
convexity. Combining with (i) gives the claim.

For the linear case ($\lambda_w=0$, $H_w=\inner{\bias}{\cdot}$) we argue through the
gradient instead, which is what makes the constant sharp. Strong convexity gives the
Polyak--\L{}ojasiewicz inequality $f_0(x)-f_0^\star\le\norm{\nabla f_0(x)}^2/(2\mu)$,
and $\nabla f_0(x)=\nabla F_w(x)-\bias(w)$, so by Young's inequality, for any
$\epsilon>0$,
\[
\norm{\nabla f_0(x)}^2\le(1+\epsilon)\Gamma_w^2+\Big(1+\tfrac1\epsilon\Big)\norm{\nabla F_w(x)}^2 .
\]
Smoothness of $F_w$ (constant $L$, since $\lambda_w=0$) gives
$\norm{\nabla F_w(x)}^2\le2L[F_w(x)-\min F_w]$, and dividing by $2\mu$ yields the
claim.

\emph{(iii)} $\nabla f_0(x_w^\star)=-\nabla H_w(x^\star_w)$ and, by $\lambda_w$-Lipschitz
continuity of $\nabla H_w$ together with (i),
\[
\norm{\nabla H_w(x_w^\star)}\ \ge\ \Gamma_w-\lambda_w\norm{x_w^\star-x_0^\star}
\ \ge\ \Gamma_w\Big(1-\frac{\lambda_w}{\mu-\lambda_w}\Big)=\Gamma_w\frac{\mu-2\lambda_w}{\mu-\lambda_w}.
\]
For convex $L$-smooth $f_0$ the descent lemma at $y^+=y-\nabla f_0(y)/L$ gives
$f_0(y)-f_0^\star\ge\norm{\nabla f_0(y)}^2/(2L)$; apply at $y=x_w^\star$.
\end{proof}

\subsection{Proof of Lemma~\ref{lem:transfer-c}}
\label{app:pf-transfer-constrained}

\begin{proof}
Put $m:=\mu-\lambda_w$; $F_w$ is $m$-strongly convex on $\Xset$ as in
Lemma~\ref{lem:transfer}(i). First-order optimality of $x^\star_w$ over $\Xset$ gives
$\inner{\nabla F_w(x^\star_w)}{y-x^\star_w}\ge0$ for every $y\in\Xset$.

\emph{Displacement.} Strong monotonicity of $\nabla F_w$ and optimality with
$y=x_0^\star$ give
\[
m\norm{x_0^\star-x^\star_w}^2\le\inner{\nabla F_w(x_0^\star)-\nabla F_w(x^\star_w)}{x_0^\star-x^\star_w}
\le\inner{\nabla F_w(x_0^\star)}{x_0^\star-x^\star_w}\le\Gamma_w\norm{x_0^\star-x^\star_w},
\]
since $\nabla F_w(x_0^\star)=\nabla f_0(x_0^\star)+\bias(w)=\bias(w)$, the truth's
minimiser being interior.

\emph{Excess risk.} Strong convexity and optimality give
$F_w(x)-F_w(x^\star_w)\ge\inner{\nabla F_w(x^\star_w)}{x-x^\star_w}+\tfrac m2\norm{x-x^\star_w}^2\ge\tfrac m2\norm{x-x^\star_w}^2$.
Since $\nabla f_0(x_0^\star)=0$ and $f_0$ is $L$-smooth on the segment
$[x_0^\star,x]\subseteq\Xset$,
$f_0(x)-f_0^\star\le\tfrac L2\norm{x-x_0^\star}^2\le L\norm{x-x^\star_w}^2+L\norm{x^\star_w-x_0^\star}^2$,
and the two displays combine to the claim.
\end{proof}

\subsection{Proof of Proposition~\ref{prop:oracle}}
\label{app:pf-oracle}

\begin{proof}
\emph{Identity.} $\E\norm{\sum_iw_i\beta_i}^2=\sum_{i,j}w_iw_j\E\inner{\beta_i}{\beta_j}$;
for $i\ne j$ independence and $\E\beta_i=\E\nabla h_i(x_0^\star)=0$ give
$\E\inner{\beta_i}{\beta_j}=\inner{\E\beta_i}{\E\beta_j}=0$, leaving
$\sum_iw_i^2\zeta_i^2$.

\emph{Optimum.} For $\one^\top w=1$, Cauchy--Schwarz gives
$1=(\sum_iw_i\sqrt{c_i}\cdot c_i^{-1/2})^2\le(\sum_iw_i^2c_i)(\sum_ic_i^{-1})$, so
$\sum_iw_i^2c_i\ge(\sum_ic_i^{-1})^{-1}$ with equality iff $w_i\propto c_i^{-1}$.
That candidate is strictly positive, hence lies in the relative interior of
$\simplex$: the affine and simplex-constrained problems have the same solution and
value. Strict convexity gives uniqueness.

\emph{(i)} The ratio of the two values is $(\AM/n)/(\HM/n)=\AM/\HM\ge1$ by the
inequality between the arithmetic and harmonic means, with equality iff all
$\zeta_i^2$ coincide. \emph{(ii)} $\Psi_0=(\sum_j\zeta_j^{-2})^{-1}$ and
$\sum_j\zeta_j^{-2}\ge\zeta_{\min}^{-2}$ give $\Psi_0\le\zeta_{\min}^2$.
\emph{(iii)} Take $\zeta_1^2=R$ and $\zeta_j^2=1$ for $j\ge2$. Then
$\AM=(R+n-1)/n$ and $\HM=n/(R^{-1}+n-1)$, so
$\AM/\HM=(R+n-1)(R^{-1}+n-1)/n^2\ge(n-1)R/n^2\to\infty$ as $R\to\infty$.
\end{proof}

\subsection{Proof of Lemma~\ref{lem:harmonic}}
\label{app:pf-harmonic}

\begin{proof}
By Proposition~\ref{prop:oracle}, $H(c)=\min_{\one^\top w=1}\phi_c(w)$ with
$\phi_c(w):=\sum_iw_i^2c_i$, and the minimiser $w^c\propto c^{-1}$ is unique. As a
pointwise minimum of functions linear in $c$, $H$ is concave, and $H(tc)=tH(c)$ for
$t>0$ is immediate. For superadditivity and its equality case, write
$w:=w^{a+b}$ and use $\phi_{a+b}=\phi_a+\phi_b$:
\[
H(a+b)=\phi_a(w)+\phi_b(w)\ \ge\ H(a)+H(b),
\]
with equality iff $\phi_a(w)=H(a)$ and $\phi_b(w)=H(b)$, that is, by uniqueness of the
minimisers, iff $w^{a+b}=w^a=w^b$. Since $w^c$ is the normalisation of
$(c_1^{-1},\dots,c_n^{-1})$, $w^a=w^b$ holds iff $a^{-1}\propto b^{-1}$, i.e.\ iff
$b=ta$ for some $t>0$; and then $w^{a+b}=w^a$ as well.
\end{proof}
\section{Proofs for Section~\ref{sec:alg}}
\label{app:alg}

This section proves the two impossibilities of Proposition~\ref{prop:ident-neg} and
establishes the property of the residual~\eqref{eq:resid} that the algorithm rests on.

\subsection{Proof of Proposition~\ref{prop:ident-neg}(i)}
\label{app:identneg-i}

For $c\in\R^d$ let $I_c$ denote the instance with truth $f_0+\inner{c}{\cdot}$ and
perturbations $h_i-\inner{c}{\cdot}$, the same noise laws, and the same realised
crowd otherwise; $I_0$ is the original instance.

\begin{proof}
\emph{Client risks.} For every $i$ and $x$,
$\big(f_0+\inner{c}{\cdot}\big)(x)+\big(h_i-\inner{c}{\cdot}\big)(x)=f_0(x)+h_i(x)=f_i(x)$,
so $I_c$ and $I_0$ have the same tuple $(f_1,\dots,f_n)$.

\emph{Transcript.} A protocol's transcript is produced from (a) the algorithm's
private randomness $U$, independent of the instance, (b) the reports, whose
conditional law given the past depends on the instance only through the queried
client's $\nabla f_i$ at the queried point, by Assumption~\ref{as:oracle}, and (c)
measurable maps of the two. We show by induction on the round $k$ that the joint law
of $(U,\text{reports up to }k)$ is the same under $I_c$ and $I_0$. For $k=0$ only $U$
is present. If the claim holds up to $k-1$, the round-$k$ query points are the same
measurable function of that history under both instances, and the conditional law of
the round-$k$ reports given the history is the same because the $\nabla f_i$ agree;
hence the joint law up to $k$ agrees.

\emph{Targets and biases.} On $\Xset=\R^d$ the minimiser of $f_0+\inner{c}{\cdot}$
solves $\nabla f_0(x)=-c$, which no solution of $\nabla f_0(x)=0$ satisfies when
$c\ne0$; so the targets differ. For linear tilts the tilts of $I_c$ at its own
minimiser are $\beta_i-c$, and $\sum_iw_i(\beta_i-c)=\bias(w)-c$ because
$\sum_iw_i=1$.
\end{proof}

The statement is about the realised crowd, which is what a protocol faces. The
shifted instance stays inside the model class. For Gaussian linear tilts every
realisation $(\beta_i-c)_i$ lies in the support of the centred law, so $I_c$ is itself
a possible realisation of an instance that satisfies Assumption~\ref{as:crowd}. At
fixed $n$ no amount of data selects between the two.

\subsection{Proof of Proposition~\ref{prop:ident-neg}(ii)}
\label{app:identneg-ii}

The deviation of client $i$ from the crowd mean at $x$ is
$\wt{\nabla h_i}(x):=\nabla h_i(x)-\tfrac1n\sum_j\nabla h_j(x)$, and the corrected
report is $\wh g^{\,\mathrm{corr}}_i:=\wh g_i-\wt{\nabla h_i}(x)$.

\begin{proof}
By Assumption~\ref{as:oracle}, $\E[\wh g_i\mid x]=\nabla f_0(x)+\nabla h_i(x)$, so
$\E[\wh g^{\,\mathrm{corr}}_i\mid x]=\nabla f_0(x)+\tfrac1n\sum_j\nabla h_j(x)=\nabla F_{w_{\mathrm{unif}}}(x)$
for every $i$. For any $w\in\simplex$ the aggregate therefore has conditional mean
$\sum_iw_i\nabla F_{w_{\mathrm{unif}}}(x)=\nabla F_{w_{\mathrm{unif}}}(x)$: it is an
unbiased gradient of $F_{w_{\mathrm{unif}}}$, whatever $w$ is. Its bias vector at
$x_0^\star$ is $\bar\beta=\bias(w_{\mathrm{unif}})$, whose expected energy is $\AM/n$
by Proposition~\ref{prop:oracle} with uniform weights, and $\AM/n>\HM/n=\Psi_0$
unless all $\zeta_i$ coincide.
\end{proof}

\subsection{Proof of Lemma~\ref{lem:residual}}
\label{app:residual}

\begin{proof}
\emph{Representation.} By Assumption~\ref{as:oracle},
$\wh g_i^k=\nabla f_0(x^{k-1})+\nabla h_i(x^{k-1})+\eps_i^k$. Averaging over $i$ and
subtracting gives
$r_i^k=\nabla h_i(x^{k-1})-\tfrac1n\sum_j\nabla h_j(x^{k-1})+\eps_i^k-\bar\eps^k
=\wt{\nabla h_i}(x^{k-1})+\wt\eps_i^k$, with $\wt\eps^k_i:=\eps^k_i-\bar\eps^k$, in
which $\nabla f_0(x^{k-1})$ has cancelled identically; no property of $x^{k-1}$ or of
$f_0$ was used. For linear tilts $\wt{\nabla h_i}\equiv\wt\beta_i$.

\emph{Maximal invariance.} If $\wh g'^k_i=\wh g^k_i+c^k$ for all $i,k$ then the
row means shift by $c^k$ and the residuals are unchanged. Conversely, if two arrays
have the same residuals then $\wh g'^k_i-\wh g^k_i=\bar{\wh g}'^k-\bar{\wh g}^k=:c^k$
does not depend on $i$. Hence the residual array separates exactly the orbits of the
group, which is the definition of a maximal invariant \citep[Ch.~6]{lehmann2005}.

\emph{Moments.} $\wt\eps^k_i=(1-\tfrac1n)\eps_i^k-\tfrac1n\sum_{j\ne i}\eps_j^k$ has
conditional mean zero. By conditional independence across clients and zero conditional
means, the cross terms vanish and
\begin{align*}
\E\big[\norm{\wt\eps_i^k}^2\mid\mathcal F_{k-1}\big]
&=\Big(1-\tfrac1n\Big)^2\E\big[\norm{\eps_i^k}^2\mid\mathcal F_{k-1}\big]
+\tfrac1{n^2}\sum_{j\ne i}\E\big[\norm{\eps_j^k}^2\mid\mathcal F_{k-1}\big]\\
&\le\frac{\sigma_i^2}{b}+\frac{1}{n^2}\sum_j\frac{\sigma_j^2}{b}=\frac{\tau_i^2}{b}. \qedhere
\end{align*}
\end{proof}
\section{Proofs for Section~\ref{sec:theory}}
\label{app:theory}

The subsections below follow the order of Section~\ref{sec:theory}: identification
and correlated crowds, then the estimation error, the end-to-end bound and the anytime
schedule, then the lower bound and its unconditional form, then the non-linear crowd and
the assumption-free guarantee, and finally checkability and partial participation.

\subsection{Proof of Theorem~\ref{thm:ident}}
\label{app:pf-ident}

The proof runs in five steps. Step~1 computes the expected disagreement energy
of one client, which is where independence removes the cross terms. Step~2 sums that
identity over clients to recover the total energy, and so inverts the map in closed
form. Step~3 reads the $\ell_\infty$ condition number off the rows of the inverse.
Step~4 checks that the quantity being inverted is a functional of the observable law.
Step~5 treats the degenerate case $n=2$.

\begin{proof}
\emph{Step 1: the forward map.} By Assumption~\ref{as:crowd},
$\E\inner{\beta_i}{\beta_j}=0$ for $i\ne j$ and $\E\norm{\beta_i}^2=\zeta_i^2$.
Hence, with $Z:=\sum_j\zeta_j^2$,
\[
\E\inner{\beta_i}{\bar\beta}=\frac1n\Big(\zeta_i^2+\sum_{j\ne i}\E\inner{\beta_i}{\beta_j}\Big)=\frac{\zeta_i^2}{n},
\qquad
\E\norm{\bar\beta}^2=\frac{1}{n^2}\sum_j\zeta_j^2=\frac{Z}{n^2},
\]
so that
\[
s_i=\E\norm{\beta_i-\bar\beta}^2=\zeta_i^2-\frac{2\zeta_i^2}{n}+\frac{Z}{n^2}
=\Big(1-\frac2n\Big)\zeta_i^2+\frac{1}{n^2}\sum_j\zeta_j^2=(T\zeta^2)_i .
\]

\emph{Step 2: inversion.} Summing over $i$,
$\sum_is_i=(1-\tfrac2n)Z+\tfrac{Z}{n}=Z\tfrac{n-1}{n}$, so
$Z=\tfrac{n}{n-1}\sum_is_i$, and therefore
\[
\zeta_i^2=\frac{n}{n-2}\Big(s_i-\frac{Z}{n^2}\Big)
=\frac{n}{n-2}\Big(s_i-\frac{1}{n(n-1)}\sum_js_j\Big),
\]
which is well defined for $n\ge3$. This exhibits a two-sided inverse of $T$, hence
$T$ is a bijection.

\emph{Step 3: conditioning.} The $i$th row of $T^{-1}$ has entries
$\frac{n}{n-2}\big(\delta_{ij}-\frac1{n(n-1)}\big)$, whose absolute values sum to
\[
\frac{n}{n-2}\Big[\Big(1-\frac1{n(n-1)}\Big)+\frac{n-1}{n(n-1)}\Big]
=\frac{n}{n-2}\cdot\frac{n(n-1)-1+(n-1)}{n(n-1)}=\frac{n^2-2}{(n-1)(n-2)} .
\]
Since $\norm{A}_{\infty\to\infty}$ is the maximum absolute row sum and all rows are
permutations of each other, $\norm{T^{-1}}_{\infty\to\infty}=C_n$ as claimed. The
numerical values follow, and $C_n=1+3/n+O(n^{-2})\to1$.

\emph{Step 4: identifiability.} For linear tilts Lemma~\ref{lem:residual} gives
$r^k_i=\wt\beta_i+\wt\eps^k_i$ with $\wt\beta$ measurable with respect to
$\mathcal F_0$ and $(\wt\eps^k_i)_k$ a martingale difference sequence. For any two
rounds $k<l$,
\[
\E\inner{r^k_i}{r^l_i}=\E\norm{\wt\beta_i}^2+\E\inner{\wt\beta_i}{\wt\eps^k_i}
+\E\inner{\wt\beta_i+\wt\eps^k_i}{\E[\wt\eps^l_i\mid\mathcal F_{l-1}]}=s_i ,
\]
because the middle term vanishes on conditioning on $\mathcal F_{k-1}$ and the last on
conditioning on $\mathcal F_{l-1}$. So $s_i$ is a functional of the joint law of the
residuals of two rounds, whatever the query points and without knowledge of the noise
level; $\zeta^2=T^{-1}s$ is then identified, and so is $w^\star_i\propto\zeta_i^{-2}$.

\emph{Step 5: failure at $n=2$.} For $n=2$, $(Tq)_i=0\cdot q_i+\tfrac14(q_1+q_2)$, so $T$
has rank one and $\zeta_1^2,\zeta_2^2$ are not separately identified.
\end{proof}

\subsection{Proof of Proposition~\ref{prop:corr}}
\label{app:pf-corr}

The proof runs in four steps. Step~1 evaluates the expected bias energy of a
correlated crowd and identifies the optimal weights. Step~2 shows that the two
candidate covariances have the same disagreement law, so no shift-invariant rule can
tell them apart. Step~3 checks that both are legitimate covariances. Step~4 computes
the two optima and their values, which differ.

\begin{proof}
\emph{Step 1: the objective.} With $\Cov(\beta_i,\beta_j)=\Sigma_{ij}I_d/d$ and
$\E\beta=0$,
\[
\E\norm{\bias(w)}^2=\sum_{i,j}w_iw_j\,\E\inner{\beta_i}{\beta_j}
=\sum_{i,j}w_iw_j\tr\big(\Sigma_{ij}I_d/d\big)=w^\top\Sigma w .
\]
Minimising over
$\{\one^\top w=1\}$ gives $w\propto\Sigma^{-1}\one$ with value
$(\one^\top\Sigma^{-1}\one)^{-1}$, by the same Cauchy--Schwarz (in the
$\Sigma$-inner product) as in Proposition~\ref{prop:oracle}.

\emph{Step 2: same disagreement law.} $\wt\beta=(J\otimes I_d)\beta$ has covariance
$J\Sigma J\otimes I_d/d$. Since $J\one=0$,
$J(\one u^\top+u\one^\top)J=(J\one)u^\top J+Ju\,(\one^\top J)=0$, so
$J\Sigma'J=J\Sigma J$. Hence $\wt\beta$ has the same law under both models, and by
Lemma~\ref{lem:residual} so does the residual array $(r_i^k)$ (the noise law is
unchanged). Any rule that is a measurable function of the residuals therefore has
the same output law under $\Sigma$ and $\Sigma'$.

\emph{Step 2$'$: the kernel is exactly this family.} Let $\Sigma$ be symmetric with
$J\Sigma J=0$. For $z\perp\one$, $Jz=z$, so $J\Sigma z=0$ and $\Sigma z=\alpha(z)\one$
for a linear functional $\alpha$ on $\one^\perp$; extend it by $\alpha(\one)=0$ and write
$\alpha(z)=a^\top z$ with $a\perp\one$. Then $\Sigma J=\one a^\top$, and with
$v:=\Sigma\one/n$,
$\Sigma=\Sigma J+\Sigma\tfrac{\one\one^\top}{n}=\one a^\top+v\one^\top$. Symmetry gives
$\one a^\top+v\one^\top=a\one^\top+\one v^\top$, i.e.\
$\one(a-v)^\top=(a-v)\one^\top$; a rank-one matrix $xy^\top$ equal to $yx^\top$ has
$y\propto x$, so $a-v=\gamma\one$ for a scalar $\gamma$. Hence
$\Sigma=\one a^\top+a\one^\top-\gamma\one\one^\top=\one u^\top+u\one^\top$ with
$u:=a-\tfrac\gamma2\one$. The map $u\mapsto\one u^\top+u\one^\top$ is injective:
applying the image to $\one$ gives $nu+(\one^\top u)\one$, which vanishes only if
$u=\gamma'\one$ for a scalar $\gamma'$, and then the image $2\gamma'\one\one^\top$
vanishes only if $\gamma'=0$. So the kernel has dimension $n$.

\emph{Step 3: positive definiteness.} $\Sigma=I_3\succ0$;
$\Sigma'=\begin{psmallmatrix}1.2&0.1&0.1\\0.1&1&0\\0.1&0&1\end{psmallmatrix}$
is symmetric and strictly diagonally dominant with positive diagonal, hence
$\succ0$.

\emph{Step 4: different optima.} For $\Sigma$ we have $\Sigma^{-1}\one=\one$, so
$w^\star(\Sigma)=\one/3$. For $\Sigma'$, solving $\Sigma'x=\one$ gives
$x_2=x_3=1-0.1x_1$ from the last two rows and, from the first,
\[
x_1\big(1+2t-2t^2\big)=1-2t\quad\text{with }t=0.1,
\qquad\text{so}\qquad
x_1=0.6780,\quad x_2=x_3=0.9322 .
\]
Normalising gives $w^\star(\Sigma')=(0.2667,0.3667,0.3667)$, interior to the simplex,
so the two optima differ. Finally the values differ strictly:
$w^\star(\Sigma)^\top\Sigma'w^\star(\Sigma)=\tfrac19\sum_{ij}\Sigma'_{ij}=\tfrac{3.6}{9}=0.4$,
whereas $(\one^\top\Sigma'^{-1}\one)^{-1}=1/2.5424=0.3933$. Hence a rule that is
optimal under $\Sigma$ (necessarily outputting $\one/3$, up to the shared output law)
is suboptimal under $\Sigma'$ by a factor $1.017$; in the other direction
$w^\star(\Sigma')^\top\Sigma w^\star(\Sigma')=0.3400$ against $1/3$, a factor $1.020$.
\end{proof}

\subsection{Proof of Lemma~\ref{lem:proxy}}
\label{app:pf-proxy}

\begin{proof}
$\norm{\bias(w)}^2=w^\top Qw=V(w)+w^\top(Q-D)w$ and, with $z=D^{1/2}w$,
$|w^\top(Q-D)w|=|z^\top D^{-1/2}(Q-D)D^{-1/2}z|\le\kappa\norm{z}^2=\kappa V(w)$.
\end{proof}

\subsection{Proof of Lemma~\ref{lem:realised}}
\label{app:pf-realised}

\begin{proof}
Expanding as in the previous proof but without expectations,
\[
\norm{\wt\beta_i}^2=q_i-\frac2n\Big(q_i+\sum_{j\ne i}\inner{\beta_i}{\beta_j}\Big)
+\frac1{n^2}\Big(\sum_jq_j+\sum_{j\ne l}\inner{\beta_j}{\beta_l}\Big)
=(Tq)_i+\eta_i,
\]
with $\eta_i=-\frac2n\sum_{j\ne i}\inner{\beta_i}{\beta_j}+\frac1{n^2}\sum_{j\ne l}\inner{\beta_j}{\beta_l}$.
Write $M:=D^{-1/2}(Q-D)D^{-1/2}$, $\norm{M}_{\mathrm{op}}\le\kappa$. Then
\[
\Big|\sum_{j\ne i}\inner{\beta_i}{\beta_j}\Big|=\big|e_i^\top(Q-D)\one\big|
=\big|(D^{1/2}e_i)^\top M(D^{1/2}\one)\big|
\le\kappa\sqrt{q_i}\,\sqrt{\textstyle\sum_jq_j}=\kappa\sqrt{q_i}\,\sqrt{n\bar q},
\]
and $\big|\sum_{j\ne l}\inner{\beta_j}{\beta_l}\big|=|\one^\top(Q-D)\one|\le\kappa n\bar q$.
Using $q_i\le n\bar q$,
\[
|\eta_i|\le\frac2n\kappa\sqrt{q_i}\,\sqrt{n\bar q}+\frac{\kappa n\bar q}{n^2}
\le\frac{2\kappa}{n}\cdot n\bar q+\frac{\kappa\bar q}{n}\le3\kappa\bar q .
\]
The final claim follows by applying $T^{-1}$ and
$\norm{T^{-1}\eta}_\infty\le C_n\norm{\eta}_\infty$.
\end{proof}

\subsection{Proof of Lemma~\ref{lem:kappa}}
\label{app:pf-kappa}

\begin{proof}
Write $\beta_i=\zeta_iu_i$ with $u_i\sim\mathcal N(0,I_d/d)$ independent, so that
$q_i/\zeta_i^2=\norm{u_i}^2$ and $d\norm{u_i}^2\sim\chi^2_d$.

\emph{Norms.} The Laurent--Massart inequality \citep{laurent2000} gives
$\PP(|\chi^2_d/d-1|>2\sqrt{t/d}+2t/d)\le2e^{-t}$. With $t=\log(4n/\delta)$ and a union
bound over $i$, the event $E_1=\{\max_i|\norm{u_i}^2-1|\le\epsilon_0\}$ has
probability at least $1-\delta/2$.

\emph{Gram matrix.} Let $Z\in\R^{n\times d}$ have the rows $\sqrt d\,u_i$, so $Z$ has
independent standard normal entries and $M_0:=\tfrac1dZZ^\top$ has entries
$\inner{u_i}{u_j}$. The rows of $Z^\top$ are independent $\mathcal N(0,I_n)$ vectors,
so \citet[Thm.~4.6.1]{vershynin2018} with $t=\sqrt{\log(4/\delta)}$ gives, with
probability at least $1-2e^{-t^2}=1-\delta/2$,
$\norm{M_0-I_n}_{\mathrm{op}}\le K^2\max(\epsilon_1',\epsilon_1'^2)$ for
$\epsilon_1':=\epsilon_1/K^2$. Since $K\ge1$, this is at most
$\max(\epsilon_1,\epsilon_1^2)=\epsilon_1$, using $\epsilon_1\le1$; call this event
$E_2$.

\emph{Normalisation.} Since $Q_{ij}=\zeta_i\zeta_j\inner{u_i}{u_j}$ and
$D_{ii}=\zeta_i^2\norm{u_i}^2$, we have $D^{-1/2}QD^{-1/2}=\Lambda M_0\Lambda$ with
$\Lambda:=\diag(\norm{u_i}^{-1})$, whose diagonal is one. Hence
\[
D^{-1/2}(Q-D)D^{-1/2}=\Lambda M_0\Lambda-I=\Lambda(M_0-I)\Lambda+(\Lambda^2-I).
\]
On $E_1$, $\norm\Lambda^2_{\mathrm{op}}=\max_i\norm{u_i}^{-2}\le1/(1-\epsilon_0)$ and
$\norm{\Lambda^2-I}_{\mathrm{op}}=\max_i|1-\norm{u_i}^2|/\norm{u_i}^2\le\epsilon_0/(1-\epsilon_0)$.
On $E_1\cap E_2$, which has probability at least $1-\delta$, the triangle inequality
gives $\kappa\le(\epsilon_1+\epsilon_0)/(1-\epsilon_0)$.

\emph{Dimension.} Let $C':=\max(64,4C^2)$ and $d\ge C'(n+\log(4n/\delta))$. Then
$\log(4n/\delta)/d\le\tfrac1{64}$, so $\epsilon_0\le\tfrac28+\tfrac2{64}<\tfrac12$;
and $n/d\le\tfrac1{4C^2}$, $\log(4/\delta)/d\le\tfrac1{4C^2}$, so
$\epsilon_1\le C(\tfrac1{2C}+\tfrac1{2C})=1$.
\end{proof}

\subsection{Proof of Lemma~\ref{lem:gauss-eta}}
\label{app:pf-gauss}

\begin{proof}
Write $\bar\zeta^2:=\frac1n\sum_j\zeta_j^2$ and recall
$\eta_i=-\frac2nX_i+\frac1{n^2}Y$ with $X_i:=\inner{\beta_i}{\sum_{j\ne i}\beta_j}$
and $Y:=\sum_{j\ne l}\inner{\beta_j}{\beta_l}$ from the proof of
Lemma~\ref{lem:realised}. We use three events, each of probability at least
$1-\delta/3$.

\emph{Event $E_0$.} Lemma~\ref{lem:kappa} with $\delta/3$ in place of $\delta$ gives
$\max_i|q_i/\zeta_i^2-1|\le\tfrac12$, hence $\tfrac12\bar\zeta^2\le\bar q\le\tfrac32\bar\zeta^2$.

\emph{Event $E_1$: the linear terms.} Conditionally on $\beta_i$, $\sum_{j\ne i}\beta_j$
is $\mathcal N(0,\tfrac1d\sum_{j\ne i}\zeta_j^2I_d)$, so $X_i$ is centred Gaussian with
variance $q_i\sum_{j\ne i}\zeta_j^2/d\le q_in\bar\zeta^2/d$. The Gaussian tail and a
union bound over $i$ give, with probability at least $1-\delta/3$ and for all $i$,
$|X_i|\le\sqrt{q_in\bar\zeta^2/d}\,\sqrt{2\log(6n/\delta)}$. On $E_0$,
$q_i\le\tfrac32\zeta_i^2\le\tfrac32n\bar\zeta^2$, so
$\tfrac2n|X_i|\le2\sqrt{3/2}\,\bar\zeta^2\sqrt{2\log(6n/\delta)/d}$.

\emph{Event $E_2$: the quadratic term.} Stack $g:=(\sqrt d\beta_j/\zeta_j)_{j}$, a
standard Gaussian vector in $\R^{nd}$. Then $Y=g^\top Ag$ where $A$ has the $d\times d$
blocks $A_{jl}=\tfrac{\zeta_j\zeta_l}{d}I_d$ for $j\ne l$ and $A_{jj}=0$. Its norms are
$\norm A_F^2=\tfrac1d\sum_{j\ne l}\zeta_j^2\zeta_l^2\le(n\bar\zeta^2)^2/d$ and
$\norm A_{\mathrm{op}}\le\tfrac1d\norm{\zeta\zeta^\top-\diag(\zeta^2)}_{\mathrm{op}}\le\tfrac1d\norm\zeta^2=n\bar\zeta^2/d$,
since $-\diag(\zeta^2)\preceq\zeta\zeta^\top-\diag(\zeta^2)\preceq\zeta\zeta^\top$.
Since $\E Y=\operatorname{tr}A=0$, the Hanson--Wright inequality
\citep[Thm.~6.2.1]{vershynin2018} gives, with probability at least $1-\delta/3$,
$|Y|\le C_{\mathrm{HW}}n\bar\zeta^2\big(\sqrt{\log(6/\delta)/d}+\log(6/\delta)/d\big)$,
which is at most $2C_{\mathrm{HW}}n\bar\zeta^2\sqrt{\log(6/\delta)/d}$ because
$\log(6/\delta)\le d$ under the dimension condition.

\emph{Conclusion.} On $E_0\cap E_1\cap E_2$, which has probability at least
$1-\delta$, $|\eta_i|\le\tfrac2n|X_i|+\tfrac1{n^2}|Y|\le C''\bar\zeta^2\sqrt{\log(6n/\delta)/d}$
for an absolute $C''$, and $\bar\zeta^2\le2\bar q$ on $E_0$.
\end{proof}

\subsection{Proof of Lemma~\ref{thm:est}}
\label{app:pf-est}

\paragraph{Notation for the proof.} Let $\mathcal F_k$ be the history through round
$k$, including the tilts and the private randomness, so that $\wt\beta$ is
$\mathcal F_0$-measurable. By Lemma~\ref{lem:residual}, which uses only
Assumption~\ref{as:oracle} and the decomposition $f_i=f_0+h_i$, the $\wt\eps^k_i$ form a
martingale difference sequence with
$\E[\norm{\wt\eps^k_i}^2\mid\mathcal F_{k-1}]\le\tau_i^2/b$. \emph{Adaptive query
points are allowed throughout.} The halves $\mathcal K_1,\mathcal K_2$ are the odd and
the even probing rounds, each of size $K_0/2$. Write
\[
\nu^{(a)}_i:=\frac2{K_0}\sum_{k\in\mathcal K_a}\wt\eps^k_i,\qquad
u^{(a)}_i:=\frac2{K_0}\sum_{k\in\mathcal K_a}\big(\wt{\nabla h_i}(x^{k-1})-\wt\beta_i\big),
\qquad m_i^{(a)}:=u_i^{(a)}+\nu_i^{(a)},
\]
so that $\wh\beta^{(a)}_i=\wt\beta_i+m_i^{(a)}$ and
\begin{equation}
\label{eq:err-decomp}
\wh s_i-\norm{\wt\beta_i}^2=\inner{\wt\beta_i}{m_i^{(1)}+m_i^{(2)}}+\inner{m_i^{(1)}}{m_i^{(2)}}.
\end{equation}
Each summand of $u^{(a)}_i$ has norm at most $\delta_i$ by hypothesis, so
$\norm{u_i^{(a)}}\le\delta_i$ on every trajectory; the proof of (ii) uses only the
weaker $(\E\norm{u^{(a)}_i}^2)^{1/2}\le\delta_i$. In the linear case $u^{(a)}_i=0$.

We use one fact about martingale sums twice.

\begin{lemma}[Cross products of martingale sums over disjoint rounds]
\label{lem:crossmart}
Let $(X_k)$ and $(Y_k)$ be $\R^d$-valued martingale difference sequences with respect
to $(\mathcal F_k)$, with $\E[\norm{X_k}^2\mid\mathcal F_{k-1}]\le v_X$ and
$\E[\norm{Y_k}^2\mid\mathcal F_{k-1}]\le v_Y$, and let $\mathcal K_1,\mathcal K_2$ be
disjoint sets of rounds. Put $A:=\sum_{k\in\mathcal K_1}X_k$ and
$B:=\sum_{l\in\mathcal K_2}Y_l$. Then $\E\norm A^2\le|\mathcal K_1|v_X$,
$\E\inner AB=0$ and $\E\inner AB^2\le|\mathcal K_1||\mathcal K_2|v_Xv_Y$.
\end{lemma}

\begin{proof}
The first claim is orthogonality of martingale increments. For the other two, order
the rounds of $\mathcal K_1\cup\mathcal K_2$ increasingly and put
$A_{t}:=\sum_{k\in\mathcal K_1,k<t}X_k$, $B_t:=\sum_{l\in\mathcal K_2,l<t}Y_l$, both
$\mathcal F_{t-1}$-measurable. Every pair $(k,l)\in\mathcal K_1\times\mathcal K_2$ has
$k\ne l$ and is counted once, at its later round, so
\[
\inner AB=\sum_{t}D_t,\qquad D_t:=\begin{cases}\inner{A_t}{Y_t},&t\in\mathcal K_2,\\
\inner{X_t}{B_t},&t\in\mathcal K_1.\end{cases}
\]
Each $D_t$ has conditional mean zero given $\mathcal F_{t-1}$, so $\E\inner AB=0$ and,
by orthogonality of the $D_t$,
$\E\inner AB^2=\sum_t\E D_t^2\le\sum_{t\in\mathcal K_2}v_Y\E\norm{A_t}^2+\sum_{t\in\mathcal K_1}v_X\E\norm{B_t}^2$.
With $\E\norm{A_t}^2\le v_X|\{k\in\mathcal K_1:k<t\}|$ and the analogue for $B_t$, the
right side is $v_Xv_Y$ times the number of pairs $(k,l)$, i.e.\
$|\mathcal K_1||\mathcal K_2|v_Xv_Y$.
\end{proof}

With $X_k=\tfrac2{K_0}\wt\eps^k_i$ this gives
$\E\norm{\nu^{(a)}_i}^2\le\frac{K_0}2\cdot\frac{4}{K_0^2}\cdot\frac{\tau_i^2}b=\vartheta_i^2$,
with $\vartheta_i^2=2\tau_i^2/(bK_0)$.

The three parts are proved in order, each from the
decomposition~\eqref{eq:err-decomp}. Part~(i) uses the disjointness of the halves to
kill both terms in expectation and then Lemma~\ref{lem:crossmart} for the variance.
Part~(ii) drops every independence claim between trajectory and noise and bounds the
terms absolutely. Part~(iii) applies $T^{-1}$ and adds the realised inversion error of
Lemma~\ref{lem:realised}.

\begin{proof}
\emph{(i) Linear crowd.} Here $m_i^{(a)}=\nu^{(a)}_i$. Each summand of
$\inner{\wt\beta_i}{\nu^{(1)}_i+\nu^{(2)}_i}$ is
$\tfrac2{K_0}\inner{\wt\beta_i}{\wt\eps^k_i}$, whose conditional mean given
$\mathcal F_{k-1}$ is zero because $\wt\beta_i$ is $\mathcal F_0$-measurable; and
$\E[\inner{\nu^{(1)}_i}{\nu^{(2)}_i}\mid\mathcal F_0]=0$ by
Lemma~\ref{lem:crossmart}, applied conditionally on $\mathcal F_0$. Hence
$\E[\wh s_i\mid\mathcal F_0]=\norm{\wt\beta_i}^2$ exactly.

For the variance, $\nu^{(1)}_i+\nu^{(2)}_i$ is a martingale sum over all $K_0$ rounds,
so $\E[\norm{\nu^{(1)}_i+\nu^{(2)}_i}^2\mid\mathcal F_0]\le4\tau_i^2/(bK_0)$ and, by
Cauchy--Schwarz with $\wt\zeta_i=\norm{\wt\beta_i}$ fixed given $\mathcal F_0$,
$\E[\inner{\wt\beta_i}{\nu^{(1)}_i+\nu^{(2)}_i}^2\mid\mathcal F_0]\le4\wt\zeta_i^2\tau_i^2/(bK_0)$.
Lemma~\ref{lem:crossmart} with $v_X=v_Y=\tfrac4{K_0^2}\cdot\tfrac{\tau_i^2}b$ and
$|\mathcal K_1|=|\mathcal K_2|=K_0/2$ gives
$\E[\inner{\nu^{(1)}_i}{\nu^{(2)}_i}^2\mid\mathcal F_0]\le\tfrac{K_0^2}4\cdot\tfrac{16\tau_i^4}{b^2K_0^4}=\tfrac{4\tau_i^4}{(bK_0)^2}$.
Applying $(u+v)^2\le2u^2+2v^2$ to~\eqref{eq:err-decomp} yields the stated bound
$8\wt\zeta_i^2\tau_i^2/(bK_0)+8\tau_i^4/(bK_0)^2$. No dimension appears because
Assumption~\ref{as:oracle} bounds $\E\norm{\eps}^2$ directly.

\emph{(ii) General.} The $u^{(a)}_i$ depend on the trajectory, which depends on the
past noise, so no independence between $u^{(a)}_i$ and $\nu^{(a)}_i$ is available and
we bound absolutely. By Cauchy--Schwarz applied to~\eqref{eq:err-decomp},
\[
\big|\wh s_i-\norm{\wt\beta_i}^2\big|
\le\wt\zeta_i\big(\norm{m_i^{(1)}}+\norm{m_i^{(2)}}\big)+\norm{m_i^{(1)}}\norm{m_i^{(2)}}.
\]
Minkowski's inequality in $L^2$ gives
$\norm{m^{(a)}_i}_{L^2}\le\norm{u^{(a)}_i}_{L^2}+\norm{\nu^{(a)}_i}_{L^2}\le\delta_i+\vartheta_i$,
using only $L^2$ bounds on each piece. Then $\E\norm{m^{(a)}_i}\le\delta_i+\vartheta_i$
by Jensen and
$\E[\norm{m^{(1)}_i}\norm{m^{(2)}_i}]\le\norm{m^{(1)}_i}_{L^2}\norm{m^{(2)}_i}_{L^2}\le(\delta_i+\vartheta_i)^2$
by Cauchy--Schwarz, so
$\E|\wh s_i-\norm{\wt\beta_i}^2|\le2\wt\zeta_i(\delta_i+\vartheta_i)+(\delta_i+\vartheta_i)^2$.
The same argument applies conditionally on $\mathcal F_0$, and on any $\mathcal F_{k_0}$
when the probing rounds come after $k_0$.

\emph{(iii) Weights.} Lemma~\ref{lem:realised} gives $\norm{\wt\beta}^2=Tq+\eta$ with
$\norm{\eta}_\infty\le3\kappa\bar q$, so that
\[
T^{-1}\wh s-q=T^{-1}\Big(\big(\wh s-\norm{\wt\beta}^2\big)+\eta\Big),
\]
and $\norm{T^{-1}}_{\infty\to\infty}=C_n$ by Theorem~\ref{thm:ident}. Clipping at zero
does not increase the error: the positive part is the projection onto $[0,\infty)$, a
$1$-Lipschitz map fixing $q_i\ge0$.

\emph{High probability.} In case (i), Chebyshev's inequality with threshold
$\sqrt{n\,\Var_i/\delta}$, $\Var_i$ the variance bound of (i), and a union bound over
$i$ give $\max_i|\wh s_i-\norm{\wt\beta_i}^2|\le\max_i\sqrt{n\Var_i/\delta}$ with
probability at least $1-\delta$. In case (ii), Markov's inequality with threshold
$n\E|\cdot|/\delta$ and a union bound give the analogous statement. A sharper form
of case (ii), used in Theorem~\ref{thm:nonlin}, applies Chebyshev to the noise alone:
$\PP(\norm{\nu^{(a)}_i}>\vartheta_i\sqrt{2n/\delta})\le\delta/(2n)$, so with
probability at least $1-\delta$, simultaneously for all $i$ and $a$,
$\norm{m^{(a)}_i}\le\delta_i+\vartheta_i\sqrt{2n/\delta}$ when
$\norm{u^{(a)}_i}\le\delta_i$ surely, and then
$|\wh s_i-\norm{\wt\beta_i}^2|\le2\wt\zeta_i(\delta_i+\vartheta_i\sqrt{2n/\delta})+(\delta_i+\vartheta_i\sqrt{2n/\delta})^2$.
\end{proof}
\subsection{Proof of Lemma~\ref{thm:weightsub}}
\label{app:pf-weightsub}

The argument uses only that the weights lie in the simplex, so it holds verbatim with
$\simplex$ replaced by any closed convex subset $\mathcal W\subseteq\simplex$, with
$\wh w$ and $w^\star$ the minimisers over $\mathcal W$, for instance a simplex with
capped coordinates.

\begin{proof}
For $w\in\simplex$, $\sum_iw_i^2\le(\sum_iw_i)^2=1$, so
$|\sum_iw_i^2c_i-\sum_iw_i^2\wh c_i|\le\sum_iw_i^2|c_i-\wh c_i|\le\Delta$ for any such
$w$. Let $w^\star$ minimise $\sum_iw_i^2c_i$ and $\wh w$ minimise $\sum_iw_i^2\wh c_i$
over the same feasible set; over $\simplex$ both are given in closed form by
Proposition~\ref{prop:oracle}, and $\wh w_i\propto\wh c_i^{-1}$. Then
\[
\sum_i\wh w_i^2c_i\le\sum_i\wh w_i^2\wh c_i+\Delta\le\sum_i(w^\star_i)^2\wh c_i+\Delta
\le\sum_i(w^\star_i)^2c_i+2\Delta ,
\]
the middle step by optimality of $\wh w$ for $\wh c$. For the specialisation, take
$c_i=q_i+\bar\sigma_i^2/(bT)$: then $V(\wh w)\le\sum_i\wh w_i^2c_i$ because
$c_i\ge q_i$, and $\min_{w\in\simplex}\sum_iw_i^2c_i=\wh\Psi(bT)$ by
Proposition~\ref{prop:oracle}. Finally consider the floored costs
$\wh c_i=\max\{(\wh q_i)_++\bar\sigma_i^2/(bT),\varrho\}$ of
Algorithm~\ref{alg:weights}, which are positive by construction -- the floor is what
guarantees positivity, since $(\wh q_i)_+=0$ is attainable and $\bar\sigma_i=0$ is
admissible. Write $\wh c_i^{\,0}:=(\wh q_i)_++\bar\sigma_i^2/(bT)$; then
$|\wh c_i^{\,0}-c_i|=|(\wh q_i)_+-q_i|\le|\wh q_i-q_i|\le\Delta$ because the positive
part is a $1$-Lipschitz projection onto a set containing $q_i$. If the floor is
inactive, $\wh c_i=\wh c_i^{\,0}$; if it is active, $\wh c_i=\varrho>\wh c_i^{\,0}$ and
$|\wh c_i-c_i|\le\max\{\varrho-c_i,\,c_i-\wh c_i^{\,0}\}\le\max\{\varrho,\Delta\}$, using
$c_i\ge0$. In both cases $\norm{\wh c-c}_\infty\le\Delta+\varrho$, and the first display
applies with $\Delta+\varrho$ in place of $\Delta$.
\end{proof}

\subsection{Proof of Lemma~\ref{lem:sgd}}
\label{app:pf-sgd}

\begin{proof}
\emph{One step.} The minimiser is a fixed point of the projected gradient map,
$x^\star=\proj(x^\star-\eta\nabla F(x^\star))$, by the first-order optimality
condition. Nonexpansiveness of $\proj$ then gives
$\norm{x_t-x^\star}\le\norm{x_{t-1}-x^\star-\eta(g_t-\nabla F(x^\star))}$. Squaring,
taking conditional expectations, and writing $e:=x_{t-1}-x^\star$ and
$\Delta_F:=\nabla F(x_{t-1})-\nabla F(x^\star)$,
\[
\E[\norm{x_t-x^\star}^2\mid\mathcal F_{t-1}]\le\norm e^2-2\eta\inner{\Delta_F}{e}+\eta^2\big(\norm{\Delta_F}^2+V\big).
\]
For a $\mu_F$-strongly convex $L_F$-smooth function,
$\inner{\Delta_F}{e}\ge\mu_F\norm e^2$ and $\inner{\Delta_F}{e}\ge\norm{\Delta_F}^2/L_F$
\citep[Sec.~2.1]{nesterov2018}; averaging the two gives
$\inner{\Delta_F}{e}\ge\tfrac{\mu_F}2\norm e^2+\tfrac1{2L_F}\norm{\Delta_F}^2$. Hence the
right side is at most
$(1-\mu_F\eta)\norm e^2-\eta(\tfrac1{L_F}-\eta)\norm{\Delta_F}^2+\eta^2V\le(1-\mu_F\eta)\norm e^2+\eta^2V$
for $\eta\le1/L_F$. Taking expectations given $\mathcal F_0$ gives the recursion.

\emph{Unrolling.} $a_T\le(1-\mu_F\eta)^Ta_0+\eta^2V\sum_{t<T}(1-\mu_F\eta)^t\le e^{-\mu_F\eta T}a_0+\eta V/\mu_F$.

\emph{Tuned step.} If $\eta=\ell/(\mu_FT)$ then
$e^{-\mu_F\eta T}a_0=a_0e^{-\ell}\le a_0V/(\mu_F^2R^2T)$ and
$\eta V/\mu_F=\ell V/(\mu_F^2T)$. Otherwise $\eta=1/(2L_F)<\ell/(\mu_FT)$, so
$e^{-\mu_F\eta T}a_0=a_0e^{-\mu_FT/(2L_F)}$ and $\eta V/\mu_F<\ell V/(\mu_F^2T)$. In
both cases the stated bound holds.
\end{proof}
\subsection{Proof of Theorem~\ref{thm:endtoend}}
\label{app:pf-endtoend}

The proof runs in three steps. Step~1 conditions on the probing phase, which freezes
the weights and makes Phase~II an ordinary stochastic-approximation run on a fixed
objective. Step~2 applies the \textsc{SGD} lemma to that objective. Step~3 converts
accuracy on $F_{\wh w}$ into accuracy on the truth by the sharp form of the transfer
lemma, and bounds the bias term by the diagonal proxy. Throughout, $\E_{K_0}$ denotes
expectation conditional on $\mathcal F_{K_0}$.

\begin{proof}
\emph{Step 1: freezing the weights.} $\wh w$, $\Delta$, $x^{K_0}$ and $R_0$ are
$\mathcal F_{K_0}$-measurable. During Phase~II the aggregate
$\sum_i\wh w_i\wh g^k_i$ has conditional mean $\nabla F_{\wh w}(x^{k-1})$ and, by
conditional independence across clients and $\sigma_i\le\bar\sigma_i$, conditional
variance at most $\sum_i\wh w_i^2\sigma_i^2/b\le V_\sigma$. Since the tilts are linear,
$\lambda_{\wh w}=0$ and $F_{\wh w}$ is $\mu$-strongly convex and $L$-smooth.

\emph{Step 2: optimisation.} Lemma~\ref{lem:sgd}, applied conditionally on
$\mathcal F_{K_0}$ with $a_0=\norm{x^{K_0}-x^\star_{\wh w}}^2\le R_0^2$, gives
$\E_{K_0}\norm{x^{K_0+T}-x^\star_{\wh w}}^2\le R_0^2e^{-\mu T/(2L)}+\ell_TV_\sigma/(\mu^2T)$,
and $L$-smoothness gives $F_{\wh w}(x)-\min F_{\wh w}\le\tfrac L2\norm{x-x_{\wh w}^\star}^2$.

\emph{Step 3: transfer.} The sharp form of Lemma~\ref{lem:transfer}(ii), a
deterministic inequality, holds pointwise at $x=x^{K_0+T}$; taking $\E_{K_0}$ gives,
for every $\epsilon>0$,
\[
\E_{K_0}[f_0(x^{K_0+T})-f_0^\star]\le\frac{1+\epsilon}{2\mu}\norm{\bias(\wh w)}^2
+\Big(1+\tfrac1\epsilon\Big)\frac{L}{\mu}\cdot\frac L2\,\E_{K_0}\norm{x^{K_0+T}-x^\star_{\wh w}}^2 .
\]
Lemma~\ref{lem:proxy} and Lemma~\ref{thm:weightsub} bound the first term through
$\norm{\bias(\wh w)}^2\le(1+\kappa)V(\wh w)\le(1+\kappa)(\wh\Psi(bT)+2\Delta)$, and
Step~2 bounds the second. This is the statement of Theorem~\ref{thm:endtoend}.
\end{proof}

\subsection{Proof of Corollary~\ref{cor:optimal}}
\label{app:pf-optimal}

\begin{proof}[Proof of Corollary~\ref{cor:optimal}]
\emph{Upper bound.} Work conditionally on the probing phase, on the event named in the
hypotheses, and apply Theorem~\ref{thm:endtoend} with $\epsilon=\epsilon_T:=T^{-1/4}$.
The bias term is $\tfrac{(1+\epsilon_T)(1+\kappa)}{2\mu}(\wh\Psi(bT)+2\Delta)$, at most
$\tfrac{(1+\epsilon_T)(1+\epsilon_1)(1+\epsilon_2)}{2\mu}\wh\Psi(bT)$ under the two
hypotheses. The other term is
$(1+T^{1/4})\tfrac{L^2}{2\mu}\big(R_0^2e^{-\mu T/(2L)}+\ell_TV_\sigma/(\mu^2T)\big)$.
At a fixed centre, $R_0$, $V_\sigma$ and $K_0$ do not depend on $T$ and
$\ell_T=O(\log T)$, so this term is $O(T^{-3/4}\log T)$; it vanishes, more slowly
than the horizon term $O(T^{-1}\log T)$ of Theorem~\ref{thm:endtoend}. Any schedule
with $\epsilon_T\to0$ and $\epsilon_T^{-1}=o(T/\log T)$ works the same way; $\epsilon$
appears only in the analysis, so the algorithm does not depend on it.

\emph{Lower bound.} \CROWD{} draws $b(K_0+T)$ samples per client, and a batch of $b$
samples returns a gradient whose noise has second moment $\sigma_i^2/b$, which a
query-based algorithm can reproduce by averaging $b$ queries. Theorem~\ref{thm:lower}
with $K=b(K_0+T)$ therefore bounds the minimax risk of the problem \CROWD{} faces by
$\Psi(b(K_0+T))/(2\mu)$ from below.

\emph{The ratio.} As $T\to\infty$, $\wh\Psi(bT)\to\wh\Psi_0$ and
$\Psi(b(K_0+T))\to\Psi_0$, since $(\sum_i(a_i+t_i)^{-1})^{-1}$ is continuous in
$t\ge0$. Hence the ratio of the upper bound to the lower bound tends to
$(1+\epsilon_1)(1+\epsilon_2)\wh\Psi_0/\Psi_0$. For the last factor, on the event of
Lemma~\ref{lem:kappa}, $(1-\epsilon_0)\zeta_i^2\le q_i\le(1+\epsilon_0)\zeta_i^2$ for
every $i$, and $H(c)=(\sum_ic_i^{-1})^{-1}$ is non-decreasing in each coordinate and
positively homogeneous, so $(1-\epsilon_0)\Psi_0\le\wh\Psi_0\le(1+\epsilon_0)\Psi_0$.
\end{proof}

The comparison is deliberately asymmetric, in two ways, and the corollary's statement
in the main text names both. First, it is pointwise in the centre: the upper bound
holds for every fixed $x_c$ as $T\to\infty$, while the lower bound is a supremum over
$x_c$. A fixed number of steps from a fixed starting point cannot be uniformly
accurate over $x_c\in\R^d$. Since $R_0$ is the only quantity through which the centre
enters, the upper bound is uniform over centres in any fixed ball for $T$ large. This
is the form in which an iterative method can approach a minimax floor over an
unbounded class. Second, it is conditional on the realisation: the upper bound is
against the realised benchmark $\wh\Psi$ on the event named in the hypotheses, while
the lower bound averages over the crowd prior. Concavity of $H$
(Lemma~\ref{lem:harmonic}) gives $\E\wh\Psi_0\le\Psi_0$ by Jensen, and
Lemma~\ref{lem:kappa} bounds the probability of the event. The corollary itself
therefore claims per-instance, asymptotic sharpness. Section~\ref{app:closure}
integrates it into an unconditional statement, through Theorem~\ref{thm:closure} and
Corollary~\ref{cor:closure}. The ingredients are a projection that makes the transient
uniform, tail control off the incoherence event, and the finite-$\rho$ Bayes floor
established in Step~4 of the proof of Theorem~\ref{thm:lower}.

\subsection{Proof of Theorem~\ref{thm:anytime}}
\label{app:pf-anytime}

The proof runs in four steps. Step~1 records the measurability fact that makes each
segment an ordinary run with a fixed objective. Step~2 locates the last refresh.
Step~3 proves the display by applying the \textsc{SGD} lemma to the final segment.
Step~4 proves the bound of Remark~\ref{rem:earlyrefresh} on $\wt R$ by a recursion
over the segments. Everything is conditional on the tilts.

\begin{proof}
\emph{Step 1: segmentation.} Number the distinct refreshes $t_1<\dots<t_J$; the
hypothesis $T\ge\gamma^{j_0+1}$ gives $J\ge1$, so the last refresh and $\Delta_J$ are
defined. Repeated grid values $\lceil\gamma^j\rceil$, possible for $\gamma$ close to one
and absent for $\gamma\ge2$, are counted once, so no segment has length zero; the
bound on the number $N_\star$ of short segments in Step~4 counts distinct refreshes and
is unaffected. Put $t_0:=0$,
$t_{J+1}:=T$; segment $j$ is the rounds $t_j<k\le t_{j+1}$, run with weights
$w^{(j)}$ ($w^{(0)}$ uniform). The weights $w^{(j)}$ are
$\mathcal F_{t_j}$-measurable, being computed from rounds before the segment starts.
Conditionally on $\mathcal F_{t_j}$ the objective $F_{w^{(j)}}$ is therefore fixed, the
aggregate is conditionally unbiased for its gradient with conditional variance at most
$\sum_i(w^{(j)}_i)^2\sigma_i^2/b\le\bar V$, and Lemma~\ref{lem:sgd} applies to the
segment.

\emph{Step 2: the last refresh.} The grid points are $\lceil\gamma^j\rceil\le T/\gamma$,
so $t_J\le T/\gamma$ and $T_J=T-t_J\ge(1-\gamma^{-1})T$. The next grid point exceeds
$T/\gamma$, so $\gamma^{j+1}+1>T/\gamma$ for the last index $j$, whence
$t_J\ge\gamma^j>T/\gamma^2-1$.

\emph{Step 3: the final segment.} The refresh at $t_J$ splits the $t_J$ rounds seen
so far by parity; when $t_J$ is odd the two halves differ in size by one, and
Lemma~\ref{thm:est-full} holds with $|\mathcal K_a|$ in place of $K_0/2$, which changes
$\vartheta_i$ by a factor at most $\sqrt{t_J/(t_J-1)}$ and leaves $\E\Delta_J$ of the
same order. Conditionally on $\mathcal F_{t_J}$, apply
Lemma~\ref{lem:sgd} to segment $J$ with $\mu_F=\mu$, $L_F=L$, $V=\bar V$,
$R=\bar R$, horizon $T_J$ and $a_0=\norm{x^{t_J}-x^\star_{w^{(J)}}}^2$. Its logarithm
is $\ln(e+\mu^2\bar R^2T_J/\bar V)\le\bar\ell$, so
\[
\E\big[\norm{x^T-x^\star_{w^{(J)}}}^2\,\big|\,\mathcal F_{t_J}\big]
\le a_0\Big(e^{-\mu T_J/(2L)}+\frac{\bar V}{\mu^2\bar R^2T_J}\Big)+\frac{\bar\ell\bar V}{\mu^2T_J}.
\]
The sharp form of Lemma~\ref{lem:transfer}(ii) holds pointwise, and
Lemmas~\ref{lem:proxy} and~\ref{thm:weightsub}, with $K=T_J$ at the last refresh, give
$\norm{\bias(w^{(J)})}^2\le(1+\kappa)(\wh\Psi(bT_J)+2\Delta_J)$. Combining as in
Appendix~\ref{app:pf-endtoend} and taking expectations over $\mathcal F_{t_J}$ gives
the display, with $\E a_0=\wt R^2$ and $\E\Delta_J$ in place of $a_0$ and $\Delta_J$.

\emph{Step 4: the bound on $\wt R$.} Put $\rho_j:=(\E\norm{x^{t_j}-x^\star_{w^{(j)}}}^2)^{1/2}$
and $s_j:=t_{j+1}-t_j$, so that $\wt R=\rho_J$. By Remark~\ref{rem:earlyrefresh} every
target lies within $G$ of $x_0^\star$, so $\rho_0\le\norm{x^0-x_0^\star}+G$, and by
Minkowski's inequality in $L^2$,
$\rho_{j+1}\le(\E\norm{x^{t_{j+1}}-x^\star_{w^{(j)}}}^2)^{1/2}+2G$. Two bounds on the
first term are available for segment $j$.
\begin{itemize}[itemsep=1pt,topsep=2pt,leftmargin=5mm]
\item \emph{Any segment.} The recursion $a_t\le(1-\mu\eta)a_{t-1}+\eta^2\bar V$ of
Lemma~\ref{lem:sgd} keeps $a_t\le\max(a_0,\eta\bar V/\mu)\le a_0+\bar V/(2L\mu)$, since
$\eta\le1/(2L)$; so $\rho_{j+1}\le\rho_j+\sqrt{\bar V/(2L\mu)}+2G\le\rho_j+c_\star$.
\item \emph{Long segments.} If $s_j\ge T_\star$ then
$e^{-\mu s_j/(2L)}\le\tfrac1{64}$ and $\bar V/(\mu^2\bar R^2s_j)\le\tfrac18$, so the
tuned bound of Lemma~\ref{lem:sgd} gives
$\E\norm{x^{t_{j+1}}-x^\star_{w^{(j)}}}^2\le\tfrac14\rho_j^2+\ell(s_j)\bar V/(\mu^2s_j)$
with $\ell(s):=1+\ln(e+\mu^2\bar R^2s/\bar V)$, hence $\rho_{j+1}\le\rho_j/2+c_\star$.
\end{itemize}
For $j\ge1$ the segment runs between consecutive grid points, so
$s_j\ge\gamma^{j'}(\gamma-1)-1$ with $j'$ its grid index, which is at least $T_\star$
once $\gamma^{j'}\ge2T_\star/(\gamma-1)$, using $T_\star\ge1$. So all but the first
$N_\star$ segments, segment $0$ included, are long. Over the short ones
$\rho\le\rho_0+N_\star c_\star$. Over the long ones, if $\rho_j\le\Lambda$ with
$\Lambda\ge2c_\star$ then $\rho_{j+1}\le\Lambda/2+c_\star\le\Lambda$. With
$\Lambda:=\rho_0+(N_\star+2)c_\star$ both regimes stay below $\Lambda$, which is the
bound of Remark~\ref{rem:earlyrefresh}.
\end{proof}

\subsection{Proof of Proposition~\ref{prop:orth}}
\label{app:pf-orth}

\begin{proof}
\emph{(i)} With $w^\star_i=\Psi_0/\zeta_i^2$ (so $\sum_iw^\star_i=1$),
independence and $\E\beta_i=0$ give, for each $j$,
\[
\E\inner{\bias(w^\star)}{\wt\beta_j}
=\E\Big\langle\sum_iw_i^\star\beta_i,\ \beta_j-\frac1n\sum_l\beta_l\Big\rangle
=w_j^\star\zeta_j^2-\frac1n\sum_iw^\star_i\zeta_i^2
=\Psi_0-\frac1n\cdot n\Psi_0=0 .
\]

\emph{(ii)} For $v\in\R^n$ put $w_i:=\tfrac1n-v_i+\bar v$, where $\bar v$ is the mean
of $v$. Since only the centred part of $v$ acts on the disagreements,
\[
\bar\beta-\sum_iv_i\wt\beta_i=\sum_iw_i\beta_i=\bias(w),
\qquad \one^\top w=1 .
\]
Conversely, given $w$ with $\one^\top w=1$, the choice $v_i=\tfrac1n-w_i$ has
$\bar v=0$ and reproduces $w$. The correspondence $v\leftrightarrow w$ is therefore
onto $\{\one^\top w=1\}$, and by Proposition~\ref{prop:oracle}
the minimal value is $\min_{\one^\top w=1}\sum_iw_i^2\zeta_i^2=\Psi_0$, attained at
$w=w^\star$.

\emph{(iii)} Write $\bar\beta=(\bar\beta-\bias(w^\star))+\bias(w^\star)$. The first
summand is $\sum_iv^\star_i\wt\beta_i$ with $v^\star_i=\tfrac1n-w^\star_i$ by (ii), a
function of $\wt\beta$. For the second, take coordinates $k,l\in[d]$ and a client $j$.
The coordinates of the $\beta_i$ are independent across $k$, so
$\E[\bias(w^\star)_k\wt\beta_{j,l}]=0$ for $k\ne l$; for $k=l$ the covariance
structure $\tfrac{\zeta_i^2}dI_d$ gives
$\E[\bias(w^\star)_k\wt\beta_{j,k}]=\tfrac1d\E\inner{\bias(w^\star)}{\wt\beta_j}=0$ by
(i). So $\bias(w^\star)$ is uncorrelated with every coordinate of $\wt\beta$, and since
$(\bias(w^\star),\wt\beta)$ is jointly Gaussian it is independent of $\wt\beta$.
Hence $\E[\bar\beta\mid\wt\beta]=\bar\beta-\bias(w^\star)$, and the conditional law of
$\bar\beta$ given $\wt\beta$ is Gaussian with covariance that of $\bias(w^\star)$,
namely $\tfrac1d\sum_i(w^\star_i)^2\zeta_i^2I_d=\tfrac{\Psi_0}dI_d$, constant in
$\wt\beta$.
\end{proof}
\subsection{Proof of Theorem~\ref{thm:lower}}
\label{app:pf-lower}

The proof runs in five steps. Step~1 replaces the minimax risk by a Bayes risk
with a diffuse prior on the centre, and shows that the transcript is a function of a
fixed array that does not depend on the query points. Step~2 reduces that array to a
per-client average. Step~3 computes the posterior variance of the centre in the
resulting Gaussian location model. Step~4 converts that variance into excess risk and
lets the prior become flat. Step~5 repeats the computation with the exact
disagreements handed to the algorithm, and shows the bound survives.

\begin{proof}
Fix $\rho>0$ and put the prior $x_c\sim\mathcal N(0,\tfrac{\rho^2}{d}I_d)$,
independent of $\beta$ and of all noises. Since the minimax risk over $x_c\in\R^d$
dominates the Bayes risk for this prior, it suffices to lower bound the latter.

\emph{Step 1: reduction to a location model.} The algorithm interacts with client
$i$ only through at most $K$ gradient queries; the $m$-th query at a point $x$ returns
$\nabla f_i(x)+\eps_i^m=\mu(x-x_c)+\beta_i+\eps_i^m$, with the $\eps_i^m$ independent
$\mathcal N(0,\tfrac{\sigma_i^2}dI_d)$. The algorithm chooses $x$ and
knows $\mu$, so the transcript is a measurable function of the private seed $U$ and
of the array
\[
\xi_i^m:=-\mu x_c+\beta_i+\eps_i^m,\qquad i\in[n],\ m\in[K],
\]
and conversely (by induction over rounds, the algorithm's state after each report
is a measurable function of $U$ and the $\xi$'s seen so far, and adding the known
$\mu x$ recovers the report). Note that the law of $\xi_i^m$ does not depend on the
query point: adaptivity provides no information. Hence
$\sigma(\text{transcript})\subseteq\sigma(U,\{\xi^m_i\})$, and it suffices to bound
the Bayes risk of estimating $x_c$ from $(U,\{\xi^m_i\})$.

\emph{Step 2: sufficiency.} Conditionally on $(x_c,\beta_i)$ the vectors
$\xi^1_i,\dots,\xi^K_i$ are i.i.d.\ Gaussian with mean $-\mu x_c+\beta_i$ and
covariance $\tfrac{\sigma_i^2}{d}I_d$; hence
$\bar\xi_i=\tfrac1K\sum_m\xi_i^m$ is sufficient, and
\[
\bar\xi_i=-t+\theta_i,\qquad t:=\mu x_c,\qquad
\theta_i:=\beta_i+\bar\eps_i\sim\mathcal N\Big(0,\frac{v_i}{d}I_d\Big),\qquad
v_i:=\zeta_i^2+\frac{\sigma_i^2}{K},
\]
with $\theta_1,\dots,\theta_n$ independent of each other and of $t$, and
$t\sim\mathcal N(0,\tfrac{\mu^2\rho^2}{d}I_d)$.

\emph{Step 3: Gaussian Bayes risk.} All quantities decouple across the $d$
coordinates. For a single coordinate, $t$ has prior variance
$P:=\mu^2\rho^2/d$ and we observe $n$ independent measurements $-t+\theta_i$ with
noise variances $v_i/d$; the posterior of $t$ is Gaussian with variance
$\big(P^{-1}+\sum_i d/v_i\big)^{-1}$. The Bayes risk for squared error is the
posterior variance, and independence of $U$ from everything else cannot reduce it.
Summing over $d$ coordinates,
\[
\E\norm{\wh t-t}^2\ \ge\ d\Big(\frac{d}{\mu^2\rho^2}+\sum_i\frac{d}{v_i}\Big)^{-1}
=\Big(\frac{1}{\mu^2\rho^2}+\sum_i\frac1{v_i}\Big)^{-1}.
\]

\emph{Step 4: back to the objective.} $f_0(\wh x)-\min f_0=\tfrac\mu2\norm{\wh x-x_c}^2
=\tfrac1{2\mu}\norm{\mu\wh x-t}^2$, and $\mu\wh x$ is an admissible estimator of
$t$. Hence the Bayes risk is at least
$\tfrac1{2\mu}\big(\tfrac{1}{\mu^2\rho^2}+\sum_iv_i^{-1}\big)^{-1}$; letting
$\rho\to\infty$ gives $\Psi(K)/(2\mu)$.

\emph{Step 5: the algorithm may even be told $\wt\beta$.} Condition on
$\wt\beta$. By Proposition~\ref{prop:orth}(iii), $\bar\beta=m(\wt\beta)+e$ with
$m(\wt\beta):=\bar\beta-\bias(w^\star)$ known once $\wt\beta$ is, and
$e:=\bias(w^\star)\sim\mathcal N(0,\tfrac{\Psi_0}dI_d)$ independent of $\wt\beta$, of
$t$ and of the noise. Subtracting the known $\wt\beta_i+m(\wt\beta)$ from $\bar\xi_i$
leaves $z_i=-t+e+\bar\eps_i$, $i\in[n]$, with $\bar\eps_i\sim\mathcal N(0,\tfrac{\sigma_i^2}{Kd}I_d)$
independent. The $z_i$ share the mean $e-t$, so their precision-weighted average
$y$ is sufficient for $e-t$, and $y=e-t+\bar\eps$ with
$\bar\eps\sim\mathcal N(0,\tfrac{s^2}dI_d)$, $s^2:=(\sum_iK/\sigma_i^2)^{-1}$ (if some
$\sigma_i=0$ then $s^2=0$ and the argument is unchanged). Per coordinate, $t$ has prior
variance $P$ and $y$ is $t$ observed through independent Gaussian noise of variance
$(\Psi_0+s^2)/d$, so $\Var(t\mid y)=P(\Psi_0+s^2)/(dP+\Psi_0+s^2)\to(\Psi_0+s^2)/d$ as
$P\to\infty$, and the Bayes risk is at least $(\Psi_0+s^2)/(2\mu)$ in the limit, as in
Step~4. By Lemma~\ref{lem:harmonic} applied with
$a_i=\zeta_i^2$ and $b_i=\sigma_i^2/K$,
\[
\Psi_0+s^2\ =\ H(a)+H(b)\ \le\ H(a+b)\ =\ \Psi(K),
\]
with equality exactly when $\zeta^2$ and $\sigma^2$ are proportional. Handing the
algorithm the disagreements therefore lowers the floor, by the harmonic defect
$\Psi(K)-\Psi_0-s^2$, and never below $\Psi_0/(2\mu)$.
\end{proof}

\begin{remark}
The instance used has $L=\mu$, $\Xset=\R^d$ and Gaussian everything, so the
obstruction is neither conditioning, nor constraint geometry, nor heavy tails. It
is the diffuse nuisance $x_c$, i.e.\ exactly the unidentifiability of
Proposition~\ref{prop:ident-neg}(i). Step 5 separates the part of the risk the
disagreement can remove from the part it cannot. Exact knowledge of $\wt\beta$ lowers the floor from
$\Psi(K)/(2\mu)$ to $(\Psi_0+s^2)/(2\mu)$. The drop is zero when the bias and
noise profiles are proportional and positive otherwise. The consensus error $\Psi_0$
survives in every case, since it is the one component the disagreement does not
determine. Proposition~\ref{prop:orth} is the companion statement for weighting rules,
where no such side information is available and the floor is $\Psi_0$ itself.
\end{remark}
\subsection{Proofs of Theorem~\ref{thm:closure} and Corollary~\ref{cor:closure}}
\label{app:pf-closure}

Throughout, $\mathfrak G:=\sigma(x_c,\beta_1,\dots,\beta_n)$ and
$\mathcal F_k:=\mathfrak G\vee\sigma(U,\text{noises through round }k)$, $U$ the
algorithm's private randomness. The oracle noises are independent of
$(x_c,\beta)$, so Assumption~\ref{as:oracle} holds conditionally on $\mathfrak G$
with $\bar\sigma_i=\sigma_i$, the tilts are $\mathcal F_0$-measurable, and
Lemmas~\ref{lem:residual}, \ref{thm:est-full} and~\ref{lem:sgd} apply conditionally.
Write $q_i=\norm{\beta_i}^2$, $\bar q=\AM(q)$, and $\mathrm{exc}:=f_0(x^{K_0+T})-f_0^\star
=\tfrac\mu2\norm{x^{K_0+T}-x_c}^2$.

\begin{proof}[Proof of Theorem~\ref{thm:closure}]
Define two events. $G$ is the event of Lemma~\ref{lem:kappa} at confidence $\delta$:
$\max_i|q_i/\zeta_i^2-1|\le\epsilon_0$ and $\kappa\le\epsilon_2$, with
$\PP(G^c)\le\delta$; $G\in\sigma(\beta)$. And
$E_c:=\{\norm{x_c-x^0}\le2\rho\}$; since $d\norm{x_c-x^0}^2/\rho^2\sim\chi^2_d$, the
Laurent--Massart bound gives $\PP(E_c^c)\le e^{-d/4}$, and $E_c\in\sigma(x_c)$ is
independent of $G$.

\emph{Step 1: off the good event.} The output lies in $\Xset=B(x^0,D)$, so on every
outcome $\norm{x^{K_0+T}-x_c}\le D+\norm{x_c-x^0}$ and
$\mathrm{exc}\le\mu(D^2+\norm{x_c-x^0}^2)$. Independence of $G$ and $x_c$ gives
$\E[\one_{G^c}\,\mathrm{exc}]\le\mu\delta(D^2+\rho^2)$. By Cauchy--Schwarz and
$\E\norm{x_c-x^0}^4=\rho^4(1+2/d)\le3\rho^4$,
\[
\E[\one_{E_c^c}\,\mathrm{exc}]
\le\mu\big(D^2e^{-d/4}+\sqrt3\,\rho^2e^{-d/8}\big)
\le\sqrt3\,\mu\,(D^2+\rho^2)\,e^{-d/8}.
\]
Since $\one_{(G\cap E_c)^c}\le\one_{G^c}+\one_{E_c^c}$,
\begin{equation}
\label{eq:closure-tail}
\E\big[\one_{(G\cap E_c)^c}\,\mathrm{exc}\big]\ \le\ 2\mu\,(D^2+\rho^2)\big(\delta+e^{-d/8}\big).
\end{equation}

\emph{Step 2: geometry on the good event.} On $G$,
$\max_i\norm{\beta_i}\le\sqrt{(1+\epsilon_0)}\,\zeta_{\max}\le\sqrt2\,\zeta_{\max}$,
so $\norm{\bias(w)}\le\sqrt2\,\zeta_{\max}$ for every $w\in\simplex$. The tilts are
linear and $f_0$ is the $\mu$-quadratic, so $F_{\wh w}$ is $\mu$-strongly convex and
$\mu$-smooth with unconstrained minimiser $x^\star_{\wh w}=x_c-\bias(\wh w)/\mu$; on
$G\cap E_c$, $\norm{x^\star_{\wh w}-x^0}\le2\rho+\sqrt2\,\zeta_{\max}/\mu=D$, so
$x^\star_{\wh w}\in\Xset$ and the minimiser of $F_{\wh w}$ over $\Xset$ is
$x^\star_{\wh w}$.

\emph{Step 3: conditional bound.} $\one_{G\cap E_c}$ is
$\mathcal F_{K_0}$-measurable. Conditionally on $\mathcal F_{K_0}$, Phase~II is
projected \textsc{SGD} on $F_{\wh w}$ over $\Xset$ with unbiased gradients of
conditional variance at most $\sum_i\wh w_i^2\sigma_i^2/b\le\bar V$. Because
\CROWD$^\star$ projects onto $\Xset$ in every round, probing included,
$x^{K_0}\in\Xset$; on $G\cap E_c$ also $x^\star_{\wh w}\in\Xset$ by Step~2, and
$\Xset$ has diameter $2D$. Lemma~\ref{lem:sgd} with $\mu_F=L_F=\mu$ and
$a_0=\norm{x^{K_0}-x^\star_{\wh w}}^2\le(2D)^2=R^2$ therefore gives
\[
\E_{K_0}\norm{x^{K_0+T}-x^\star_{\wh w}}^2\ \le\ 4D^2e^{-T/2}+\frac{\ell_T\bar V}{\mu^2T}.
\]
Lemma~\ref{lem:transfer}(ii) is a pointwise inequality on $\R^d$ for a linear crowd,
with $\min F_{\wh w}$ the unconstrained minimum. On $G\cap E_c$ that minimum is
attained at $x^\star_{\wh w}\in\Xset$, where $\nabla F_{\wh w}(x^\star_{\wh w})=0$, so
$\mu$-smoothness gives
$F_{\wh w}(x)-F_{\wh w}(x^\star_{\wh w})\le\tfrac\mu2\norm{x-x^\star_{\wh w}}^2$ at the
output $x=x^{K_0+T}\in\Xset$. (Off $G\cap E_c$ the unconstrained minimiser may leave
$\Xset$, and the display below is neither claimed nor used there.) For every
$\epsilon_T>0$, on $G\cap E_c$,
\begin{equation}
\label{eq:closure-cond}
\E_{K_0}[\mathrm{exc}]\ \le\
\frac{1+\epsilon_T}{2\mu}\,\Gamma_{\wh w}^2
+\big(1+\epsilon_T^{-1}\big)\frac\mu2\Big(4D^2e^{-T/2}+\frac{\ell_T\bar V}{\mu^2T}\Big).
\end{equation}

\emph{Step 4: the bias term in expectation.} On $G$, Assumption~\ref{as:incoh} holds
with $\kappa\le\epsilon_2<1$, so Lemma~\ref{lem:proxy} and
Lemma~\ref{thm:weightsub} with the floored costs and $c_i=q_i+\sigma_i^2/(bT)$ give
$\Gamma^2_{\wh w}\le(1+\kappa)\big(\wh\Psi(bT)+2(\Delta+\varrho)\big)$ with $\Delta$
as in Lemma~\ref{thm:est-full}(iii). Three bounds hold on $G$. First,
$q_i+\sigma_i^2/(bT)\le(1+\epsilon_0)(\zeta_i^2+\sigma_i^2/(bT))$, so monotonicity and
homogeneity of $H$ give $\wh\Psi(bT)\le(1+\epsilon_0)\Psi(bT)$. Second,
$\Delta\le C_n(M_s+3\kappa\bar q)$ with $M_s:=\max_i|\wh s_i-\norm{\wt\beta_i}^2|$,
and $3\kappa\bar q\,\one_G\le3\epsilon_2(1+\epsilon_0)\AM$. Third, the estimation
term is bounded without conditioning on $G$, since $\one_{G\cap E_c}M_s\le M_s$.
Lemma~\ref{thm:est-full}(i) applies conditionally on $\mathfrak G$: the crowd is
linear ($\delta_i=0$), $K_0$ is even, all $n$ clients report in every probing round,
and projected query points are adaptive query points, which the lemma allows. It gives
$\E[|\wh s_i-\norm{\wt\beta_i}^2|\mid\mathfrak G]
\le(4\wt\zeta_i^2\vartheta_i^2+2\vartheta_i^4)^{1/2}
\le2\wt\zeta_i\vartheta_i+\sqrt2\,\vartheta_i^2$ with $\vartheta_i$ deterministic.
Taking the unconditional expectation and using
$\E\wt\zeta_i\le(\E\wt\zeta_i^2)^{1/2}=\sqrt{s_i}$, where the equality is the
unconditional identity of Theorem~\ref{thm:ident},
$\E[\one_{G\cap E_c}M_s]\le\E M_s\le\sum_i\E|\wh s_i-\norm{\wt\beta_i}^2|\le a(K_0)$.
(Summing over $i$ bounds the maximum crudely; a union bound would replace the
factor $n$ by about $\sqrt n$.) Combining, and using $1+\kappa\le1+\epsilon_2$ and
$1\le1+\epsilon_0$ where convenient,
\[
\E\big[\one_{G\cap E_c}(1+\kappa)\big(\wh\Psi(bT)+2\Delta+2\varrho\big)\big]
\ \le\ (1+\epsilon_2)(1+\epsilon_0)\Big[\Psi(bT)+2C_na(K_0)+6C_n\epsilon_2\AM+2\varrho\Big].
\]

\emph{Step 5: assembly.} Multiply~\eqref{eq:closure-cond} by $\one_{G\cap E_c}$, take
full expectations, insert Step~4, and add~\eqref{eq:closure-tail}.
\end{proof}

\begin{proof}[Proof of Corollary~\ref{cor:closure}]
The ratio is at least one at every parameter value by~\eqref{eq:bayes-floor}. For the
upper limit, note first that $\Psi(K)$ is non-increasing in $K$ with limit $\Psi_0$,
so $\mathfrak B^\star_\rho(K)\ge\mathfrak B^\star_\rho(\infty)
:=\big((\mu^2\rho^2)^{-1}+\Psi_0^{-1}\big)^{-1}/(2\mu)>0$ for every $K$; dividing the
bound of Theorem~\ref{thm:closure} by $\mathfrak B^\star_\rho(\infty)$ therefore bounds
the ratio at every $T$ and $K_0$. Hold $\rho$, $d$, $K_0$ fixed and let $T\to\infty$
with $\varrho=1/T$ and $\epsilon_T=T^{-1/4}$: the transient of
Theorem~\ref{thm:closure} vanishes ($\ell_T=O(\log T)$), $\Psi(bT)\to\Psi_0$ and
$\varrho\to0$, so
\[
\limsup_{T\to\infty}\frac{\mathfrak B_\rho(\CROWD^\star)}{\mathfrak B^\star_\rho(b(K_0+T))}
\ \le\
\frac{(1+\epsilon_2)(1+\epsilon_0)\big[\Psi_0+2C_na(K_0)+6C_n\epsilon_2\AM\big]
+4\mu^2(D^2+\rho^2)(\delta+e^{-d/8})}
{\big((\mu^2\rho^2)^{-1}+\Psi_0^{-1}\big)^{-1}} .
\]
Now let $K_0\to\infty$: $a(K_0)\to0$ since every $\vartheta_i\to0$. Then let
$d\to\infty$ with $\delta=e^{-\sqrt d}$: $\log(4n/\delta)=\log4n+\sqrt d$, so the
$\epsilon_0$ and $\epsilon_1$ of Lemma~\ref{lem:kappa} tend to zero and $\epsilon_2\to0$.
The hypotheses of Lemma~\ref{lem:kappa}, $\epsilon_0\le\tfrac12$ and $\epsilon_1\le1$,
hold once $d\ge C'(n+\log4n+\sqrt d)$, which is true for all large $d$, and then
$\epsilon_2<1$ as Theorem~\ref{thm:closure} requires; $D$ does not depend on $d$, so
the tail term vanishes. The bound tends to
$\Psi_0\big/\big((\mu^2\rho^2)^{-1}+\Psi_0^{-1}\big)^{-1}=1+\Psi_0/(\mu^2\rho^2)$,
which is the displayed limit at fixed $\rho$ and tends to one as $\rho\to\infty$.
For the $(1+\epsilon)$ form with $\epsilon\in(0,1]$: choose $\rho_0$ with
$\Psi_0/(\mu^2\rho_0^2)\le\epsilon/4$; for each $\rho\ge\rho_0$ choose $d_0$ and
$K_0'$ so that the $T\to\infty$ limit of the bound is at most $(1+\epsilon/4)^2$;
then choose $T_0$ so that the bound itself is at most $(1+\epsilon/4)^3$ for
$T\ge T_0$, which is possible because it converges and the denominator was replaced by
its $T$-independent lower bound $\mathfrak B^\star_\rho(\infty)$. Finally
$(1+\epsilon/4)^3\le1+\tfrac34\epsilon+\tfrac3{16}\epsilon^2+\tfrac1{64}\epsilon^3\le1+\epsilon$
for $\epsilon\le1$.
\end{proof}

\subsection{Proof of Theorem~\ref{thm:nonlin}}
\label{app:pf-nonlin}

The proof runs in three steps. Step~1 bounds the staleness of the residual by the
diameter of the domain times the Lipschitz constant of the centred gradient, which is
what Lemma~\ref{thm:est} needs. Step~2 runs the optimisation phase with the moduli of
the weighted objective. Step~3 transfers to the truth with the constrained transfer
lemma, keeping the two smoothness constants apart.

\begin{proof}
\emph{Step 1: staleness.} On $\Xset$, $\norm{\nabla^2h_i}_{\mathrm{op}}\le\ell_i$, so
$\wt{\nabla h_i}=\nabla h_i-\frac1n\sum_j\nabla h_j$ is $\wt\ell_i$-Lipschitz with
$\wt\ell_i=\ell_i+\frac1n\sum_j\ell_j$. Every query point lies in $\Xset$, as does
$x_0^\star$, so for every $k\le K_0$ and every $i$,
\[
\norm{\wt{\nabla h_i}(x^{k-1})-\wt\beta_i}\le\wt\ell_i\norm{x^{k-1}-x_0^\star}\le\wt\ell_i\bar D=:\delta_i
\]
surely. Lemma~\ref{thm:est}(ii)--(iii) therefore applies with these $\delta_i$, and its
high-probability form gives, with probability at least $1-\delta$,
$\Delta\le C_n\big(\max_i[2\wt\zeta_i(\delta_i+\vartheta_i\sqrt{2n/\delta})+(\delta_i+\vartheta_i\sqrt{2n/\delta})^2]+3\kappa\bar q\big)$.

\emph{Step 2: optimisation.} Conditionally on $\mathcal F_{K_0}$ the weights $\wh w$ are
fixed, $\nabla^2F_{\wh w}=\nabla^2f_0+\nabla^2H_{\wh w}$ with
$\norm{\nabla^2H_{\wh w}}_{\mathrm{op}}\le\lambda$, so $F_{\wh w}$ is
$(\mu-\lambda)$-strongly convex and $L'$-smooth on $\Xset$, and the aggregate has
conditional variance at most $V_\sigma$. Lemma~\ref{lem:sgd} with
$\mu_F=\mu-\lambda$, $L_F=L'$, $R=\bar D$ and $a_0\le\bar D^2$ (both points lie in
$\Xset$) gives
\[
\E_{K_0}\norm{x^{K_0+T}-x^\star_{\wh w}}^2\le\bar D^2e^{-(\mu-\lambda)T/(2L')}+\frac{\ell_TV_\sigma}{(\mu-\lambda)^2T},
\]
where $x^\star_{\wh w}$ minimises $F_{\wh w}$ over $\Xset$.

\emph{Step 3: transfer.} Lemma~\ref{lem:transfer-c} applies because $x_0^\star$ is
interior and $\lambda<\mu$; it gives, pointwise,
$f_0(x)-f_0^\star\le\tfrac{2L}{\mu-\lambda}[F_{\wh w}(x)-F_{\wh w}(x^\star_{\wh w})]+\tfrac{L}{(\mu-\lambda)^2}\norm{\bias(\wh w)}^2$.
Two smoothness constants occur: $L$ is that of the truth, used by the lemma, while
$L'$ is that of $F_{\wh w}$, which converts iterate accuracy into function accuracy.
By Lemma~\ref{lem:transfer-c}, Lemma~\ref{lem:proxy} and $V(\wh w)\le\max_iq_i$,
$\norm{x^\star_{\wh w}-x_0^\star}\le\Gamma_{\wh w}/(\mu-\lambda)\le\Gamma_{\max}/(\mu-\lambda)$,
so $x^\star_{\wh w}$ lies in the interior of $\Xset$ by hypothesis and
$\nabla F_{\wh w}(x^\star_{\wh w})=0$. $L'$-smoothness then gives
$F_{\wh w}(x)-F_{\wh w}(x^\star_{\wh w})\le\tfrac{L'}2\norm{x-x^\star_{\wh w}}^2$, and
\[
\E_{K_0}[f_0-f_0^\star]\le\frac{2L}{\mu-\lambda}\cdot\frac{L'}{2}\,\E_{K_0}\norm{x-x^\star_{\wh w}}^2
+\frac{L}{(\mu-\lambda)^2}\norm{\bias(\wh w)}^2 ,
\]
and $\norm{\bias(\wh w)}^2\le(1+\kappa)V(\wh w)\le(1+\kappa)(\wh\Psi(bT)+2\Delta)$ by
Lemmas~\ref{lem:proxy} and~\ref{thm:weightsub}. Collecting terms gives the display.
\end{proof}
\subsection{Proof of Proposition~\ref{thm:free}}
\label{app:pf-free}

\begin{proof}
Every ingredient is deterministic. Lemma~\ref{thm:weightsub} applied with $c=q$ and
the estimate $\wh c$, $\norm{\wh c-q}_\infty\le\Delta$, gives
\begin{equation}
\label{eq:free-core}
V(\wh w)=\sum_i\wh w_i^2q_i\ \le\ \min_{w\in\simplex}\sum_iw_i^2q_i+2\Delta .
\end{equation}
Let $w^{\mathrm{opt}}\in\argmin_{w\in\simplex}\norm{\bias(w)}^2=\argmin_{w\in\simplex}w^\top Qw$,
which exists by compactness. Then
$\min_{w\in\simplex}\sum_iw_i^2q_i\le V(w^{\mathrm{opt}})\le\frac{1}{1-\kappa}\norm{\bias(w^{\mathrm{opt}})}^2$
by the lower half of Lemma~\ref{lem:proxy}, which needs $\kappa<1$. Combining with the
upper half, $\norm{\bias(\wh w)}^2\le(1+\kappa)V(\wh w)$, and~\eqref{eq:free-core},
\[
\norm{\bias(\wh w)}^2\le(1+\kappa)\Big(\frac{\norm{\bias(w^{\mathrm{opt}})}^2}{1-\kappa}+2\Delta\Big)
=\frac{1+\kappa}{1-\kappa}\min_{w\in\simplex}\norm{\bias(w)}^2+2(1+\kappa)\Delta,
\]
which is~\eqref{eq:free-1}. For the gain over uniform weighting,
Proposition~\ref{prop:oracle} applied to the realised vector $q$ gives
$\min_w\sum_iw_i^2q_i=(\sum_iq_i^{-1})^{-1}=\HM(q)/n$, so
$\norm{\bias(\wh w)}^2\le(1+\kappa)(\HM(q)/n+2\Delta)$, while
$\norm{\bias(w_{\mathrm{unif}})}^2\ge(1-\kappa)V(w_{\mathrm{unif}})=(1-\kappa)\AM(q)/n$;
dividing gives the claim.

\emph{The size of $\Delta$.} \CROWD{} uses $\wh c_i=(\wh q_i)_++\bar\sigma_i^2/(bT)$, so
$\Delta\le\max_i|\wh q_i-q_i|+\max_i\bar\sigma_i^2/(bT)$. The first term is bounded by
Lemma~\ref{thm:est}(iii), which assumes Assumption~\ref{as:oracle} and nothing about
the tilts beyond Assumption~\ref{as:incoh}; the second vanishes with the horizon.
\end{proof}

\subsection{Proof of Proposition~\ref{prop:check}}
\label{app:pf-check}

\begin{proof}
\emph{(i)} This is Proposition~\ref{prop:ident-neg}(i). The optimal weights of the two
instances are computed from the energies $\norm{\beta_i}^2$ and $\norm{\beta_i-c}^2$,
which differ for generic $c$.

\emph{(ii), non-identifiability.} Under the shift of (i) with linear tilts the tilts of
the second instance at its own minimiser are $\beta_i-c$, so its Gram matrix is
$Q_{ij}-\inner{\beta_i}{c}-\inner{\beta_j}{c}+\norm c^2$, as stated. The transcript
law is the same, and $\kappa$ changes in general, so $\kappa$ is not a function of
that law. The disagreements $\wt\beta_i=\beta_i-\bar\beta$ are invariant, and so is
$JQJ=(\inner{\wt\beta_i}{\wt\beta_j})_{ij}$.

\emph{(ii), unbiasedness.} For linear tilts
$\wh\beta^{(a)}_i=\wt\beta_i+\nu^{(a)}_i$ in the notation of
Appendix~\ref{app:pf-est}. Expanding,
$\E[\inner{\wh\beta^{(1)}_i}{\wh\beta^{(2)}_j}\mid B]=\inner{\wt\beta_i}{\wt\beta_j}+\E[\inner{\wt\beta_i}{\nu^{(2)}_j}+\inner{\nu^{(1)}_i}{\wt\beta_j}\mid B]+\E[\inner{\nu^{(1)}_i}{\nu^{(2)}_j}\mid B]$.
The middle terms vanish because each is a sum of martingale increments against an
$\mathcal F_0$-measurable vector, and the last vanishes by
Lemma~\ref{lem:crossmart} applied to the martingale difference sequences
$(\tfrac2{K_0}\wt\eps^k_i)_k$ and $(\tfrac2{K_0}\wt\eps^k_j)_k$ over the disjoint halves.
Symmetrising gives $\E[\wh G\mid B]=JQJ$.

\emph{(ii), the correction.} Write $\varsigma:=(Q-D)\one$, so $Q\one=q+\varsigma$ and
$\one^\top Q\one=\one^\top q+\one^\top\varsigma$. Expanding $JQJ$,
\[
JQJ=Q-\tfrac1n\big((q+\varsigma)\one^\top+\one(q+\varsigma)^\top\big)+\tfrac{\one^\top q+\one^\top\varsigma}{n^2}\one\one^\top .
\]
The population version of $\wh Q$ adds back
$\tfrac1n(q\one^\top+\one q^\top)-\tfrac{\one^\top q}{n^2}\one\one^\top$, so
$\wh Q_\circ=Q+E$ with $E:=-\tfrac1n(\varsigma\one^\top+\one\varsigma^\top)+\tfrac{\one^\top\varsigma}{n^2}\one\one^\top$;
under Assumption~\ref{as:crowd} each $\varsigma_i$ has conditional mean zero given
$\beta_i$, which is the sense in which the correction sets the row sums to their
mean. With $\wh D=D$, $\wh\kappa_\circ=\norm{D^{-1/2}(Q-D+E)D^{-1/2}}_{\mathrm{op}}$, so
$|\wh\kappa_\circ-\kappa|\le\norm{D^{-1/2}ED^{-1/2}}_{\mathrm{op}}$. Put
$u:=D^{-1/2}\one$, with $\norm u^2=\sum_iq_i^{-1}=n/\HM(q)$, and
$r:=D^{-1/2}\varsigma=MD^{1/2}\one$, $M:=D^{-1/2}(Q-D)D^{-1/2}$, so that
$\norm r\le\kappa\norm{D^{1/2}\one}=\kappa\sqrt{n\AM(q)}$ and
$|\one^\top\varsigma|=|(D^{1/2}\one)^\top M(D^{1/2}\one)|\le\kappa n\AM(q)$. Then
$D^{-1/2}ED^{-1/2}=-\tfrac1n(ru^\top+ur^\top)+\tfrac{\one^\top\varsigma}{n^2}uu^\top$ and
\[
\norm{D^{-1/2}ED^{-1/2}}_{\mathrm{op}}\le\frac{2\norm r\norm u}{n}+\frac{|\one^\top\varsigma|\norm u^2}{n^2}
\le2\kappa\sqrt{\AM(q)/\HM(q)}+\kappa\,\AM(q)/\HM(q). \qedhere
\]
\end{proof}
\subsection{Proof of Proposition~\ref{prop:partial}}
\label{app:pf-partial}

\begin{proof}
\emph{(i)} Fix $k$ and $i$ and condition on $\mathcal F_{k-1}$ and on $i\in P_k$. By
Assumption~\ref{as:part} the other $m-1$ members of $P_k$ form a uniform
$(m-1)$-subset $S$ of $[n]\setminus\{i\}$, independent of the past and of the round's
noise. For linear tilts,
$r^k_i=\beta_i-\tfrac1m\sum_{j\in P_k}\beta_j+\wt\eps^{P,k}_i$ with
$\wt\eps^{P,k}_i:=\eps^k_i-\tfrac1m\sum_{j\in P_k}\eps^k_j$, whose conditional mean is
zero. Since $\E\sum_{j\in S}\beta_j=\tfrac{m-1}{n-1}\sum_{j\ne i}\beta_j=\tfrac{m-1}{n-1}(n\bar\beta-\beta_i)$,
\[
\E\Big[\beta_i-\tfrac1m\sum_{j\in P_k}\beta_j\Big]
=\Big(1-\tfrac1m+\tfrac{m-1}{m(n-1)}\Big)\beta_i-\tfrac{(m-1)n}{m(n-1)}\bar\beta
=\rho(\beta_i-\bar\beta),
\]
because $1-\tfrac1m+\tfrac{m-1}{m(n-1)}=\tfrac{m-1}{m}\cdot\tfrac{n}{n-1}=\rho$. The
bounds on $\rho$ follow from $\rho=\tfrac{m-1}m\cdot\tfrac n{n-1}$, which is increasing
in $m$, equals $n/(2(n-1))>\tfrac12$ at $m=2$ and $1$ at $m=n$.

\emph{(ii)} Condition on the tilts and on the indicators $(\one\{i\in P_k\})_k$. The
$P_k$ are independent across rounds, independent of the noise, and given
$i\in P_k$ the set $P_k\setminus\{i\}$ is a uniform $(m-1)$-subset of
$[n]\setminus\{i\}$. So, enlarging $\mathcal F_0$ by the indicators, which are
independent of everything else, the facts used in (i) still hold at every round with
$i\in P_k$. On such a round decompose
$r^k_i=\rho\wt\beta_i+\chi^k_i+\wt\eps^{P,k}_i$ with
$\chi^k_i:=\beta_i-\tfrac1m\sum_{j\in P_k}\beta_j-\rho\wt\beta_i$. By (i),
$\E[\chi^k_i\mid\mathcal F_{k-1}]=0$, and $\wt\eps^{P,k}_i$ has conditional mean zero
given $P_k$; so $X_k:=\chi^k_i+\wt\eps^{P,k}_i$ is a martingale difference sequence
over the rounds with $i\in P_k$, and the two terms are orthogonal. For the first, the
variance of the sum of a simple random sample of size $m-1$ from the $n-1$ vectors
$(\beta_j)_{j\ne i}$ is
$\tfrac{(m-1)(n-m)}{(n-1)(n-2)}\sum_{j\ne i}\norm{\beta_j-\bar\beta_{-i}}^2\le\tfrac{(m-1)(n-m)}{(n-1)(n-2)}\,n\bar q$
\citep[Thm.~2.2]{cochran1977}, so
$\E[\norm{\chi^k_i}^2\mid\mathcal F_{k-1}]\le\tfrac{(m-1)(n-m)n}{m^2(n-1)(n-2)}\bar q\le\psi_{m,n}\bar q$.
For the second, conditional independence of the noise across clients gives
$\E[\norm{\wt\eps^{P,k}_i}^2\mid\mathcal F_{k-1}]\le\tfrac1b\big((1-\tfrac1m)^2\sigma_i^2+\tfrac{m-1}{m^2(n-1)}\sum_{j\ne i}\sigma_j^2\big)\le\tau_{i,P}^2/b$.
Hence $\wh\beta^{(a)}_i=\rho\wt\beta_i+m^{(a)}_i$ with $m^{(a)}_i$ the average of $N^{(a)}_i$
martingale increments, $\E\norm{m^{(a)}_i}^2\le\vartheta_{i,a}^2$. The
decomposition~\eqref{eq:err-decomp} with $\rho\wt\beta_i$ in place of $\wt\beta_i$ and
Lemma~\ref{lem:crossmart} over the disjoint halves give $\E\wh s_i=\rho^2\norm{\wt\beta_i}^2$,
$\E\inner{\rho\wt\beta_i}{m^{(1)}_i+m^{(2)}_i}^2\le\rho^2\norm{\wt\beta_i}^2(\vartheta_{i,1}^2+\vartheta_{i,2}^2)$
and $\E\inner{m^{(1)}_i}{m^{(2)}_i}^2\le\vartheta_{i,1}^2\vartheta_{i,2}^2$, and
$(u+v)^2\le2u^2+2v^2$ gives the variance bound. Finally
$N^{(a)}_i\sim\mathrm{Bin}(K_0/2,p)$, and the multiplicative Chernoff bound gives
$\PP(N^{(a)}_i<pK_0/4)\le e^{-pK_0/16}$; a union bound over $i$ and $a$ finishes.

\emph{(iii)} As in Lemma~\ref{thm:est}(iii), with $\wh s/\rho^2$ in place of $\wh s$.
\end{proof}

\subsection{Proof of Proposition~\ref{prop:part-corr}}
\label{app:pf-partcorr}

\begin{proof}
\emph{(i)} The transcript is generated by the reporting clients only, and
Proposition~\ref{prop:ident-neg}(i), applied to the instance restricted to
$\mathcal R$, shows that the shift by $c$ leaves its law unchanged. In the shifted
instance the tilts are $\beta_i-c=\beta'_i$, independent and centred, which is
Assumption~\ref{as:crowd} on $\mathcal R$ for linear tilts.
\emph{(ii)} $\sum_iw_i\beta_i=\sum_iw_i(c+\beta'_i)=c+\bias'(w)$ because
$\sum_iw_i=1$.
\emph{(iii)} $\beta_i-\bar\beta=\beta'_i-\bar\beta'$ for every $i\in\mathcal R$, so
the residuals, and every statistic built from them, coincide. In the shifted instance
the tilts are $\beta'_i$, whose energies are $\norm{\beta'_i}^2$, and the analysis of
Section~\ref{sec:theory} applies to it.
\end{proof}
\section{Experimental details and further figures}
\label{app:exp}

This section records how the runs of Section~\ref{sec:exp} were produced and reports
the ones that did not fit there. Every figure comes from a seeded script. The
supplementary materials contain the code, the recorded output of every run holding
each quoted number, a README with per-experiment instructions and runtimes, and a
document mapping the quoted numbers to the statements they test. The synthetic
instances match the model by construction: their purpose is to check constants,
exponents and floors, not to argue about realism. The studies on real scans
begin in Section~\ref{app:regime}.

\paragraph{Default instance.} $n=20$, $d=150$,
$f_0(x)=\tfrac12(x-x_c)^\top A(x-x_c)$ with $A=Q\Lambda Q^\top$, $Q$ random
orthogonal and $\Lambda$ geometric in $[\mu,L]=[0.5,5]$;
$x_c\sim\mathcal N(0,4I_d/d)$; $\log\zeta_i\sim\mathrm U[\log0.2,\log3]$,
$\beta_i\sim\mathcal N(0,\zeta_i^2I_d/d)$; $\log\sigma_i\sim\mathrm U[\log0.5,\log3]$,
Gaussian gradient noise of second moment $\sigma_i^2/b$, $b=1$.

\paragraph{Rules compared.} \emph{Uniform}; \emph{size-proportional}
($w_i\propto n_i$, $n_i\sim\mathrm U\{20,\dots,199\}$ drawn independently of $\beta$,
so size and quality are unrelated); \emph{inverse-noise} ($w_i\propto\sigma_i^{-2}$);
\emph{best single client} ($w=\delta_{i^\star}$, $i^\star=\argmin_i\norm{\beta_i}^2$,
chosen with oracle knowledge); \emph{oracle} ($w_i\propto\norm{\beta_i}^{-2}$); and
\CROWD{} in its anytime form, refreshing at every round. All take the same stochastic-gradient steps and differ
only in $w$.

\paragraph{Figure-specific settings.}
\begin{itemize}[itemsep=2pt,topsep=2pt]
\item \textbf{Figure~\ref{fig:main}}: 12 seeds, $K=4000$, $\eta=1/(2L)$;
reports the exact realised bias energy $\norm{B^\top w^k}^2$.
\item \textbf{Figure~\ref{fig:exp-a}}: 12 seeds, $K=20000$,
$\eta_k=\min\{1/(2L),2/(\mu(k+2L/\mu))\}$; the decaying step makes the sampling term
vanish so the bias floors become visible. Dotted lines are
$\tfrac12\bias(w)^\top A^{-1}\bias(w)$, which for a quadratic $f_0$ is the exact
excess risk of $\argmin F_w$ -- a strengthening of Lemma~\ref{lem:transfer}(iii)
holding with equality.
\item \textbf{Figure~\ref{fig:est}}: 12 seeds; panel (a) reports
$\max_i|\wh s_i-\norm{\wt\beta_i}^2|$ at $d=60$, panel (b)
$\max_i|\wh q_i-\norm{\beta_i}^2|$ for $d\in\{60,240,960\}$. The query point is fixed
and random, which is legitimate because Lemma~\ref{lem:residual} makes the residual
independent of it.
\item \textbf{Figure~\ref{fig:exp-d}}: 24 seeds, $K=3000$, $A=I$, $\sigma_i\equiv1$,
$\log\zeta_i\sim\mathrm U[-s/2,s/2]$ with spread $s\in\{0,0.5,\dots,3\}$.
\item \textbf{Figure~\ref{fig:misspec}}: 6 seeds, $K=3000$, $n=20$. Panel (a):
$d=100$, $h_i(x)=\inner{\beta_i}{x}-\tfrac{s_i\lambda}{\omega^2}\sum_j\cos(\omega z_j)$
with $z=x-x_0^\star$, $\omega=3$ and random signs $s_i=\pm1$, so
$\nabla h_i(x_0^\star)=\beta_i$ exactly and
$\norm{\nabla^2h_i}_{\mathrm{op}}\le\lambda$; $\lambda/\mu$ is swept over
$\{0,0.05,0.1,0.2,0.4,0.6,0.8\}$. Panel (b): $d=400$,
$\beta_i=\zeta_i(\sqrt{1-\rho}\,\xi_i+\sqrt\rho\,\chi_i)$ where $\chi_i$ is a shared
factor within a pair of clients (\emph{pairwise}) or common to all
(\emph{common factor}); the optimal $w$ for the realised $Q$ is computed by projected
gradient on the simplex.
\end{itemize}

\paragraph{Reported statistics.} Trajectory plots show medians and interquartile
bands over seeds. Log--log slopes are least-squares fits over the plotted grid.
Figure~\ref{fig:est}(a) is fitted over the last nine grid points, to avoid the
small-$K$ regime where the quadratic remainder of Lemma~\ref{thm:est}(i) dominates.

\begin{figure}[t]
\centering
\includegraphics[width=0.5\textwidth]{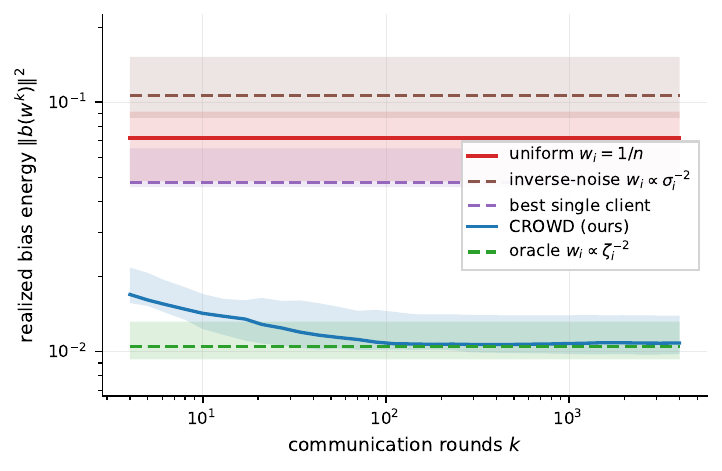}
\caption{Realised bias energy $\norm{\bias(w^k)}^2$ against the round index on the
default instance ($n=20$, $d=150$, quadratic $f_0$ with spectrum geometric in
$[0.5,5]$, $\log\zeta_i$ and $\log\sigma_i$ uniform with $\zeta_i\in[0.2,3]$ and
$\sigma_i\in[0.5,3]$, $b=1$; medians and interquartile bands over 12 seeds), all
rules taking identical stochastic gradients and differing only in $w$. \CROWD{}
reaches the oracle level within $3.4\%$ in a few tens of rounds; the
oracle-selected best single client is $4.6\times$ worse, as
Proposition~\ref{prop:oracle}(ii) implies. The companion excess-risk trajectories,
with the exact floors $\tfrac12\bias(w)^\top A^{-1}\bias(w)$, are
Figure~\ref{fig:exp-a}.}
\label{fig:main}
\end{figure}

\begin{figure}[t]
\centering
\includegraphics[width=0.8\textwidth]{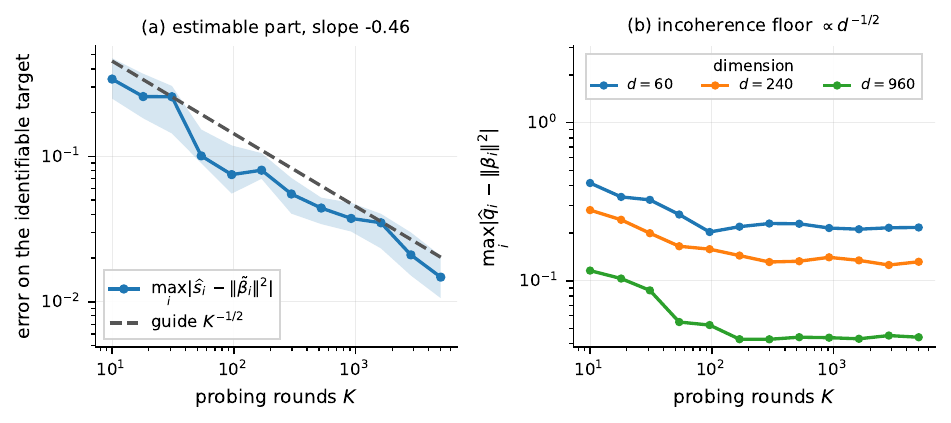}
\caption{Estimator accuracy on the default instance (12 seeds), error against the
number of probing rounds $K_0$ on log--log axes. (a) On the identifiable target
$\norm{\wt\beta_i}^2$
the error decays at the predicted rate: measured log--log slope $-0.46$ against the
$-1/2$ of Lemma~\ref{thm:est}(i). (b) On $q_i$ the error stalls at the incoherence
floor of Lemma~\ref{lem:realised}, which does not decay with $K$ at all; the floor
falls by $4.9\times$ when $d$ grows by $16\times$, against the $4.0\times$ that the
$d^{-1/2}$ of Lemma~\ref{lem:gauss-eta} predicts. The two individual $4\times$ steps
give $1.6\times$ and $3.0\times$ against $2.0\times$ each, so the floor is of the
predicted size and direction but not cleanly power-law at these dimensions.}
\label{fig:est}
\end{figure}

\begin{figure}[t]
\centering
\includegraphics[width=0.68\textwidth]{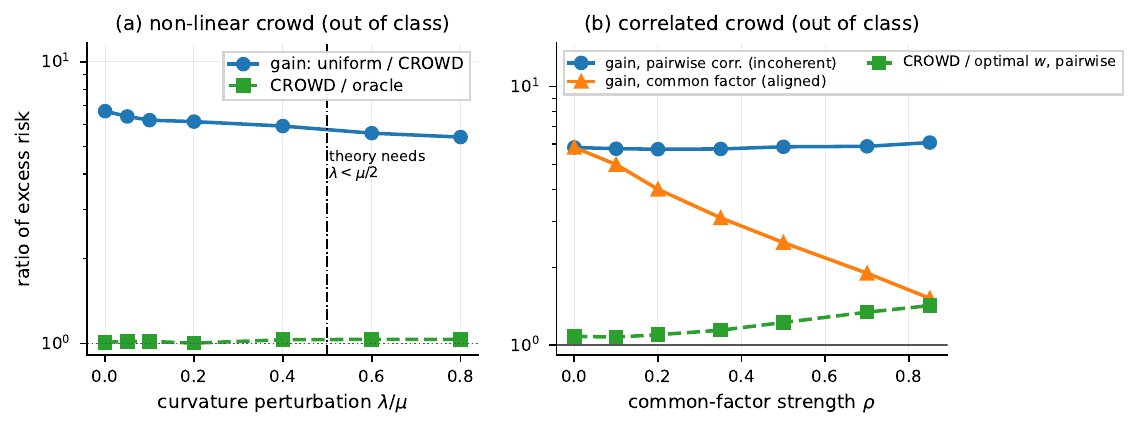}
\caption{\CROWD{} applied \emph{outside} its class: $n=20$, $K=3000$, 6 seeds, all
rules taking identical gradients. \emph{(a)} Non-linear
crowds $\nabla h_i(x)=\beta_i\pm\tfrac\lambda\omega\sin(\omega(x-x^\star_0))$: the
gain over uniform weights falls only from $6.7\times$ to $5.4\times$ as $\lambda$
grows to $0.8\mu$, and \CROWD{} stays within $1.00$--$1.03\times$ of the oracle --
including where Theorem~\ref{thm:nonlin} does not apply. \emph{(b)} Correlated
crowds: under pairwise correlation the gain stays between $5.7\times$ and $6.1\times$,
even where $\kappa$ exceeds one, and the competitive ratio stays in $[1.07,1.42]$,
while a common factor -- the invisible kernel of
Proposition~\ref{prop:corr} -- erodes it smoothly.}
\label{fig:misspec}
\end{figure}

\begin{figure}[t]
\centering
\includegraphics[width=0.85\textwidth]{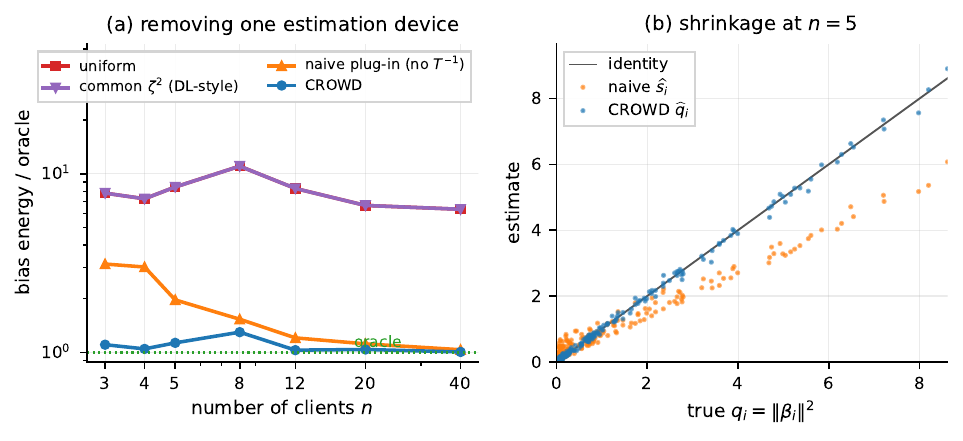}
\caption{The contribution of each of \CROWD's two estimation devices, on the default instance
with the crowd size $n$ swept over $\{3,\dots,40\}$ (12 seeds; medians).
\emph{(a)} Realised bias energy relative to the oracle. Fitting a \emph{single}
heterogeneity variance in the manner of random-effects meta-analysis is
indistinguishable from uniform weighting at every $n$ ($7.80$ versus $7.80$ at
$n=3$): the gain comes from the per-client second moments. Dropping the inversion
$T^{-1}$ and using the raw disagreement energy costs $3.1\times$ at $n=3$ and
$1.04\times$ at $n=40$, decreasing with $n$ as $C_n$ does.
\emph{(b)} Estimated versus true $q_i$ at $n=5$: the naive statistic is shrunk
towards the crowd mean with measured slope $0.634$ against the $1-2/n+1/n^2=0.64$ that
Lemma~\ref{lem:realised} predicts for a fit pooled over seeds, while \CROWD's slope is
$0.995$.}
\label{fig:baselines}
\end{figure}

\paragraph{What the baselines are, under their usual names.} Three of the rules of
Section~\ref{sec:exp} implement published families. \emph{Inverse-noise} is
precision-weighted federated learning \citep{reyes2021,toelle2025}.
\emph{No-inversion} is the scoring step of truth discovery \citep{chen2026td},
which by Lemma~\ref{lem:realised} estimates $(Tq)_i$ rather than $q_i$.
\emph{Common-$\zeta^2$} is the \citet{dersimonian1986} fit. \textsc{SCAFFOLD},
\textsc{FedProx} and their relatives change the optimiser rather than
the weights, and by Lemma~\ref{lem:transfer}(iii) every convergent optimiser stalls at
the floor set by $w$. Figure~\ref{fig:fedopt} checks this with $E=5$ local steps per
round ($8$ seeds). With
uniform weights, \textsc{FedAvg}, \textsc{FedProx} and \textsc{SCAFFOLD} all end at
$8.65$ times the oracle floor, against $8.67$ for the exact floor of uniform weighting.
With \CROWD's weights, computed from the first local gradient of each round,
\textsc{FedAvg} and \textsc{SCAFFOLD} both end at $1.05$. We also run a quadratic crowd whose
client Hessians differ by centred perturbations of norm $0.4\mu$, with a constant
local step so that client drift is present. There the three optimisers end at
$8.63$--$8.66$ with uniform weights, and both end at $1.07$ with \CROWD's weights. The
drift-correcting optimiser and the weights act on different errors, and they compose.

\begin{figure}[t]
\centering
\includegraphics[width=0.95\textwidth]{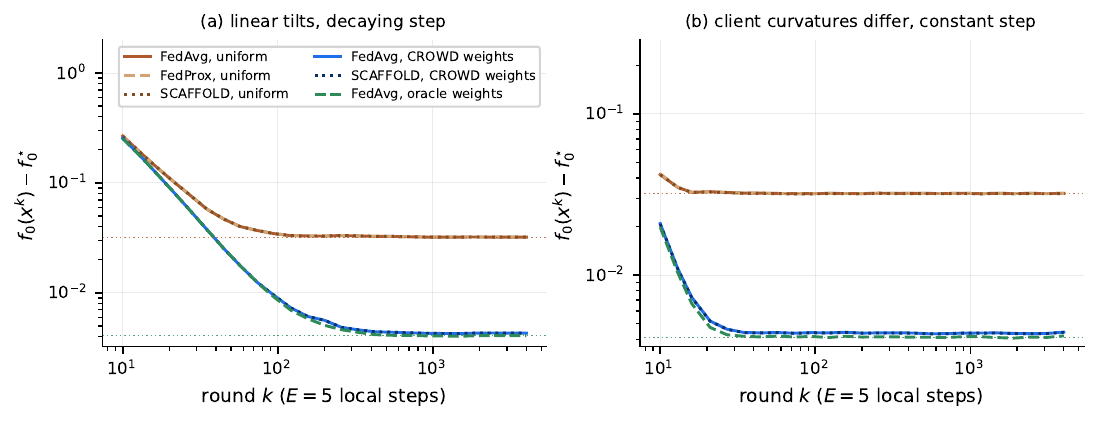}
\caption{Optimisers with local steps against weights ($n=20$, $E=5$, $8$ seeds,
medians). Dotted lines are the exact floors of uniform (upper) and oracle (lower)
weights. (a) Linear tilts, decaying step. (b) Client Hessians differ, constant local
step.}
\label{fig:fedopt}
\end{figure}

\subsection{Refresh schedules}
\label{app:refresh}

The synthetic runs of Section~\ref{sec:exp}
refresh at every round with a decaying stepsize, a schedule the theorem does not cover;
the real-data runs refresh once, after probing, as in Theorem~\ref{thm:endtoend}.
Figure~\ref{fig:refresh} compares the schedules on the default synthetic instance and
on the same instance with five times the noise ($12$ seeds). After $2\cdot10^4$ rounds,
refreshing every round and refreshing on the doubling grid ($14$ refreshes) end at
$1.08$ and $1.08$ times the oracle's exact floor under the default noise, and at $1.21$
and $1.21$ under the larger noise. A single refresh pays the bias of uniform weighting
until it happens, and if it happens early it freezes a large estimation error. Under
the larger noise a refresh at $k=100$ ends at $1.89$ times the floor, with an error in
the energies of $0.31\,\bar q$ against $0.11\,\bar q$ for the doubling grid. A single
refresh at $k=1000$ ends at $1.17$.

\begin{figure}[t]
\centering
\includegraphics[width=0.95\textwidth]{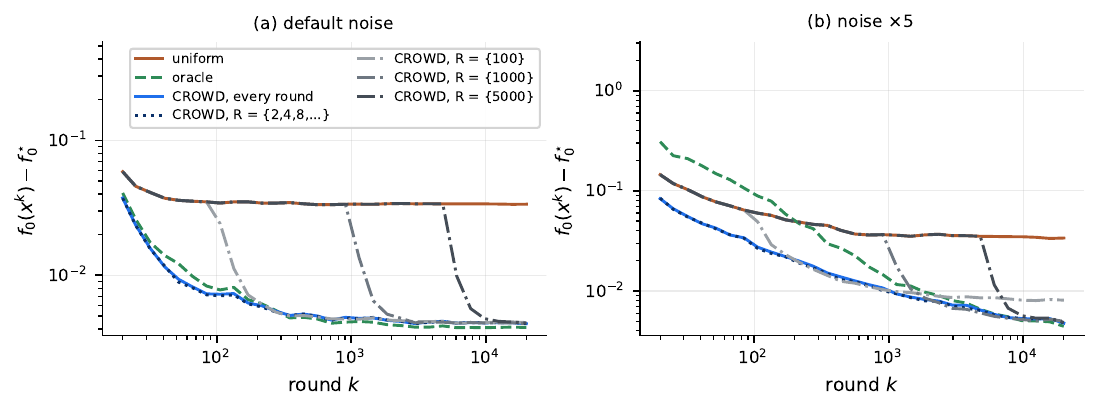}
\caption{Excess risk on the truth for different refresh sets $\mathcal R$, synthetic
instance, medians over $12$ seeds.
All schedules of one seed see the same gradient noise.}
\label{fig:refresh}
\end{figure}

\subsection{Diagnostics: when is a crowd incoherent, and what breaks the estimator}
\label{app:diag}

\begin{figure}[t]
\centering
\includegraphics[width=\textwidth]{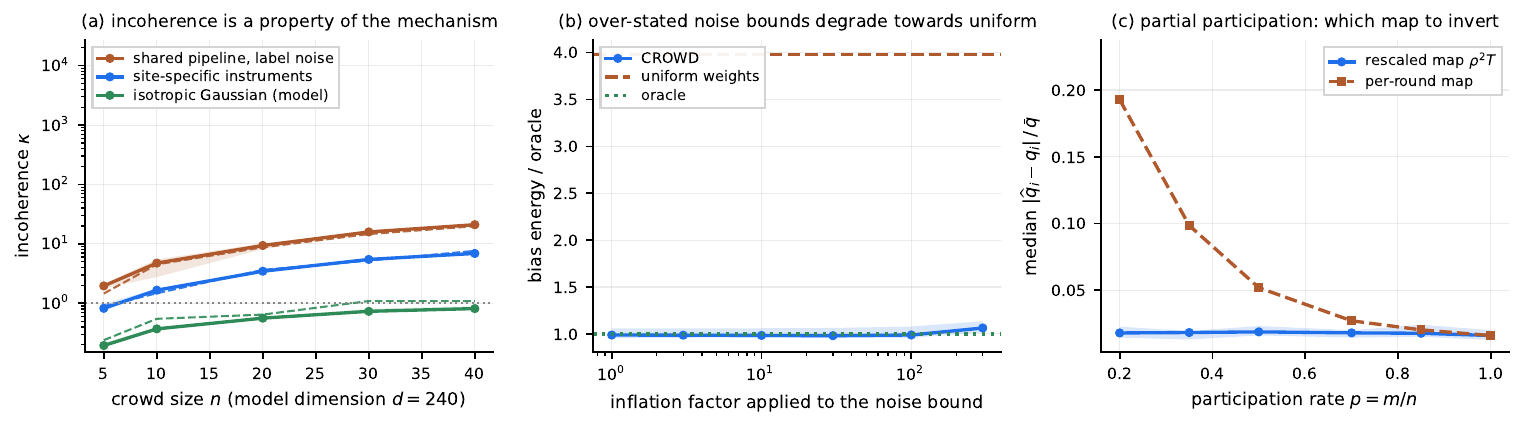}
\caption{Diagnostics.
(a) Incoherence $\kappa$ of the realised tilts under three mechanisms of heterogeneity,
at fixed model dimension $d=240$ ($5$ seeds, medians and interquartile bands; dashed
lines are the estimate $\wh\kappa$ of Proposition~\ref{prop:check}, computed from
residuals alone). (b) \CROWD{} with the noise bound $\bar\sigma_i^2$ inflated by a
factor $\rho$, bias energy relative to the oracle, $12$ seeds. (c) Median error of the estimated energies under
uniform partial participation, inverting the rescaled map of
Proposition~\ref{prop:partial} or the per-round map, $16$ seeds.}
\label{fig:diag}
\end{figure}

\paragraph{(a) Incoherence is a property of the heterogeneity mechanism.}
Assumption~\ref{as:incoh} is the one hypothesis of this paper that concerns the
realised instance rather than the model, so the question is which mechanisms produce
it. We compare three, all with the same
convex model ($\ell_2$-logistic regression on a frozen random feature map, parameter
dimension $d=240$) and the same crowd sizes.
\begin{itemize}[itemsep=1pt,topsep=2pt,leftmargin=6mm]
\item \emph{Shared pipeline, label noise.} Every client runs the same feature map on
exchangeable data and differs only through its label channel. Measured $\kappa$ rises
from $1.96$ at $n=5$ to $21.0$ at $n=40$.
\item \emph{Site-specific instruments.} Each client additionally applies its own
per-feature gain and offset before the shared map, drawn once per client -- the
motivating deployment, different machines at different hospitals. Measured $\kappa$
falls to $0.82$ at $n=5$ and $1.66$ at $n=10$.
\item \emph{Isotropic Gaussian tilts}, the model reference: $0.19$ at $n=5$ and $0.81$
at $n=40$, consistent with Lemma~\ref{lem:kappa}.
\end{itemize}
The pattern has a structural explanation, and it is the main limitation of the
method. Label-channel heterogeneity is \emph{low rank}: with $c$ classes the channel
deviations live in a space of dimension at most $c^2$, and all tilts factor through
one design matrix, so a crowd larger than that dimension cannot be incoherent whatever
$d$ is. Instrument heterogeneity has as many independent causes as there are
measured features, which pushes the threshold out. The condition $\kappa<1$
therefore compares $n$ with the \emph{number of independent causes of heterogeneity},
and Lemma~\ref{lem:kappa} is the special case in
which that number is $d$. A standard benchmark with a single frozen feature map and
synthetic label noise sits at the wrong end of this scale, which is why we do not
present one as supporting evidence.

The dashed curves are the estimate $\wh\kappa$ of Proposition~\ref{prop:check},
which uses no knowledge of the tilts. It tracks $\kappa$ across all three mechanisms
and over two orders of magnitude. For Gaussian tilts it over-states $\kappa$ at the
larger crowds, $1.09$ against $0.73$ at $n=30$, which errs towards refusing. So the
guard of Remark~\ref{rem:guard} is implementable: a deployment can estimate which
regime it is in before deciding to weight at all.

\paragraph{(b) Over-stated noise bounds.} The bound $\bar\sigma_i$ is the one input a
practitioner must supply, and the natural failure is to set it far too large. Inflating
$\bar\sigma_i^2$ by up to $100\times$ leaves the bias energy within $2\%$ of the oracle,
and a $300\times$ inflation costs $6\%$, against $3.98\times$ for uniform weighting on
the same instances. An inflated noise term is common to the
cost of every client up to the spread of $\sigma_i$, so it shrinks the weights towards
uniform, and
Proposition~\ref{thm:free} bounds that loss. Under-stating $\bar\sigma_i$ is the dangerous
direction, and there the guarantee is the one of Lemma~\ref{thm:weightsub} with
$\Delta$ enlarged by the shortfall.

\paragraph{(c) Partial participation.} At each round a uniformly drawn set of $m=pn$
clients reports, the residual is centred within it, and each client's two
half-averages are taken over the rounds in which it participated. Inverting the
rescaled full-crowd map $\rho^2T$ of Proposition~\ref{prop:partial} gives a median
error of $0.016$ at $p=1$ and between $0.018$ and $0.019$ for every $p$ down to $0.2$,
so at this probing budget the loss from partial participation is small. Inverting
instead the map of the per-round squared disagreement, which also contains the
variance of drawing the participants, leaves an error that grows as participation
falls, $0.020$, $0.027$, $0.052$, $0.098$ and $0.193$ at
$p=0.85,0.7,0.5,0.35,0.2$. That error is a bias of the per-round map and does not
decrease as $K_0$ grows.

\subsection{The applicability window, the guard, and a run on real data}
\label{app:regime}

\begin{figure}[t]
\centering
\includegraphics[width=\textwidth]{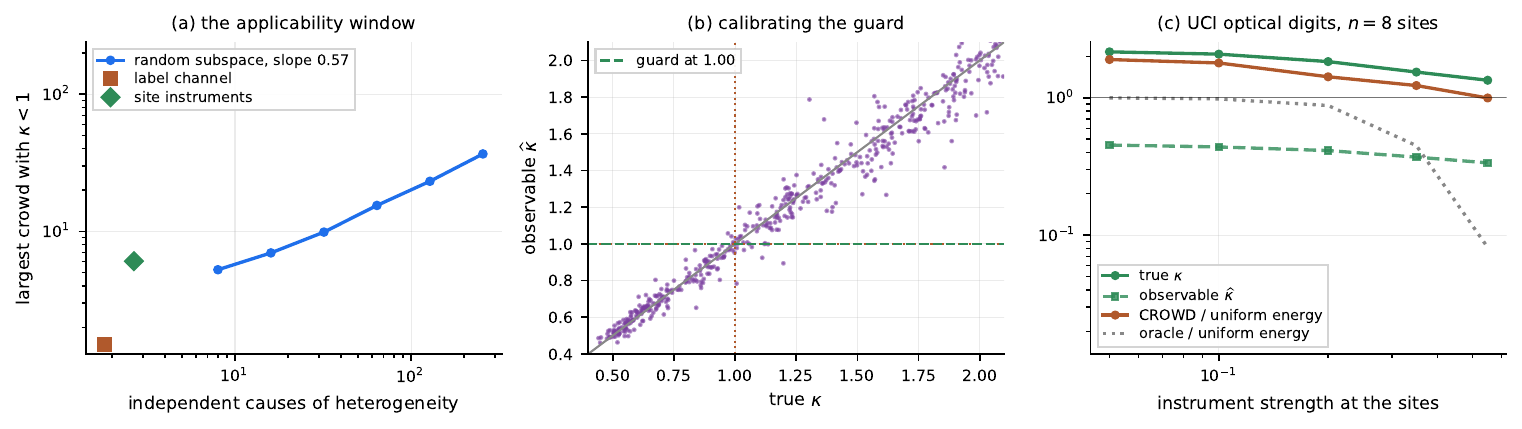}
\caption{The applicability window.
(a) The largest crowd for which the realised $\kappa$ still lies below one, against the
number of independent causes of heterogeneity, for tilts drawn from a random subspace
of that dimension ($7$ seeds); the two markers are the label-channel and site-instrument
mechanisms of Section~\ref{app:diag}, placed at their measured effective rank.
(b) Estimated $\wh\kappa$ against true $\kappa$ on $541$ instances with $n=20$,
$d=240$ and $2500$ probing rounds, in the linear regime.
(c) UCI optical digits with $n=8$ sites and per-site instrument distortion of varying
strength ($4$ seeds): true and estimated incoherence, and the bias energy of \CROWD{}
and of the oracle weights relative to uniform weighting.}
\label{fig:regime}
\end{figure}

\paragraph{(a) How wide is the window.} Section~\ref{app:diag} argued that
Assumption~\ref{as:incoh} is a condition on crowd size against the number of
independent causes of heterogeneity. Panel (a) turns that into a number. For tilts
drawn from a random subspace of dimension $r$ -- exactly $r$ independent causes -- the
largest crowd with $\kappa<1$ grows from $5.3$ at $r=8$ to $36.7$ at $r=256$, with
fitted exponent $0.57$. The two mechanisms sit at the left of that curve: measured
effective rank $1.8$ for label-channel heterogeneity with threshold $n^\star=1.5$, and
$2.7$ for site instruments with $n^\star=6.1$. The window is therefore narrow: with the site
instruments of Section~\ref{app:diag} incoherence already fails at ten clients, and the
whole-array federation of Section~\ref{app:gain} keeps $\kappa<1$ up to sixteen sites
and loses it by twenty-two. A federation in the regime the
theory describes needs heterogeneity with many independent causes, and
identifying such a federation is the practical question this paper leaves open.

\paragraph{(b) How tightly the guard can be set.} The decision $\wh\kappa<1$ is taken
where the estimate matters most, so panel (b) sweeps instances with $\kappa$ near one.
In the linear regime the mean absolute error of $\wh\kappa$ for $\kappa\in[0.8,1.2]$ is
$0.048$, its ninetieth percentile is $0.11$, and the probability that $\kappa\ge1$
given $\wh\kappa<1$ is $4.2\%$. No margin is needed there: the plain threshold already
holds the false-accept rate below five per cent.

\paragraph{(c) Real data, and a failure of the guard.} We run the protocol on UCI
optical digits ($1797$ real scans, ten classes). One frozen random feature map is
shared by all sites, the model is $\ell_2$-logistic regression, and there are eight
sites of unequal size. Each site has its own instrument: its scanner is miscalibrated
on its own subset of pixels, a minority of sites badly. Sweeping the distortion strength gives the following.
\begin{itemize}[itemsep=1pt,topsep=2pt,leftmargin=6mm]
\item \emph{Incoherence fails throughout.} The realised $\kappa$ runs from $2.16$ at the
weakest distortion to $1.34$ at the strongest, never below one. The crowd is small
($n=8$) and the parameter space large, so this is not a dimension effect; the tilts
share the digit structure.
\item \emph{The oracle gain is large.} The ratio $\AM/\HM$ of the realised energies
grows from $1.42$ to $11.85$, and the oracle weights reach $0.08$ times the uniform bias
energy at the strongest setting.
\item \emph{\CROWD{} does not realise it.} \CROWD{} attains $1.90,1.79,1.42,1.23,1.00$
times the uniform bias energy across the sweep -- never better than uniform. This is
what Lemma~\ref{thm:est}(iii) predicts once $\kappa>1$: the floor $3\kappa\bar q$
exceeds the spread of the $q_i$, and the inverted statistic carries no usable signal.
\item \emph{The guard fails here.} The estimate
$\wh\kappa$ reports $0.45$ down to $0.34$ across the sweep, against a true $\kappa$
between $2.16$ and $1.34$. This is an under-estimate by a factor three to five, on the
wrong side of the threshold. The cause is the staleness of Lemma~\ref{thm:est}. On a
strongly non-linear crowd the residual measured at the iterate is not the disagreement
at $x_0^\star$, and the same bias that inflates $\delta_i$ deflates the off-diagonal
Gram entries. Proposition~\ref{prop:check} is proved for the linear crowd, and panel (b)
checks it there. On this non-linear crowd $\wh\kappa$ under-reports by a factor we do
not control, and Remark~\ref{rem:guard} must be read with that restriction.
\item \emph{The standard construction is worse.} With an \textsc{IID} split and label
noise as the only heterogeneity, the same data give $\kappa=2.25$, $\AM/\HM=1.23$, and
\CROWD{} at $2.36$ times the uniform energy: the oracle gain is small and the
estimate is uninformative.
\end{itemize}
The three requirements of the method pull against each other on a small benchmark.
It needs many independent causes, so that $\kappa<1$; a wide spread of site quality,
so that $\AM/\HM$ is large; and perturbations small enough for the first-order model
to hold.
On UCI digits with eight sites we could not satisfy all three at once, and the
configuration above is reported as it stands.

\subsection{Three federations on the same real scans}
\label{app:gain}

\begin{figure}[t]
\centering
\includegraphics[width=\textwidth]{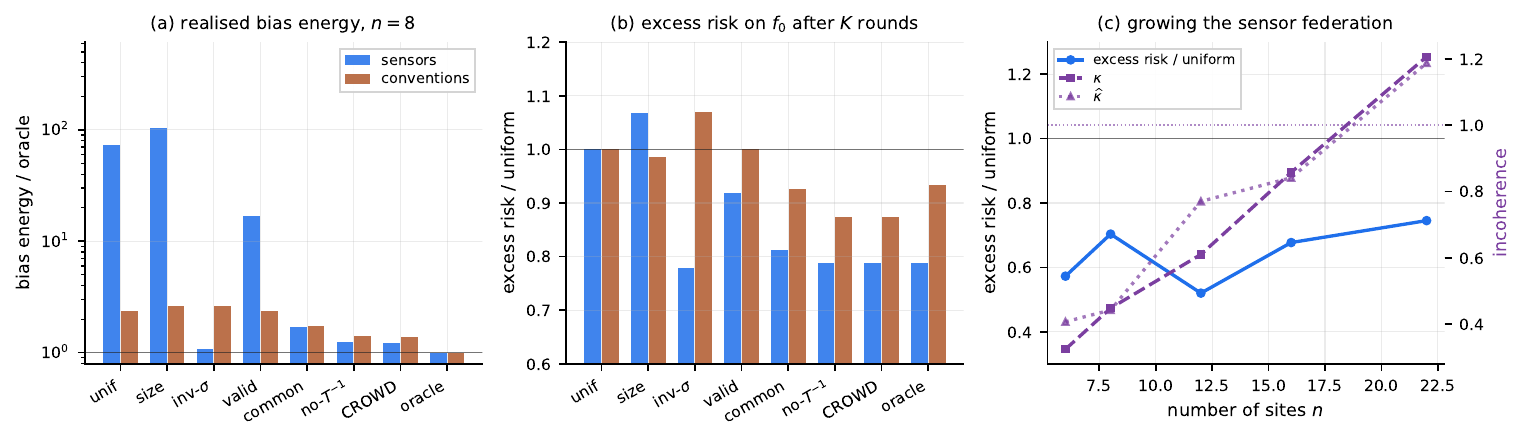}
\caption{UCI optical digits under two constructed federations ($n=8$, $8$ seeds,
medians). (a) Realised bias
energy relative to the oracle weights, log scale. (b) Excess risk on $f_0$ after
$K=2500$ rounds, relative to uniform weighting; below one is a gain. (c) The same ratio
as the sensor federation grows ($4$ seeds per $n$), with the measured and estimated
incoherence on the right axis.}
\label{fig:gainreal}
\end{figure}

\paragraph{What is real and what is constructed.} The scans are real: $1797$ scanned
$8\times8$ images, one frozen random feature map shared by all sites,
$\ell_2$-regularised logistic regression, so the learning problem is the convex one the
theory assumes. The \emph{federation} is not found in nature. We partition one dataset
into sites and impose a distortion on each. The independence of the tilts across
sites, their vanishing crowd mean and the spread of quality are consequences of that
construction, not measurements of an observed federation. So is $f_0$: we define it
as the risk averaged over forty draws from the population law of the distortions.
Assumption~\ref{as:crowd} then holds by construction, up to the Monte Carlo error of
those draws, and excess risk is evaluated exactly on them. The study
therefore tests a conditional claim: that the method behaves as predicted \emph{if}
heterogeneity is arranged this way. Whether federations with $\kappa<1$ occur in
practice is a question a constructed federation cannot settle. On real data \CROWD{}
and the oracle weight by $\wh q_i+\sigma_i^2/n_i$ and $q_i+\sigma_i^2/n_i$, the second
term being the variance of a site's mean gradient over its cohort, in place of the
$\bar\sigma_i^2/(bK)$ of Algorithm~\ref{alg:weights}. Three constructions follow, on the
same scans and the same model.

\begin{table}[h]
\centering
\small
\setlength{\tabcolsep}{3pt}
\begin{tabular}{@{}p{34mm}ccccccc@{}}
\toprule
Mechanism & $\kappa$ & $\wh\kappa$ & $\AM/\HM$ & size & inv-$\sigma$ & valid. & \CROWD{} \\
\midrule
Few pixels $+$ label channel & $1.34$--$2.16$ & $0.34$--$0.45$ & $1.4$--$11.9$ & --- & --- & --- & $\ge1^\dagger$ \\
Whole sensor arrays & $0.48$ & $0.52$ & $86.9$ & $1.069$ & $0.778$ & $0.918$ & $\mathbf{0.787}$ \\
Annotation conventions & $0.48$ & $0.53$ & $2.2$ & $0.985$ & $1.069$ & $1.000$ & $\mathbf{0.875}$ \\
\bottomrule
\end{tabular}
\caption{The three federations built on the same $1797$ real scans. The last four
columns are the \emph{excess} risk on $f_0$ after $K=2500$ rounds, relative to uniform
weighting: below one is a gain, and the oracle weights reach $0.789$ and $0.934$ in the
second and third rows. Row one is the failure of Section~\ref{app:regime}(c); the other
two are described below. $^\dagger$Bias energy relative to uniform
weighting, $1.00$--$1.90$; those runs do not record the risk.}
\label{tab:threefed}
\end{table}

\paragraph{Whole sensor arrays: incoherence holds and the estimator is accurate.}
Each site has its own detector array, miscalibrated across \emph{all} of its pixels
with a per-pixel gain and offset, and sites differ in severity over two decades. This
is the mechanism the motivating application is usually \emph{described} as having --
different machines at different hospitals -- reproduced on real images. The realised
incoherence is $\kappa=0.48$ and the estimate $\wh\kappa=0.52$, so the guard of
Remark~\ref{rem:guard} says proceed, correctly. The crowd-hypothesis ratio
$\norm{\bar\beta}^2/(\bar q/n)$, close to one for a centred independent crowd, is
$1.00$; the Spearman correlation between $\wh q$ and the realised energies is $0.95$,
and $\AM/\HM=86.9$. Uniform weighting sits at $72.9$ times the oracle bias energy and
\CROWD{} at $1.22$.

We report \emph{excess} risk $f_0(x)-f_0^\star$, not total risk: the total contains
$f_0^\star$, which no weighting can change, and ratios of totals understate every
effect. The end-of-run ratios are those of Section~\ref{sec:exp}; at the limit points
of the weighted objectives, computed exactly, \CROWD{} and the oracle reach $0.53$
and $0.54$ of the uniform excess. This federation cannot show the separation of
contribution~(1): a badly miscalibrated sensor inflates its site's gradient variance
as well as its bias, so inverse-noise weighting lands level with \CROWD{} ($0.52$ at
the limit). The mechanism confounds the two costs.

\paragraph{Annotation conventions: bias and noise pulled apart.} To separate them
we need a source of bias that does not scramble the covariates, and cohort sizes that do
not follow it. Every site now carries a detector array of the \emph{same} moderate
severity, which supplies the independent causes and keeps the gradient variances
comparable. Sites differ instead through their annotation convention: a site-specific
confusion channel applied to the labels, with strength spread over two decades. The
cohort sizes are then assigned \emph{against} quality, so the most badly conventioned
site receives the largest cohort. Such a site is quiet and biased, and every rule that
reads precision must prefer it. Precision weighting is then \emph{worse than not weighting
at all} and \CROWD{} beats the oracle (Section~\ref{sec:exp}); size-proportional
weighting reaches $0.985$, the no-inversion variant $0.874$ and the common-$\zeta^2$
fit $0.926$. At the limit points the ordering is the same: \CROWD{} $0.86$, the
oracle $0.92$, inverse-noise $1.01$. The
oracle weights minimise the realised bias energy, a first-order proxy for the excess
risk, and on this instance the proxy is loose (Section~\ref{app:region}). The
federation lies inside the region $\kappa<1$, with $\kappa=0.48$ and
$\wh\kappa=0.53$, so the guard says proceed. The result is consistent with
contribution~(1): the rule that reads precision does worse than uniform weighting, and
\CROWD{} does better.

\paragraph{Accuracy.} Held-out accuracy, measured on clean labels, moves with the risk
in both federations. It rises from $0.8469$ under uniform weighting to $0.8570$ under
\CROWD{} in the sensor federation, and from $0.5632$ to $0.5884$ in the convention one,
where inverse-noise weighting falls to $0.5231$. The theorems bound the risk under $f_0$, not
the accuracy, so we treat accuracy as corroborating evidence only.

\paragraph{The guard on these federations.} Both give $\wh\kappa$ on the correct side
of the threshold, $0.52$ against $\kappa=0.48$ on the sensor federation and $0.53$
against $0.48$ on the convention one. The drift check of Remark~\ref{rem:staleguard}
raises no alarm on either probe.

\paragraph{Validation-tuned weights.} The only validation data a federation has are
those of its sites. In that baseline each site scores the
consensus model on its own held-out split, weights are taken proportional to a power of
the inverse score, and the exponent is chosen by the pooled held-out loss. In the
convention federation it selects the uniform exponent and returns exactly uniform
weighting, $1.000$; in the sensor federation it reaches $0.918$, against $0.787$ for
\CROWD{}. This follows from Proposition~\ref{prop:ident-neg}(i): held-out data drawn from a
site's own tilted distribution measure how well the model fits \emph{that site}, and no
sample from $f_0$ exists to measure anything else. Validation tuning therefore detects
sites that the model fits poorly, which need not be the sites with large tilts.

\paragraph{Beyond $\kappa<1$.} Section~\ref{app:region} sweeps the common detector
severity and finds \CROWD{} below uniform weighting on both sides of $\kappa=1$, at
$0.851$ when $\kappa=1.28$ as well as at $0.832$ when $\kappa=0.37$.
Incoherence is what makes Proposition~\ref{thm:free} non-vacuous; the estimator itself
needs less, and the floor $3\kappa\bar q$ of Lemma~\ref{thm:est}(iii) is worst case
over sign patterns. Closing the gap between the condition the theory uses and the condition
the estimator needs is left open.

\paragraph{How far the crowd can grow.} Panel (c) of Figure~\ref{fig:gainreal} grows the
sensor federation on the
same data. The excess risk relative to uniform weighting is
$0.57$, $0.70$, $0.52$, $0.68$ and $0.75$ at $n=6,8,12,16,22$, while the measured
$\kappa$ rises through $0.32$, $0.45$, $0.61$, $0.86$ and $1.21$. The gain persists
over the whole range, including $n=22$ where $\kappa>1$, and with four seeds per point
the series is too noisy to locate a threshold. The evidence that the window closes is
the incoherence measurement itself, here and in Section~\ref{app:regime}(a). The
intended scale is a federation of a few to a dozen sites with physically distinct
instruments; device-scale federated learning lies outside it.

\subsection{The gain over a region of settings}
\label{app:region}

\begin{figure}[t]
\centering
\includegraphics[width=\textwidth]{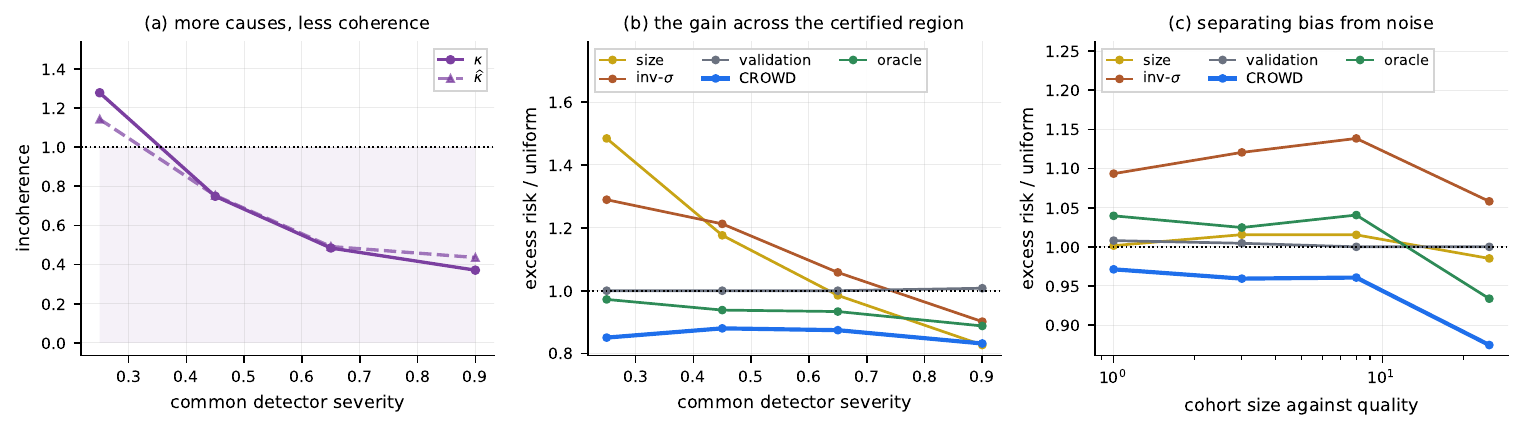}
\caption{Sweeping the two knobs of the convention federation ($n=8$, $6$ seeds,
medians). (a) The common
detector severity controls incoherence, the shaded band marking the region
$\kappa<1$. (b) Excess risk relative to uniform weighting along that axis. (c) The same
against the cohort-size contrast, which controls how far bias and sampling noise are
pulled apart. Below one is a gain.}
\label{fig:region}
\end{figure}

The federation of Section~\ref{app:gain} is built rather than found, so what needs
testing is whether the gain holds over a region of settings. We sweep the two knobs the theory says
should matter, keep everything else fixed, and fix the predictions in advance. The
\emph{common detector severity} is the same at every site, so it creates no quality
spread; what it supplies is independent causes of heterogeneity, which should push
$\kappa$ down. The \emph{cohort-size contrast} is how strongly the cohort sizes run
against quality, which should decide whether a precision rule can see the bias at all.

\paragraph{Along the incoherence axis.} Raising the common severity from $0.25$ to
$0.90$ takes $\kappa$ from $1.28$ to $0.37$ and the estimate $\wh\kappa$ from $1.14$ to
$0.44$, so the estimate tracks the crossing. \CROWD{} reaches
$0.851,0.880,0.875,0.832$ of the uniform excess risk along that axis: a gain at every
point, on both sides of $\kappa=1$, and below the oracle diagonal weights
($0.972,0.938,0.934,0.888$) at every point. Precision-weighted inverse-noise reaches
$1.289,1.213,1.058,0.902$ and size-proportional weighting $1.484,1.176,0.985,0.827$.
Both lose at the smaller severities and gain at the largest, where the detector
distortion, of the same scale at every site, outweighs the conventions, so that large
cohorts are no longer the badly biased ones.

\paragraph{Along the decoupling axis.} With the severity fixed at $0.65$, so that
$\kappa$ stays between $0.46$ and $0.52$ throughout, and the cohort contrast swept from
$1$ to $25$, \CROWD{} gives $0.971,0.959,0.961,0.875$ and inverse-noise
$1.094,1.121,1.139,1.058$. The two sweeps share one setting, so they cover seven
distinct ones. \CROWD{} improves on uniform weighting at all seven and inverse-noise
weighting is worse than uniform at six. The separation of contribution~(1) therefore
holds over a region of the mechanism's parameters.

\paragraph{Validation-tuned weights.} Along both axes validation tuning returns
essentially uniform weights, at $1.000$ to $1.008$ of the uniform excess risk. A site's
own held-out loss measures how well the consensus model fits that site, and by
Proposition~\ref{prop:ident-neg}(i) no held-out sample from $f_0$ exists to measure
anything else.

\paragraph{Where the oracle weights lose.} The oracle diagonal weights do worse than
\CROWD{} at every setting, and worse than uniform weighting at the three smaller
contrasts ($1.025$ to $1.041$). They minimise the realised bias energy, which is a
first-order proxy for the excess risk: it ignores the curvature of $f_0$, the
curvature differences between sites, and the error in the sampling term
$\sigma_i^2/n_i$. On a real instance the error of this proxy can exceed the error of the
estimator.

\subsection{FEMNIST: writers as a natural federation}
\label{app:femnist}

The federations of Sections~\ref{app:gain} and~\ref{app:region} are constructed.
Federated EMNIST supplies a crowd without construction: each client is one real
writer, and nothing is distorted or relabelled.

\paragraph{Setup.} We use the digits of $1000$ randomly chosen writers of the release,
about $100$ training digits each, one frozen random ReLU layer $784\to128$ shared by all
writers, and $\ell_2$-regularised logistic regression with coefficient $10^{-3}$. In
each seed $400$ writers form a pool whose writer-averaged risk is $f_0$, and $n$ other
writers are the clients, each with the empirical risk of its own training digits.
Probing takes $K_0=2000$ rounds with batch $16$ after a burn-in of $600$, and training
takes $K=3000$ rounds under each rule with identical seeds. The excess risk is
evaluated exactly on $f_0$. \CROWD{} and the oracle add the noise term $\sigma_i^2/(bK)$
of Algorithm~\ref{alg:weights}. We run $n\in\{10,20,40,80\}$ with $12$ seeds each.

\paragraph{Why the gain of natural writers fades with $n$.} The first four rows of
Table~\ref{tab:femnist} and Figure~\ref{fig:exp-c} show the
natural federation. With ten writers
$\kappa=0.52$, and \CROWD{} reaches $0.897$ of the uniform excess risk, as does the
oracle; it beats uniform weighting in all twelve seeds. The gain then fades: $0.922$
at $n=20$, $0.991$ at $n=40$, and at $n=80$ every weighting rule, the oracle included,
is about $10\%$ worse than uniform weighting. As $n$ grows, two effects combine. The
excess risk of uniform weighting halves with each doubling of $n$ (the ``uniform
excess'' column), as the
bias energy $\AM/n$ of the uniform crowd does, so the weights act on a shrinking part
of the risk. And $\kappa$ roughly doubles, from $0.52$ to $4.08$, so the estimate of
the energies degrades. The first effect is why even the oracle loses its gain, as the amplified series
below confirms.

\paragraph{Keeping the bias from vanishing.} To separate the two effects we amplify each
writer's style. Every digit is moved by $(\alpha-1)$ times the difference between its
writer's mean image for that class and the pooled class mean. The writers' mean
deviation stays zero, so the crowd remains centred, while every tilt grows with
$\alpha$. The incoherence would be unchanged by a pure rescaling of the tilts; here it
changes little up to $n=40$ ($1.10$ against $1.09$, $1.97$ against $2.07$). With $\alpha=\sqrt{n/10}$ the uniform crowd's excess risk decays much more
slowly, from $0.70$ at $n=10$ to $0.28$ at $n=80$ instead of $0.10$. The gain
persists (Figure~\ref{fig:exp-c}): \CROWD{} reaches $0.922$ at $n=20$ and $0.929$ at
$n=40$, beating uniform weighting in eleven and ten of twelve seeds, although
$\kappa\approx2$ at $n=40$. At $n=80$, $\kappa=3.39$ and the estimate stops working: the oracle still
reaches $0.971$, but \CROWD{} reaches $1.027$ and beats uniform weighting in five of
twelve seeds.

\begin{table}[h]
\centering
\small
\setlength{\tabcolsep}{3.5pt}
\begin{tabular}{@{}lcccccccc@{}}
\toprule
Writers & $\alpha$ & $\kappa$ & $\wh\kappa$ & uniform excess & inv-$\sigma$ & \CROWD{} & oracle & \CROWD{}${}<{}$unif. \\
\midrule
$10$ & $1$ & $0.52$ & $0.50$ & $0.70$ & $0.904$ & $\mathbf{0.897}$ & $0.897$ & $12/12$ \\
$20$ & $1$ & $1.09$ & $0.92$ & $0.38$ & $0.946$ & $\mathbf{0.922}$ & $0.909$ & $10/12$ \\
$40$ & $1$ & $2.07$ & $1.87$ & $0.19$ & $0.973$ & $0.991$ & $0.980$ & $6/12$ \\
$80$ & $1$ & $4.08$ & $3.54$ & $0.10$ & $1.106$ & $1.112$ & $1.102$ & $0/12$ \\
\midrule
$20$ & $1.41$ & $1.10$ & $0.93$ & $0.51$ & $0.925$ & $\mathbf{0.922}$ & $0.891$ & $11/12$ \\
$40$ & $2$ & $1.97$ & $1.75$ & $0.38$ & $0.930$ & $\mathbf{0.929}$ & $0.920$ & $10/12$ \\
$80$ & $2.83$ & $3.39$ & $3.03$ & $0.28$ & $1.010$ & $1.027$ & $0.971$ & $5/12$ \\
\bottomrule
\end{tabular}
\caption{FEMNIST, medians over $12$ seeds. The last four columns are the excess risk on
$f_0$ relative to uniform weighting and the number of seeds in which \CROWD{} beats
uniform weighting; ``uniform excess'' is the absolute excess risk of uniform weighting.
Size-proportional weighting, validation tuning and the common-$\zeta^2$ fit stay
within $0.01$ of uniform weighting in every row.}
\label{tab:femnist}
\end{table}

\paragraph{The trade-off.} The weights help when three conditions hold together. First, the
crowd's summed tilt must vanish slowly as $n$ grows, so that the bias term is a
sizeable part of the risk. Second, the clients must differ in quality, here only
mildly ($\AM/\HM\approx1.7$ to $2.3$, against $86.9$ in the sensor federation). Third,
the crowd must not be much larger than the number of independent causes of
heterogeneity, or the estimate fails. On FEMNIST the first condition limits natural writers to about twenty,
and amplification extends the window to forty, where the third condition takes over.
The theory's sufficient condition $\kappa<1$ is conservative here: \CROWD{} keeps its
gain up to $\kappa\approx2$.

\paragraph{Noise and accuracy.} Inverse-noise weighting tracks \CROWD{} throughout. A
writer whose digits are hard to fit has both a larger gradient variance and a larger
tilt, as with the miscalibrated sensors of Section~\ref{app:gain}, so FEMNIST does not
separate bias from noise; only the convention federation does. At the limit points of
the weighted objectives every ratio changes by less than $0.02$. Accuracy on the pool
writers' test digits moves with the risk, from $0.729$ under uniform weighting to
$0.734$ under \CROWD{} at $n=10$ (oracle $0.741$). The guard is wrong at $n=20$, where
$\wh\kappa\approx0.92$ against $\kappa\approx1.1$, the under-reporting of
Section~\ref{app:regime}(c); in these runs \CROWD{} still gains there.

\subsection{Neural networks outside the theory}
\label{app:mlp}

Every guarantee of the paper assumes a strongly convex truth. To see what happens
without it, we train a two-layer network, $784\to64\to10$ with ReLU and
$\ell_2$-coefficient $10^{-4}$, end to end on the raw FEMNIST pixels. The clients are
the writers of Section~\ref{app:femnist} (batch $16$, step $0.3$, burn-in $600$,
$K_0=2000$ probing rounds, $K=3000$ training rounds). The truth $f_0$ is again the writer-averaged risk of $400$ pool writers. It is
non-convex, so we report $f_0(x)-f_0(x_{\mathrm{ref}})$ relative to uniform weighting,
with $x_{\mathrm{ref}}$ a network trained on the pool by minibatch SGD from the common
initialisation. The oracle weights by the energies of $\nabla f_i-\nabla f_0$ at
$x_{\mathrm{ref}}$, which is a local reference and not a target the theory defines.

The risk ratios, incoherences and win rates are those of Table~\ref{tab:nn} in the
main text ($8$, $8$ and $6$ seeds for the three rows). \CROWD{} lowers the risk in
every setting, by $5$ to $11\%$ in the median, although the incoherence of the tilts
at the reference point exceeds one throughout. Inverse-noise weighting and the oracle
do well with ten writers and poorly with forty amplified ones, where they are $43\%$
and $33\%$ worse than uniform weighting. \CROWD{} measures the disagreement along its
own trajectory rather than at a fixed point, and it is the only rule that improves on
uniform weighting in all three settings. Accuracy on the pool writers' test digits
moves by less than one point in either direction ($0.807\to0.815$, $0.862\to0.859$
and $0.852\to0.846$ under \CROWD{}). We read this as a stress test: the mechanism
carries over to a non-convex model, and the size of the effect is modest.

\paragraph{A convolutional network on the digit scans, and the edge of the window.} To move further
from the theory we replace the frozen feature map of the sensor federation
(Section~\ref{app:gain}) by a network trained end to end on the raw scans. It has a
conv $3{\times}3$ layer with eight channels, ReLU, $2{\times}2$ max-pooling and a
linear layer, with $\ell_2$-coefficient $10^{-3}$, and is trained on each site's
distorted $8\times8$ images ($n=8$, batch $32$, step $0.15$, burn-in $600$,
$K_0=2000$, $K=2500$, $6$ seeds). The step was fixed on a pilot seed outside the
evaluation set. The truth is defined as in
Section~\ref{app:gain}, on $40$ population draws, and the protocol is that of the
two-layer network above, with the reference $\theta_{\mathrm{ref}}$ trained on the
population draws. The energies at the reference spread over three decades
($\AM/\HM\approx2\cdot10^3$), and the realised incoherence is $\kappa=1.63$ in the
median. The cross-fitted statistic is uncorrelated with the realised energies
(Spearman $-0.15$). No diagonal rule improves on uniform weighting
(Table~\ref{tab:nn}). The no-inversion and common-variance rules reach $1.06$, as
\CROWD{} does, the fixed-point oracle is again the worst rule, and accuracy changes
by less than one point. The guard of Remark~\ref{rem:guard} declines in five of six
seeds and would return uniform weighting, which is the better choice on this
instance. On the two-layer network it declines as well, in $20$ of $22$ seeds,
although unguarded \CROWD{} improves on uniform weighting there. On these models the
guard is therefore conservative: it does not separate the case where weighting
helps from the case where it does not.

\end{document}